\documentclass[11pt]{amsart}
\usepackage{graphicx} 
\usepackage{url}
\usepackage{pdfsync}
\usepackage{amssymb}
\usepackage[alphabetic,lite,nobysame]{amsrefs}
\usepackage{amsmath}
\usepackage{amsthm}
\usepackage{fullpage}	
\usepackage{amscd}   	
\usepackage{comment} 	
\usepackage{mathrsfs}      
\usepackage{booktabs}
\usepackage[table,xcdraw]{xcolor}
\usepackage{caption}

\usepackage{booktabs} 

\usepackage[utf8]{inputenc}

\usepackage[T1]{fontenc}
\usepackage{microtype}
\usepackage{color}
\usepackage[shortlabels]{enumitem}

\usepackage{tikz}
\usetikzlibrary{3d, calc}

\usepackage{tikz-cd}
\usetikzlibrary{graphs,decorations.pathmorphing,decorations.markings}

\usepackage[all, cmtip]{xy}
\usepackage{mathtools}
\usepackage[new]{old-arrows}

\usepackage{dutchcal}

\usepackage{tikzpagenodes}
\newcommand{\rtag}[1]{\begin{tikzpicture}[baseline=(tmp.base),remember picture]
\node[inner sep=0pt](tmp){\vphantom{1}};
\begin{scope}[overlay]
 \path (current page text area.east|-tmp.base)
  node[anchor=base east,inner sep=0pt,outer sep=0pt]{#1};
\end{scope}
\end{tikzpicture}}

\usepackage{setspace}
\usepackage[hidelinks]{hyperref}

\newcommand{\uu}{\rule{.25cm}{0.4pt}}

\newcommand{\hideqed}{\renewcommand{\qed}{}} 

\DeclareMathOperator{\pic}{Pic}

\DeclareMathOperator{\SL}{SL}
\DeclareMathOperator{\PSL}{PSL}
\DeclareMathOperator{\SO}{SO}
\DeclareMathOperator{\tO}{O}

\DeclareMathOperator{\id}{id}

\DeclareMathOperator{\br}{Br}
\DeclareMathOperator{\rk}{rk}

\DeclareMathOperator{\spec}{Spec}

\DeclareMathOperator{\spin}{Spin}

\DeclareMathOperator{\Gal}{Gal}

\DeclareMathOperator{\ord}{ord}
\DeclareMathOperator{\Amp}{Amp}
\DeclareMathOperator{\Kthree}{\mathrm{K}3}

\DeclareMathOperator{\kernel}{Ker}
\DeclareMathOperator{\clif}{Cl}
\DeclareMathOperator{\res}{res}
\DeclareMathOperator{\Emb}{Emb}
\DeclareMathOperator{\Diag}{Diag}
\DeclareMathOperator{\Kum}{Kum}
\DeclareMathOperator{\Pic}{Pic}

\renewcommand{\sp}{\mathrm{sp}}

\newcommand{\A}{\mathbb{A}}

\newcommand{\C}{\mathbb{C}}

\newcommand{\F}{\mathbb{F}}
\newcommand{\G}{\mathbb{G}}

\newcommand{\Q}{\mathbb{Q}}
\newcommand{\R}{\mathbb{R}}

\newcommand{\Z}{\mathbb{Z}}

\newcommand{\PP}{\mathbb{P}}

\newcommand{\cA}{\mathcal{A}}
\newcommand{\cB}{\mathcal{B}}
\newcommand{\cC}{\mathcal{C}}
\newcommand{\cD}{\mathcal{D}}

\newcommand{\cF}{\mathcal{F}}

\newcommand{\cH}{\mathcal{H}}
\newcommand{\cI}{\mathcal{I}}

\newcommand{\cM}{\mathcal{M}}

\newcommand{\cO}{\mathcal{O}}
\newcommand{\cP}{\mathcal{P}}
\newcommand{\cQ}{\mathcal{Q}}

\newcommand{\cU}{\mathcal{U}}

\newcommand{\cW}{\mathcal{W}}
\newcommand{\cX}{\mathcal{X}}
\newcommand{\cY}{\mathcal{Y}}
\newcommand{\cZ}{\mathcal{Z}}

\newcommand{\rH}{\mathrm{H}}

\newcommand{\rM}{\mathrm{M}}

\newcommand{\rO}{\mathrm{O}}
\newcommand{\rP}{\mathrm{P}}
\newcommand{\rQ}{\mathrm{Q}}
\newcommand{\rR}{\mathrm{R}}

\newcommand{\iso}{\xrightarrow{\sim}}

\newcommand{\ksep}{k^{\textrm{s}}}
\newcommand{\kbar}{{\overline{k}}}
\newcommand{\Xbar}{{\overline{X}}}

\newcommand{\frakm}{\mathfrak{m}}

\DeclareMathOperator{\disc}{disc}
\DeclareMathOperator{\gspin}{GSpin}

\DeclareMathOperator{\tors}{tors}
\DeclareMathOperator{\K3}{K3}

\renewcommand{\hom}{\text{Hom}}

\newcommand{\et}{\mathrm{\acute{e}t}}

\theoremstyle{plain}
\newtheorem{theorem}{Theorem}[section]
\newtheorem{lemma}[theorem]{Lemma}
\newtheorem{proposition}[theorem]{Proposition}
\newtheorem{corollary}[theorem]{Corollary}
\newtheorem{conjecture}[theorem]{Conjecture}

\theoremstyle{definition}
\newtheorem{definition}[theorem]{Definition}
\newtheorem{remark}[theorem]{Remark}
\newtheorem{example}[theorem]{Example}
\newtheorem{pg}[theorem]{}

\title{Moduli of lattice-polarized K3 surfaces and \\ boundedness of Brauer groups}
\author{Danny Bragg, Emma Brakkee, and Anthony V\'arilly-Alvarado}
\address{Department of Mathematics, University of Utah, 155 South 1400 East, JWB 233, Salt Lake City, UT 84112, USA}
\email{bragg@math.utah.edu}
\address{Leiden University, Mathematical Institute,
Einsteinweg 55, 2333 CC Leiden, The Netherlands}
\email{{e.l.brakkee@math.leidenuniv.nl}}
\address{Department of Mathematics, Rice University MS 136, Houston, TX 77005, USA}
\email{av15@math.rice.edu}

\date{\today}

\begin{document}

\begin{abstract}
Inspired by constructions over the complex numbers of Dolgachev and Alexeev--Engel, we define moduli stacks $\cM_{(L,\cA)/\Z}$ of lattice-polarized K3 surfaces over arbitrary bases, paying particular attention to the open locus $\cP_{(L,\cA)/\Z}$ of primitive lattice polarizations. 
We introduce the notion of very small ample cones $\mathcal{a}$, after Alexeev and Engel's small cones, to construct smooth, separated stacks of lattice polarized K3 surfaces $\cP_{(L,\mathcal{a})/\Z[1/N]}$ over suitable open subsets of $\spec(\Z)$.

We add level structures, coming from classes in $\rH^2(X,\mu_n)$, to build moduli stacks $\cP^{[n]}_{(L,\cA)/\Z}$ with a natural action by $\cP_{(L,\cA)}\otimes \Z/n\Z$ whose associated quotient $\cQ^{[n]}_{(L,\cA)}$ contains an open substack $\cQ^{(n)}_{(L,\cA)}$ whose points parametrize pairs K3 surfaces $X$ such that $\pic(X) \simeq L$, together with a class $\alpha \in \br(X)$ of order $n$. 

When $L$ has rank 19, we show that the coarse moduli space $\mathrm{Q}_{(L,\mathcal{a}),\C}^{(n)}$ is a union of quasi-projective curves, each isomorphic to an open subvariety of the quotient of the upper half plane by a discrete subgroup of $\SL_2(\R)$.
Fixing a prime $\ell$, we use this comparison to prove that the genus and the gonality of the components of $\mathrm{Q}_{(L,\mathcal{a}),\C}^{(\ell^{m})}$ grows with $m$, and hence that they have finitely many points over number fields of bounded degree. 
As an application, we furnish a new proof of a result by Cadoret--Charles, showing uniform boundedness of the $\ell$-primary torsion of Brauer groups of K3 surfaces over number fields varying in a $1$-dimensional lattice-polarized family.
\end{abstract}

\maketitle

\setcounter{tocdepth}{1}
\tableofcontents

\section{Introduction}

\subsection{Moduli spaces of lattice-polarized K3 surfaces}

Moduli spaces of lattice-polarized complex K3 surfaces were first studied in detail by Dolgachev~\cite{Dolgachev} in the context of mirror symmetry. 
They have since found many more applications, especially over the complex numbers, despite pervasive technical difficulties caused by their inveterate lack of separatedness. 
Recently, Alexeev and Engel \cite{AE}*{\S2.2} described a mechanism for producing separated moduli spaces, relying on a certain subdivision of the ample cones of a lattice, whose chambers they called \emph{small cones}. 
The objects of their moduli spaces include certain data involving singular cohomology, and so \emph{a priori} these spaces are defined only over $\C$.

We are interested in moduli of lattice-polarized K3 surfaces over non-closed fields, particularly number fields. 
With this motivation, we give an algebraic construction of moduli spaces of lattice-polarized K3 surfaces over $\spec\Z$ (and hence, by base change, also over arbitrary bases). 
We identify some basic geometric properties of these spaces in the general context of an arbitrary base scheme. 
We introduce a variant of the small cones of Alexeev--Engel \cite{AE}*{\S2.2}, which we call \emph{very small ample cones}. 
Unlike small cones, our very small ample cones are canonically defined only in terms of the polarizing lattice. 
We use them to construct \emph{separated} moduli spaces of lattice-polarized K3 surfaces over $\spec\Z$.

To state our results, we introduce some notation. We let
\[
    \Lambda:=U^{\oplus 3}\oplus E_8(-1)^{\oplus 2}
\]
denote the K3 lattice. 
This is an even unimodular lattice of signature $(3,19)$. 
Let $L$ be an even lattice such that $L_{\R} := L\otimes_\Z \R$ has signature $(1,n)$ for some positive integer $n$; we say $L$ is \emph{hyperbolic}. 
Let $\cA\subset L_{\R}$ be a subset. 
An \emph{$(L,\cA)$-polarized K3 surface} over an algebraically closed field $k$ is a pair $(X,j)$, where $X$ is a K3 surface over $k$ and $j:L\hookrightarrow\Pic(X)$ is an isometric embedding such that the image $j(\cA)\subset\Pic(X)_{\R}$ of $\cA$ under $j$ contains an ample class. 
This notion extends without difficulty to families, and we define $\cM_{(L,\cA)/\Z}$ to be the resulting moduli stack over $\spec\Z$. 
We show in Theorem \ref{thm:rep, easy version} that this is a Deligne--Mumford stack, quasi-separated and locally of finite type over $\spec\Z$. 
To obtain good geometric properties, we restrict our attention to the open substack $\cP_{(L,\cA)/\Z}\subset\cM_{(L,\cA)/\Z}$ of primitively $(L,\cA)$-polarized K3 surfaces. 
This substack need not be closed, owing to the existence of specializations of K3 surfaces (necessarily with supersingular special fiber) whose induced restriction map on Picard groups is non-primitive. 
To include such families, it is natural also to consider the closure $\overline{\cP}_{(L,\cA)/\Z}$ of the primitive locus in $\cM_{(L,\cA)/\Z}$.

\begin{theorem}\label{thm:intro good properties of moduli spaces}
    Let $L$ be an even hyperbolic lattice, write $d=\disc(L)$ for its discriminant, and let $\cA\subset L_{\R}$ be a subset.
    \begin{enumerate}[leftmargin=*]
        \item There exists a closed substack $\mathcal{Z}\subset\overline{\cP}_{(L,\cA)/\Z}$ that is both supported on the supersingular locus and supported over the union of the closed points in $\spec\Z$ corresponding to the prime divisors of $d$, such that $\overline{\cP}_{(L,\cA)/\Z}\setminus\mathcal{Z}$ is smooth over $\spec\Z$ of relative dimension $20-\rk(L)$ at every point. 
        In particular, $\overline{\cP}_{(L,\cA)/\Z[1/d]}$ is smooth over $\spec\Z[1/d]$ of relative dimension $20-\rk(L)$ at every point.
        \smallskip
        
        \item If $\rk(L)\leq 10$ then $\overline{\cP}_{(L,\cA)/\Z}$ is reduced and is flat and lci over $\spec\Z$, of relative dimension $20-\rk(L)$ at every point.
    \end{enumerate}
\end{theorem}

\begin{proof}
    Combine Theorems~\ref{thm:smoothness and dimension} and~\ref{thm:rank leq 10}.
\end{proof}

The stacks $\cM_{(L,\cA)/\Z}$ are often non-separated over $\spec\Z$. 
The same is true for the primitive substacks, as well as their closures. 
We show that taking $\cA = \mathcal{a}$ to be a very small ample cone (Definition~\ref{def:very small amples cones}) results in a separated moduli space, at least away from the supersingular locus.

\begin{theorem}[ $=$ Theorem~\ref{thm:separated moduli stack}]\label{thm:intro separated thm}
    Let $L$ be an even hyperbolic lattice that admits an embedding into the K3 lattice and let $\mathcal{a}\subset L_{\R}$ be a very small ample cone. 
    There exists a closed substack $\mathcal{Z}\subset\cP_{(L,\mathcal{a})/\Z}$ supported on the supersingular locus such that the complement $\cP_{(L,\mathcal{a})/\Z}\setminus\mathcal{Z}$ is separated over $\spec\Z$. 
\end{theorem}

\begin{corollary}[ $=$ Corollary~\ref{cor: separated stack over Z[1/n]}]
    There exists a positive integer $N := N(L,\mathcal{a})$ such that $\cP_{(L,\mathcal{a})/\Z}$ is separated over $\spec\Z[1/N]$. 
    In particular, the fiber $\cP_{(L,\mathcal{a})/\Q}=\cP_{(L,\mathcal{a})/\Z}\otimes_{\Z}\Q$ is separated over $\Q$.
\end{corollary}

\subsection{Incorporating twists}
We also construct moduli spaces of lattice-polarized \emph{twisted} K3 surfaces, generalizing work of the first and second authors \cites{Bragg, Brakkee}. 
While one might like to have moduli spaces parameterizing polarized K3 surfaces equipped with a Brauer class, such stacks are rarely algebraic (c.f.\ \cite{Brakkee}*{\S1.2}). 
Instead, we fix a positive integer $n$ and consider the moduli stack $\cM^{[n]}_{(L,\cA)/\Z}$ over $\spec\Z$ whose geometric points correspond to lattice-polarized K3 surfaces $X$ equipped with a class in $\rH^2(X,\mu_n)$. 
To such a class, one associates an $n$-torsion Brauer class on $X$ by taking the image under the map
\[
    \rH^2(X,\mu_n)\to\br(X)[n]
\]
obtained from the Kummer sequence. We show in Theorem \ref{thm:rep, easy version, for twisted K3s} that the stack $\cM^{[n]}_{(L,\cA)/\Z}$ is Deligne--Mumford, quasi-separated, and locally of finite type over $\Z$. 
As in the non-twisted case, to obtain good properties we restrict our attention to certain substacks. Let $\cP^{[n]}_{(L,\cA)/\Z} \subset \cM^{[n]}_{(L,\cA)/\Z}$ be the substack where the polarization is primitive, and denote by $\cP^{(n)}_{(L,\cA)/\Z}\subset \cP^{[n]}_{(L,\cA)/\Z}$ the substack where the twisting class has order exactly $n$. 
These substacks are open in $\cM^{[n]}_{(L,\cA)/\Z}$. 
We show in Theorem \ref{thm:properties of the moduli stacks, twisted version} that their closures have good geometric properties, similar to those given in Theorem \ref{thm:intro good properties of moduli spaces} in the untwisted case.
When $L$ has rank $1$, the stack $\cP_{(L,\mathcal{A})}^{[n]}$ is related to the stack of K3 surfaces with level structures of Rizov~\cite{Rizov}*{\S5}.

There is a natural forgetful morphism
\[
    \cM^{[n]}_{(L,\cA)/\Z}\to\cM_{(L,\cA)/\Z},\hspace{1cm}(X,j,\alpha)\mapsto (X,j).
\]
Over $\spec\Z[1/n]$ this map is finite \'{e}tale, with fibers isomorphic to $(\Z/n\Z)^{\oplus 22}$. 
However, over primes dividing $n$, this map need not be flat, and it even has positive-dimensional fibers over the supersingular locus. 
Despite this, our results show that the stack $\cM^{[n]}_{(L,\cA)/\Z}$ itself is frequently flat over $\spec\Z$. Our analysis of the fibers of the moduli spaces $\cM^{[n]}_{(L,\cA)/\Z}$ over primes dividing $n$ relies on some deformation theoretic results in~\cite{Bragg}.

\subsection{From twists to Brauer classes}

We consider stacks $\cQ_{(L,\mathcal{A})/\Z}^{[n]}$ which are quotients of $\cP_{(L,\mathcal{A})/\Z}^{[n]}$ and are more closely related to Brauer groups. Their geometric points correspond to K3 surfaces $X$ with a primitive lattice polarization $j\colon L\hookrightarrow \Pic(X)$ and a class $[\alpha]$ in the quotient of $\rH^2(X,\mu_n)$ by the image of $j(L)$. For the very general such K3 surface, when $j$ is an isomorphism, the quotient $\rH^2(X,\mu_n)/j(L)$ is isomorphic to $\br(X)[n]$.
The stack $\cQ_{(L,\mathcal{A})}^{[n]}$ has the same good properties as $\cM^{[n]}_{(L,\cA)/\Z}$ (see \S4.4), and the same holds for the substack $\cQ_{(L,\mathcal{A})/\Z}^{(n)}\subset\cQ_{(L,\mathcal{A})/\Z}^{[n]}$
where the class $[\alpha]$ has order exacly $n$. 
We expect the moduli spaces $\cQ^{(n)}_{(L,\cA)/\Z}$ 
to be useful in the study Brauer groups and rational points on K3 surfaces over fields of arithmetic interest.  
We offer an application in this direction in~\S\ref{ss: arithmetic application}.

\medskip

Write $\cM_{(L,\cA)/\C}=\cM_{(L,\cA)/\Z}\otimes_{\Z}\C$ for the fiber over the complex numbers, and similarly for the other moduli stacks above. The moduli stacks $\cP_{(L,\mathcal{A})/\Z}$ are closely related to those studied in~\cite{AE25}:
In Theorem~\ref{thm:torelli, unmarked version} we show that the stack $\cP_{(L,\mathcal{A})/\C}$ is a union of (open substacks of) quotient stacks of period domains, like in~\cite{AE25}*{Corollary~4.20}.
In Theorem~\ref{thm:period morphism corollary, twisted}, we show the same is true for the stacks $\cQ_{(L,\mathcal{A})/\C}^{[n]}$ and $\cQ_{(L,\mathcal{A})/\C}^{(n)}$.
We obtain the following result for their coarse moduli spaces 
$\rQ_{(L,\mathcal{A})/\C}^{[n]}$ and $\rQ_{(L,\mathcal{A})/\C}^{(n)}$; see~\S\ref{sec:complex moduli spaces} for details on how the orthogonal modular varieties $\cD(\Pi_{\epsilon})^{\circ}/\widetilde{\rO}_w(\Pi_{\epsilon})$ are defined.

\begin{corollary}[ $=$ Corollary~\ref{cor: Complex Period Space Components Comparison}]
    We have isomorphisms
    \[
        \begin{split}
            \mathrm{Q}_{(L,\mathcal{A})/\C}^{[n]}&\simeq \bigsqcup_{\substack{\epsilon\in\Emb(L,\Lambda),\\
            \omega\in\mathrm{Hom}(\Pi_{\epsilon},\Z/n)/\widetilde{\rO}(\Pi_{\epsilon})}} \cD(\Pi_{\epsilon})^{\circ}/\widetilde{\rO}_w(\Pi_{\epsilon})\\
            \mathrm{Q}_{(L,\mathcal{A})/\C}^{(n)}&\simeq \bigsqcup_{\substack{\epsilon\in\Emb(L,\Lambda),\\
            \omega\in\mathrm{Hom}(\Pi_{\epsilon},\Z/n)/\widetilde{\rO}(\Pi_{\epsilon}),\\
            \ord(\omega)=n}} \cD(\Pi_{\epsilon})^{\circ}/\widetilde{\rO}_w(\Pi_{\epsilon})\\
        \end{split}
    \]
\end{corollary}

Specializing to the case when $L$ has rank 19 and $\mathcal{A}=\mathcal{a}$ is a very small ample cone, we deduce that $\rQ^{(n)}_{(L,\mathcal{a})/\C}$ is a union of quasi-projective curves, each isomorphic to the quotient of the upper half plane by a discrete subgroup of $\SL_2(\R)$. 
Fix a prime $\ell$. 
On the orthogonal modular side, we apply results of Bergeron--Clozel~\cite{BergeronClozel_book} and Long--Maclachlan--Reid~\cite{LMR} to show that the genus of the components in $\rQ^{(\ell^m)}_{(L,\mathcal{a})/\C}$ grows with $m$, and a result of Abramovich~\cite{Abramovich} to show an analogous statement for the gonality. 
Using a standard argument of Frey~\cite{Frey}, we conclude that the spaces $\rQ^{(\ell^m)}_{(L,\mathcal{a})/\C}$ have finitely many points over number fields of bounded degree for all sufficiently large $m$. 

\begin{corollary}[$=$ Corollary~\ref{cor:finitely many degree d points}]
    \label{cor: finitely many degree d points intro}
    Let $L$ be an even hyperbolic lattice of rank $19$ that admits an embedding into the K3 lattice, and let $\mathcal{a}\subset L_{\R}$ be a very small ample cone. 
    Fix a number field $k$, a prime $\ell$, and a positive integer $d$. 
    There is an $n_0 := n_0(d,k,\ell,L) \in \Z_{> 0}$ such that, for all $m\geq n_0$, the space $\mathrm{Q}_{(L,\mathcal{a})/\C}^{(\ell^m)}$ has finitely many points over fields $K/k$ with $[K:k]\leq d$.
\end{corollary}

\subsection{An arithmetic application}
\label{ss: arithmetic application}

The coarse moduli spaces $\mathrm{Q}_{(L,\mathcal{a})}^{(\ell^m)}$ of Corollary~\ref{cor: finitely many degree d points intro} lie at the heart of a new proof of a result due to Cadoret--Charles (Theorem~\ref{thm: Cadoret--Charles}), showing uniform-boundedness of the $\ell$-primary torsion of Brauer groups of K3 surfaces over number fields varying in a $1$-dimensional lattice-polarized family; see~Theorem~\ref{thm:uniform bound}. 
To put this result in a broader context, we begin by fixing some notation and reviewing some history and developments in this direction.

Let $k$ be a number field; fix an algebraic closure $\kbar$ of it. 
Let $X/k$ be a K3 surface, and write $\Xbar$ for the base extension $X\times_{\spec k}\spec\kbar$. 
Denote by $\br(X) := \rH^2_\et(X,\G_m)_{\tors}$ the cohomological Brauer group of $X$, and let $\br_0(X) \subset \br(X)$ be the subgroup of constant algebras, i.e., the image of the pullback map $\br(\spec k) \to \br(X)$ for the structure morphism $X \to \spec k$. 
By a result of Skorobogatov and Zarhin~\cite{SkorobogatovZarhin}*{Theorem~1.2}, the quotient $\br(X)/\br_0(X)$ is finite. 
How large can it be? 
What are its possible orders?

Based on analogies between the group $\br(X)/\br_0(X)$ and torsion of elliptic curves---for example, they are both instances of unramified cohomology---see~\cite{VA-AWS}*{\S5.2}, the third named author conjectured that Brauer groups of K3 surfaces over number fields in sufficiently nice families should be uniformly bounded.  
More precisely:

\begin{conjecture}(Strong Uniform Boundedness~{\cite{VA-AWS}*{Conjecture~5.6}}) 
    \label{conj:VAstrong}
    Fix a positive integer $d$ and a primitive sublattice $L\hookrightarrow \Lambda$. 
    Let $X$ be a K3 surface of degree over a number field $k$ with $[k:\Q] \leq d$ such that $\pic(\Xbar) \simeq L$. 
    There is a constant $B := B(d,L)$, independent of $X$, such that
    \[
        \#\frac{\br(X)}{\br_0(X)} < B.
    \]
\end{conjecture}

\begin{remark}
    There is an intermediate subgroup $\br_0(X) \subset \br_1(X) \subset \br(X)$ comprising the algebraic classes in $\br(X)$, i.e., $\br_1(X) := \ker\left(\br(X) \to \br(X_\kbar)\right)$ consists of elements killed by a finite extension of the ground field. 
    The quotient $\br_1(X)/\br_0(X)$ is well-known to be uniformly bounded (see, e.g.,~\cite{VAVSigma}*{Proposition~6.3}), so one can replace $\br(X)/\br_0(X)$ in Conjecture~\ref{conj:VAstrong} with $\br(X)/\br_1(X)$.    
\end{remark}

In late 2015, Fran\c{c}ois Charles and Alexei Skorobogatov pointed out to the third-named author a conjecture of Shafarevich~\cite{Shafarevich} that allows one to dispense with the lattice $L$ in Conjecture~\ref{conj:VAstrong}, giving rise to a strong uniform boundedness conjecture, a version of which first appeared in print in~\cite{OrrSkorobogatovZarhin}*{Conjecture~$\br(\Kthree)$}.

\begin{conjecture}
    \label{conj:StrongBoundednessK3s}
    If $X$ is a K3 surface over a number field of bounded degree $d$, then the cardinality of $\br(X)/\br_0(X)$ is bounded in terms of $d$, independent of $X$.
\end{conjecture}

Any effective bound arising from special cases of Conjectures~\ref{conj:VAstrong} or~\ref{conj:StrongBoundednessK3s} has striking implications for the study of rational points on K3 surfaces. 
Indeed, by work of Kresch and Tschinkel~\cite{KreschTschinkel}*{Theorem~1.1}, such bounds imply the existence of an effective procedure that takes as input a system of homogeneous polynomial equations defining $X$ and equations for generators of $\pic(\Xbar)$ and outputs a description of the Brauer--Manin set $X(\A_k)^{\br(X)}$. 
A conjecture of Skorobogatov~\cite{SkorobogatovOberwolfach} implies this information suffices to determine whether $X$ has a $k$-rational point or not.
\medskip

Evidence for Conjectures~\ref{conj:VAstrong} or~\ref{conj:StrongBoundednessK3s} has accumulated over the last decade, including
\begin{enumerate}[leftmargin=*]
    \item Analogous uniform boundedness results by Abramovich and the third-named author for full-level structures of on abelian varieties of fixed dimension over a number field of bounded degree, conditional on Lang's Conjecture~\cite{AVAAdvances} or Vojta's Conjecture~\cites{AVACompositio,AVABordeaux}. 
    It would be quite interesting to obtain conditional results of this kind using the moduli spaces of K3 surfaces with level structures constructed in this paper, or some appropriate modifications of these spaces. 
    Brunebarbe's results on hyperbolicity of certain locally symmetric varieties~\cite{Brunebarbe} could well play a role in such an undertaking.
    \smallskip
    
    \item A proof of Conjecture~\ref{conj:StrongBoundednessK3s} for K3 surfaces with complex multiplication by Orr and Skorobogatov~\cite{OrrSkorobogatov}, and a proof in the general case by Orr, Skorobogatov, and Zarhin~\cite{OrrSkorobogatovZarhin}, conditional on a conjecture of Coleman on the finiteness of geometric endomorphism rings of abelian varieties of bounded dimension defined over number fields of bounded degree.
    \smallskip

    \item The principle of increasing hyperbolicity of varieties supporting a variation of Hodge structures with level structures, as the level increases, best exemplified in results of Brunebarbe~\cite{BrunebarbePreprint}. 
    In particular, such varieties are often of general type~\cite{MaKod}. 
    This is the case for moduli spaces $\rM_{2d}$ of high-degree K3 surfaces~\cites{Kondo,GHSK3} and $\cC_{d}$ of high-discriminant special cubic fourfolds~\cite{TVA}, each of which parametrizes K3 surfaces with Brauer level structures for well-known values of $d$~\cite{MSTVA}*{\S\S2.6,~2.7}.
\end{enumerate}
\smallskip

For a fixed prime $\ell$, one can also ask for uniform bounds of the form
\begin{equation}
    \label{eq:lprimary}
    \#\frac{\br(X)}{\br_i(X)}\{\ell^\infty\} < B_i(d,\ell,L) \text{ or } B_i(d,\ell),\quad i = 0, 1,
\end{equation}
as $X$ varies along a family of K3 surfaces (possibly polarized by a lattice $L$) over number fields of degree~$\leq d$, which could prove more accessible to current mathematical technology than Conjectures~\ref{conj:VAstrong} or~\ref{conj:StrongBoundednessK3s}. 
Indeed, decades before Merel proved uniform boundedness of torsion of elliptic curves over number fields of bounded degree~\cite{Merel}, Manin proved\footnote{Today, this follows easily from the fact that the modular curve $X_1(\ell^n)$, which is a fine moduli space, has genus $g > 1$ for $n \gg 0$, by Faltings' theorem.  But in 1969, Faltings' theorem was still Mordell's Conjecture.} that $E(k)[\ell^\infty] < B(d,\ell)$ for all elliptic curves $E$ over number fields of degree~$\leq d$~\cite{Manin}.

In~\cite{VAVSigma}, Viray and the third-named author proved the ``$\ell$-primary version''~\eqref{eq:lprimary} of Conjecture~\ref{conj:VAstrong} when $L$ is isometric to the N\'eron-Severi lattice of a Kummer surface of a product on non-CM elliptic curves related by a cyclic isogeny of fixed degree~\cite{VAVSigma}*{Corollary~1.6}. 
They also leveraged work of Gonz\'alez-Jim\'enez and Lozano-Robledo~\cite{GJLR} to show that if $E/\Q$ is a non-CM elliptic curve with a $\Q$-rational cyclic group $C$ of fixed order $n$, and if $X = \Kum(E\times E/C)$, then
\[
    \#\br(X)/\br_1(X) \leq (8n)^3,
\]
confirming that effective bounds for Conjecture~\ref{conj:VAstrong} in particular families of K3 surfaces are possible.

Using open image properties of $\ell$-adic representations of K3 surfaces, Cadoret and Charles proved a remarkable result about $\ell$-primary uniform boundedness of transcendental Brauer groups in $1$-dimensional families, of which the following is a special case.

\begin{theorem}[{\cite{CadoretCharles}*{c.f.~Theorem~1.2.1}}]
    \label{thm: Cadoret--Charles}
    Let $S$ be a curve over a number field $k$. 
    Let $\pi\colon \cX \to S$ be a smooth proper morphism such that the fiber $\cX_s$ over every geometric point $s \in S(\kbar)$ is a K3 surface. Fix a positive integer $d$ and a rational prime $\ell$. 
    There is a constant $B := B(d,\ell)$ such that for every extension $K$ of $k$ in $\kbar$ with $[K:k] \leq d$ and each point $s \in S(K)$ we have
    \[
    \# \br(\cX_{s,\kbar})^{\Gal(\kbar/K)}\{\ell^\infty\} < B.
    \]
\end{theorem}

From this result, one can deduce the $\ell$-primary version~\eqref{eq:lprimary} of Conjecture~\ref{conj:VAstrong} whenever $L \subset \Lambda$ is a rank-19 primitive sublattice of signature $(1,18)$ (c.f.~\cite{CadoretCharles}*{Proposition~2.2.3}). 
Our main result in this direction is a new, moduli-theoretic proof of this result, in the case that $S$ is a modular curve of rank-19 lattice-polarized K3 surfaces.

\begin{theorem}
    \label{thm:uniform bound}
    Fix a positive integer $d$, a rational prime $\ell$, and a rank-$19$ even lattice $L$ of signature $(1,18)$. 
    There is a constant $B := B(d,\ell,L)$ such that
    \begin{equation}
        \label{eq:uniform bound}
        \#\frac{\br(X)}{\br_1(X)}\{\ell^\infty\} < B
    \end{equation}
    for any K3 surface $X$ over a number field $k$ of degree $d$ whose geometric Picard group admits an primitive embedding $L \hookrightarrow \pic \left(X_{\overline{k}}\right)$ with an ample class in its image.
\end{theorem}

\begin{remark}
    By definition of $\br_1(X)$, the natural map $\br(X) \to \br(\overline{X})^{\Gal(\overline{k}/k)}$ yields an inclusion
    \[
        \frac{\br(X)}{\br_1(X)} \hookrightarrow \br(\overline{X})^{\Gal(\overline{k}/k)}.
    \]
    Thus, to prove Theorem~\ref{thm:uniform bound}, it would be enough to find a constant $B' = B'(d,\ell,L)$ such that
    \begin{equation}
        \label{eq:bound for Galois Invariant part}
        \#\br(\overline{X})^{\Gal(\overline{k}/k)}\{\ell^\infty\} < B'
    \end{equation}
    for any K3 surface $X$ over a number field $k$ of degree $d$, with a  primitive embedding $L \hookrightarrow \pic \left(X_{\overline{k}}\right)$ containing an ample class in its image. 
    Conversely, Theorem~\ref{thm:uniform bound} implies the seemingly more general bound~\eqref{eq:bound for Galois Invariant part}: this follows from~\cite{CTSko}*{Proposition~5}, which shows that the cokernel of the map $\br(X) \to \br(\overline{X})^{\Gal(\overline{k}/k)}$ is finite, and its order divides $\delta_0^{22 - \rho}$, where $\rho = \rk \pic \left(X_{\overline{k}}\right)$ and $\delta_0 = \disc\left(\pic \left(X_{\overline{k}}\right)\right)$. 
    If $\rho = 19$, then $\delta_0 = \disc(L)$, and if $\rho = 20$, then by~\cite{OrrSkorobogatov}*{Corollary~B.1} there are only finitely many possible values for $\delta_0$, depending on~$L$. 
    Hence, from a constant $B$ as in~\eqref{eq:uniform bound}, we can construct a uniform constant $B'$ so that~\eqref{eq:bound for Galois Invariant part} holds.
\end{remark}

\subsection*{Acknowledgments}
We are grateful to Arend Bayer, Eva Bayer--Fl\"uckiger, Pietro Bieri, Simon Brandhorst, Philip Engel, Daniel Huybrechts, Arno Kret, Stevell Muller, Andrei Rapinchuk, and Alan Reid for numerous discussions that informed and shaped many of the ideas in this paper. 
Part of this work was done during the special trimester \emph{Arithmetic geometry of K3 surfaces} at the Bernoulli Center for Fundamental Studies in 2025. 
The first author was supported by the NSF RTG grant DMS-1840190 while working on this project.
The second author was supported by NWO grants 016.Vidi.189.015 and VI.Veni.212.209.
The third author conducted some of this work, supported by the NSF under grant DMS-1928930, while in residence at the Simons Laufer Mathematical Sciences Research Institute (formerly MSRI) in Berkeley, California, during the spring of 2023. 
He was also supported by the individual NSF grants DMS-1902274 and DMS-2302231.

\subsection{Notation}

For a lattice $L$ and a ring $R$, we set $L_R:=L\otimes_{\Z}R$; this extension is equipped with an induced $R$-valued symmetric bilinear form $L_R\otimes_RL_R\to R$. 
We write $D(L)=L^{\vee}/L$ for the discriminant group of $L$. 
We write $\rO(L)$ for the orthogonal group of isometries of $L$ and $\widetilde{\rO}(L)\subset\rO(L)$ for the subgroup of isometries that act trivially on $D(L)$. 
We say $L$ is \emph{hyperbolic} if $L_\R$ has signature $(1,n)$ for some $n\geq 1$.

For conventions and notation concerning algebraic spaces and algebraic stacks we will follow the Stacks Project \cite{stacks-project}. 
For an algebraic stack $\cX$ we write $\mathcal{I}_{\cX}\to\cX$ for its inertia stack and $|\cX|$ for its associated topological space \cite{stacks-project}*{04Y8}. 
A morphism $f:\cX\to\mathcal{Y}$ induces a continuous map on associated topological spaces $|f|:|\cX|\to|\cY|$.
If $f$ is an open immersion (resp.\ a closed immersion), then $|f|$ is injective and open (resp.\ injective and closed).

\section{Moduli of Lattice-Polarized K3 Surfaces}\label{sec:moduli of lattice-polarized K3 surfaces}

In this section, we construct moduli stacks of lattice-polarized K3 surfaces over the integers. 
We fix throughout an even lattice $L$ of signature $(1,n)$. 
Write
\[
    \mathcal{C}(L):=\{x\in L_{\R}\mid x^2>0\}
\]
for the positive cone of $L$. 
It is an open subset of $L_{\R}$, and has two connected components that are exchanged under multiplication by $-1$. 
Let
\[
    \Delta(L):=\{\delta\in L\mid \delta^2=-2\}
\]
denote the set of $(-2)$-classes in $L$. 
A $(-2)$-class $\delta\in\Delta(L)$ determines a reflection $r_{\delta}\in\mathrm{O}(L_{\R})$ by the usual formula $x\mapsto x+(x.\delta)\delta$. 
The fixed set of the reflection $r_{\delta}$ is the wall $\delta^{\perp}\subset L_{\R}$. 
We write $W(L)\subset\rO(L)$ for the subgroup generated by the reflections $r_{\delta}$ for $\delta\in\Delta(L)$ and $\pm W(L)$ for the subgroup generated by $W(L)$ and $-\mathrm{id}_L$. 
Set
\[
    \cC(L)^{\circ}:=\cC(L)\setminus\left(\textstyle\bigcup_{\delta\in\Delta_L}\delta^{\perp}\right)\subset\cC(L).
\]
That is, $\cC(L)^{\circ}$ is the set of $x\in\cC(L)$ such that $x.\delta\neq 0$ for all $\delta\in\Delta(L)$. 
An \emph{ample cone} of $L$ is a connected component of $\cC(L)^{\circ}$. 
The union of walls $\bigcup_{\delta\in\Delta(L)}\delta^{\perp}$ is closed in $\cC(L)$ \cite{Huy16}*{8, \S2.2 ff.}, so $\cC(L)^{\circ}$ is an open subset of $L_{\R}$. 
Furthermore, if $\cA\subset\cC(L)^{\circ}$ is an ample cone, then $\cA$ intersects the integral lattice $L\subset L_{\R}$ nontrivially, and its closure $\overline{\cA}$ is locally polyhedral in the interior $\cC(L)^{\circ}$ of $\cC(L)$ \cite{Huy16}*{8, \S2.4}. 
Finally, the group $\pm W(L)$ acts simply and transitively on the set of ample cones $\cA\subset\cC(L)^{\circ}$ \cite{Huy16}*{8, \S2.3}.


If $X$ is a K3 surface over an algebraically closed field then we write $\cC(X)$ and $\Delta(X)$ for the above notions applied to the lattice $L=\Pic(X)$. 
In this case, the eponymous ample cone
\[
    \Amp(X)=\left\{\textstyle\sum_i\lambda_ix_i\,|\,\lambda_i\in\R_{\geq 0}\text{ and }x_i\in\Pic(X)\text{ ample}\right\}\subset\Pic(X)_{\R}
\]
provides a distinguished choice of ample cone for the lattice $\Pic(X)$.

\begin{definition}[Lattice polarizations]
    Let $X$ be a K3 surface over an algebraically closed field. 
    An \emph{$L$-lattice polarization} on $X$ is an isometric embedding
    \[
        j\colon L\hookrightarrow\pic(X)
    \]
    of lattices (not necessarily primitive) whose image $j(L)\subset\pic(X)$ contains an ample class.

    Given a subset $\cA\subset L_{\R}$, an \emph{$(L,\cA)$-lattice polarization} (or just $(L,\cA)$-polarization) on $X$ is an $L$-lattice polarization $j$ such that $j(\cA)\subset\Pic(X)_{\R}$ contains an integral ample class, that is, such that $j(\cA)$ intersects $\Amp(X)\cap\Pic(X)$ nontrivially.
\end{definition}

If $\cA\subset L_{\R}$ is a subset then an isometric embedding $j:L\hookrightarrow\Pic(X)$ is an $(L,\cA)$-polarization if and only if it is an $(L,\cA\cap L)$-polarization. In the case when $\cA=L_{\R}$ (or when $\cA=L$), an $(L,L_{\R})$-lattice polarization is the same thing as an $L$-lattice polarization. Other cases we have in mind are when $\cA$ is an ample cone of $L$ and when $\cA=\left\{h\right\}$ for a single vector $h\in L$ with positive square.

\begin{remark}
    Let $X$ be a K3 surface, let $\cA\subset L_{\R}$ be an ample cone for $L$, and let $j:L\hookrightarrow\Pic(X)$ be an $(L,\cA)$-lattice polarization on $X$. 
    We then have $j(\cC(L))\subset\cC(X)$ and $j(\Delta(L))\subset\Delta(X)$. However, $j$ need not map $\cA$ into the ample cone of $X$, as there may be $(-2)$-classes in $\Pic(X)$ that lie outside the image of $j$.
\end{remark}

\begin{remark}
   One can generalize the above definition by replacing ampleness with a weaker positivity property, e.g., quasi-ampleness~\cites{Dolgachev,AE,AE25}. 
   In this more general setting, the above should be called \emph{ample lattice polarizations}.
\end{remark}

\begin{definition}[Lattice polarizations in families]
    A \emph{family of K3 surfaces} over a base scheme $T$ is an algebraic space $X$ equipped with a smooth proper finitely presented morphism $X\to T$ all of whose geometric fibers are K3 surfaces. 
    Let $f\colon X\to T$ be a family of K3 surfaces. 
    A \emph{family of $L$-lattice polarizations} on $f\colon X\to T$ is an embedding
    \[
        j\colon\underline{L}_T\hookrightarrow\pic_{X/T}
    \]
    compatible with the respective pairings such that for every geometric point $t\in T$ the fiber $j_t\colon L\hookrightarrow\pic(X_t)$ is a lattice polarization on $X_t$. 
    Given a subset $\cA\subset L_{\R}$, we say that $j$ is a \emph{family of $(L,\cA)$-lattice polarizations} if each geometric fiber is an $(L,\cA)$-lattice polarization. 
    We say that $j$ is \emph{primitive} if its restriction to every geometric fiber is a primitive embedding of lattices.
\end{definition}

\subsection{Moduli of lattice-polarized K3 surfaces}


\begin{definition}[The moduli stack of $(L,\cA)$-polarized K3 surfaces]
    \label{def:moduli stack of lattice-pol K3s}
    Let $S$ be a base scheme and let $\cA\subset L_{\R}$ be a subset. 
    The \emph{moduli stack of $(L,\cA)$-polarized $\K3$-surfaces over $S$}, denoted $\cM_{(L,\cA)/S}$, is the category fibered in groupoids over the category of $S$-schemes whose objects over an $S$-scheme $T$ are pairs $(X,j)$, where $X\to T$ is a family of $\K3$-surfaces and $j\colon\underline{L}_T\to \pic_{X/T}$ is a family of $(L,\cA)$-lattice polarizations, and whose morphisms between objects $(X,j)\to (X',j')$ lying over a morphism $v:T\to T'$ of $S$ schemes are morphisms $u:X\to X'$ such that the diagram
    \[
        \begin{tikzcd}
            X\arrow{d}\arrow{r}{u}&X'\arrow{d}\\
            T\arrow{r}{v}&T'
        \end{tikzcd}
    \]
    commutes and is Cartesian, and also the diagram
    \[
        \begin{tikzcd}
            v^*\underline{L}_{T'}\arrow[hook]{d}[swap]{v^*j'}\arrow{r}{\sim}&\underline{L}_{T}\arrow[hook]{d}{j}\\
            v^*\Pic_{X'/T'}\arrow{r}{\sim}[swap]{u^*}&\Pic_{X/T}
        \end{tikzcd}
    \]
    commutes.
    
    We write
    \[
        \cP_{(L,\cA)/S}\subset\cM_{(L,\cA)/S}
    \]
    for the full substack whose objects are families of primitively $(L,\cA)$-polarized K3 surfaces.
    
    In the special case when $\cA=L_{\R}$, the stack $\cM_{(L,L_{\R})}$ (resp.\ $\cP_{(L,L_{\R})}$) is the usual moduli stack of $L$-polarized K3 surfaces (resp. primitively $L$-polarized K3 surfaces). We denote it by
    \[
        \cM_{L/S}:=\cM_{(L,L_{\R})/S}\hspace{1cm}(\text{resp.\ }\cP_{L/S}:=\cP_{(L,L_{\R})/S}).
    \]
    If confusion is unlikely to occur, we will suppress the base scheme $S$ from the notation and write $\cM_{(L,\cA)}$ for $\cM_{(L,\cA)/S}$ and so on.
\end{definition}

\begin{remark}
    When working with moduli spaces of lattice-polarized K3 surfaces over the complex numbers, it is traditional---and mostly harmless---to work only with primitive lattice polarizations. 
    Primitive polarizations do not suffice in positive characteristic to include certain families involving specializations into/within the supersingular locus. 
    Thus, it seems prudent to consider the entire stack $\cM_{(L,\cA)/S}$ until we have reason to do otherwise.
\end{remark}

\begin{remark}
    For a subset $\cA\subset L_{\R}$, giving an $(L,\cA)$-polarization is the same as giving an $(L,\cA\cap L)$-polarization. 
    Thus we have a canonical isomorphism $\cM_{(L,\cA\cap L)/S}\simeq\cM_{(L,\cA)/S}$ of stacks. 
    We can therefore safely identify the stacks $\cM_{(L,\cA)/S}$ and $\cM_{(L,\cB)/S}$ when $\cA,\cB\subset L_{\R}$ are subsets such that $\cA\cap L=\cB\cap L$. 
\end{remark}

\begin{remark}\label{rem:remark1}
    If $\cA\subset L_{\R}$ is a subset that does not contain an integral class with positive square that is orthogonal to every $(-2)$-class in $L$, then $\cM_{(L,\cA)/S}$ is empty. 
    Thus, if $\cM_{(L,\cA)/S}$ is nonempty then $\cA$ must intersect $\cC(L)^{\circ}\cap L$ nontrivially, and hence $\cA$ must contain an integral class in some ample cone of $L$.
\end{remark}

\begin{example}
     When $\cA=\left\{h\right\}$ for a single vector $h\in L$ we write $\cM_{(L,h)/S}:=\cM_{(L,\left\{h\right\})/S}$. 
     As in Remark \ref{rem:remark1}, if $h$ is not contained in $\cC(L)^{\circ}\cap L$ then $\cM_{(L,h)/S}$ is empty.
\end{example}

We now fix a base scheme $S$, which we immediately suppress from the notation. 
Given subsets $\cA\subset\cB\subset L_{\R}$, an $(L,\cA)$-polarization is in particular an $(L,\cB)$-polarization. 
This gives a morphism of stacks
\[
    \iota_{(\cA,\cB)}:\cM_{(L,\cA)}\to\cM_{(L,\cB)}.
\]
In particular, for any subset $\cA\subset L_{\R}$ we have a morphism
\[
    \iota_{(\cA,L_{\R})}:\cM_{(L,\cA)}\to\cM_{(L,L_{\R})}=:\cM_L.
\]
At the opposite extreme, for a single integral vector $h\in\cA\cap\cC(L)^{\circ}\cap L$ we have a morphism 
\[
    \iota_{(h,\cA)}:\cM_{(L,h)}\to\cM_{(L,\cA)}.
\]
Taking the disjoint union of the latter as $h$ ranges over all integral vectors in $\cA\cap\cC(L)^{\circ}$ we obtain a morphism
\[
    \pi_{\cA}:\bigsqcup_{h\in\cA\cap\cC(L)^{\circ}\cap L}\cM_{(L,h)}\to\cM_{(L,\cA)}.
\]

\begin{lemma}\label{lem:open immersion}
    For any subsets $\cA\subset\cB\subset L_{\R}$ the morphism $\iota_{(\cA,\cB)}$ is an open immersion and the morphism $\pi_{\cA}$ is a Zariski open cover.
\end{lemma}

\begin{proof}
    Let $h\in\cA\cap\cC(L)^{\circ}\cap L$ be an integral element of $\cA$ with positive square that is orthogonal to every $(-2)$-class in $L$. 
    We first show that $\iota_{(h,\cA)}$ is an open immersion. 
    Let $T$ be an $S$-scheme, let $(X,j)$ be a family of $(L,\cA)$-polarized K3 surfaces over $T$, and let $\mu:T\to\cM_{(L,\cA)}$ be the corresponding morphism. 
    Let $U$ denote the fiber product in the diagram
    \[
        \begin{tikzcd}
            U\arrow{r}\arrow[hook]{d}&\cM_{(L,h)}\arrow[hook]{d}{\iota_{(h,\cA)}}\\
            T\arrow{r}{\mu}&\cM_{(L,\cA)}.
        \end{tikzcd}
    \]
    The locus of geometric points of $T$ over which the section $j(h)\in\Pic_{X/T}(T)$ restricts to an ample class on the fiber is open \cite{EGA4.3}*{Corollaire~9.6.4}, so $U\hookrightarrow T$ is an open immersion. 
    It follows that $\iota_{(h,\cA)}$ is an open immersion.
    
    We next show that $\pi_{\cA}$ is a Zariski cover. 
    By the above, $\pi_{\cA}$ restricts to an open immersion on each term of the disjoint union. 
    To see that $\pi_{\cA}$ is surjective, note that by definition the image of $\cA$ under an $(L,\cA)$-polarization contains an integral ample class, which must be the image of an element of $\cA\cap\cC(L)^{\circ}\cap L$.
    
    It remains to show that $\iota_{(\cA,\cB)}$ is an open immersion. 
    For this note that the composition
    \[
        \bigsqcup_{h\in\cA\cap\cC(L)^{\circ}\cap L}\cM_{(L,h)}\xrightarrow{\pi_{\cA}}\cM_{(L,\cA)}\xrightarrow{\iota_{(\cA,\cB)}}\cM_{(L,\cB)}
    \]
    is a disjoint union of the open immersions $\iota_{(h,\cB)}$ for some integral classes $h\in\cB\cap\cC(L)^{\circ}\cap L$. 
    We have already shown that $\pi_{\cA}$ is a Zariski cover, so $\iota_{(\cA,\cB)}$ is an open immersion.
\end{proof}

We recall that by definition an algebraic stack is \emph{separated} over $S$ if its relative diagonal over $S$ is proper \cite{stacks-project}*{04YW}. For a Deligne--Mumford stack, this is equivalent to the relative diagonal being finite.
We say that an algebraic stack $\cX$ over $S$ is \emph{Zariski-locally separated} over $S$ if it admits a Zariski cover $\cU\to\cX$ where $\cU$ is an algebraic stack that is separated over $S$.

\begin{theorem}\label{thm:rep, easy version}
    For any subset $\cA\subset L_{\R}$, the stack $\cM_{(L,\cA)}$ is algebraic, is Deligne--Mumford, locally of finite type, quasi-separated, and Zariski-locally separated over $S$, and has finite inertia.
\end{theorem}

\begin{proof}
    We first consider the case when $\cA=\left\{h\right\}$ for a single integral vector $h\in L$ with positive square, say $h^2=2d$. 
    Let $\cM_{2d}$ denote the usual moduli stack of (not necessarily primitively) polarized K3 surfaces of degree $2d$. 
    This is a Deligne--Mumford stack that is of finite type and separated over $S$ (see \cite{Rizov}*{Theorem~4.3.3}, or \cite{Huy16}*{Theorem 5.1.3} for the primitively polarized substack). 
    Consider the forgetful morphism
    \[
        \phi:\cM_{(L,h)}\to\cM_{2d}.
    \]
    Let $T$ be an $S$-scheme, let $X\to T$ be a family of K3 surfaces, let $v\in\Pic_{X/T}(T)$ be a section of the relative Picard scheme that restricts to an ample class of degree $2d$ over every geometric point of $T$, and let $\mu:T\to\cM_{2d}$ be the resulting morphism. 
    Let $\cF_T$ denote the fiber product in the diagram
    \[
        \begin{tikzcd}
            \cF_T\arrow{r}\arrow{d}&\cM_{(L,h)}\arrow{d}{\phi}\\
            T\arrow{r}{\mu}&\cM_{2d}.
        \end{tikzcd}
    \]
    Let $\mathcal{Hom}_T(\underline{L}_T,\Pic_{X/T})$ denote the sheaf on the category of $T$-schemes whose sections over a $T$-scheme $T'\to T$ are homomorphisms of $T'$-group schemes $\underline{L}_{T'}\to\Pic_{X_{T'}/T'}$. 
    This sheaf is representable, locally of finite type, and separated over $T$. 
    Indeed, if $L$ has rank $r$, then choosing a basis for $L$ gives rise to an isomorphism
    \[
        \mathcal{Hom}_T(\underline{L}_T,\Pic_{X/T})\simeq\left(\pic_{X/T}\right)^r.
    \]
    We have an injective map
    \[
        \cF_T\hookrightarrow\mathcal{Hom}_T(\underline{L}_T,\Pic_{X/T})
    \]
    which identifies $\cF_T$ with the subfunctor of the right-hand side consisting of those homomorphisms that are compatible with the pairings (which implies injectivity) and send $h$ to $v$. 
    These conditions are open and closed in families, so $\cF_T$ is an open and closed subfunctor of $\mathcal{Hom}_T(\underline{L}_T,\Pic_{X/T})$.
    In particular, $\cF_T$ is representable, locally of finite type, and separated. 
    Thus, the morphism $\phi$ is representable and is locally of finite type and separated over $T$. 
    From the properties of $\cM_{2d}$ noted above, it follows that $\cM_{(L,h)}$ is Deligne--Mumford, locally of finite type, and separated over $S$.

    We now consider the case of an arbitrary subset $\cA\subset L_{\R}$.
    By Lemma \ref{lem:open immersion}, we have a Zariski cover of stacks
    \[
        \pi_{\cA}:\bigsqcup_{h\in\cA\cap\cC(L)^{\circ}\cap L}\cM_{(L,h)}\to\cM_{(L,\cA)}.
    \]
    We may verify the properties of being algebraic, Deligne--Mumford, and locally of finite type after passing to a Zariski cover. 
    Thus these properties for $\cM_{(L,\cA)}$ follow from the above results for the $\cM_{(L,h)}$. 
    We have shown that each $\cM_{(L,h)}$ is separated, so $\cM_{(L,\cA)}$ is Zariski-locally separated. 
    By Lemma \ref{lem:quasi-sep lemma}, the inertia of $\cM_{(L,\cA)}$ is proper, and hence finite.
    
    It remains only to show that $\cM_{(L,\cA)}$ is quasi-separated over $S$. 
    This stack is obtained by base change from the stack $\cM_{(L,\cA)/\Z}$ over $\Z$. 
    Thus, it suffices to prove the result in this case, and in particular, to prove the result, we may assume that $S$ is Noetherian. 
    By the above, we know that each of the $\cM_{(L,h)}$ is quasi-compact, separated, and of finite type over $S$. 
    It follows that each $\cM_{(L,h)}$ is Noetherian (in the sense of \cite{stacks-project}*{0510}). 
    By Lemma \ref{lem:quasi-sep lemma} $\cM_{(L,\cA)}$ is quasi-separated over $S$.
\end{proof}

\subsection{Primitively polarized K3 surfaces}

We consider the substack $\cP_{(L,\cA)}\subset\cM_{(L,\cA)}$ of primitively polarized K3 surfaces. 

\begin{proposition}\label{prop:primitive substack is open}
    The inclusion $\cP_{(L,\cA)}\hookrightarrow\cM_{(L,\cA)}$ is an open immersion.
\end{proposition}

\begin{proof}
    Let $(X,j)$ be a family of $(L,\cA)$-polarized K3 surfaces over an $S$-scheme $T$. 
    Let $\mu:T\to\cM_{(L,\cA)}$ denote the corresponding morphism and let $U$ denote the fiber product in the Cartesian diagram
    \[
        \begin{tikzcd}
            U\arrow[hook]{d}\arrow{r}&\cP_{(L,\cA)}\arrow[hook]{d}\\
            T\arrow{r}{\mu}&\cM_{(L,\cA)}.
        \end{tikzcd}
    \]
    The image of the map $U\hookrightarrow T$ is the locus in $T$ of points over which the polarization $j:\underline{L}_T\hookrightarrow\Pic_{X/T}$ is primitive. 
    Let $G$ denote the cokernel of $j$ regarded as a homomorphism of fppf sheaves of abelian groups, so that we have a short exact sequence
    \[
        0\to\underline{L}_T\to\Pic_{X/T}\to G\to 0.
    \]
    As $\underline{L}_T$ is flat over $T$, the quotient $G$ is representable by a $T$-group scheme. 
    The fiber product of the identity section of $G$ with the quotient map $\Pic_{X/T}\twoheadrightarrow Q$ is exactly $\underline{L}_T$, which is closed in $\Pic_{X/T}$. 
    As the quotient map is surjective and flat, it follows that the identity morphism of $G$ is a closed immersion, and thus $G$ is separated \cite{stacks-project}*{047G}. 
    Furthermore, as $X/T$ is smooth, $\Pic_{X/T}$ is \emph{essentially proper} over $T$, meaning that it satisfies the existence and uniqueness parts of the valuative criterion for properness. 
    It follows that the same is true for $G$. 
    Let
    \[
        G_{\mathrm{tors}}:=\bigcup_{n\geq 1}\ker(G\xrightarrow{\cdot n}G)
    \]
    denote the torsion part of $G$. The torsion in $G$ is bounded, so $G_{\mathrm{tors}}$ is representable, is closed in $G$, and is quasi-finite over $T$. 
    Therefore $G_{\mathrm{tors}}$ is finite over $T$. 
    It follows that the image of the morphism $G_{\mathrm{tors}}\to T$ is closed. 
    The subscheme $U$ is its complement, and hence is open.
\end{proof}

If $S$ is defined over $\Q$, then we show in Proposition \ref{prop:boundary on ssing locus} below that the primitively polarized substack $\cP_{(L,\cA)}\subset\cM_{(L,\cA)}$ is also closed. 
This is not the case in general, however, because of the existence of specializations of K3 surfaces (necessarily with supersingular special fibers) whose associated specialization map on Picard groups is non-primitive. 
Thus, from a geometric perspective, it seems most natural to consider not only the primitively polarized substack but also its closure in $\cM_{(L,\cA)}$.

\begin{definition}
   We denote the closure of $\cP_{(L,\cA)}$ in $\cM_{(L,\cA)}$ by $\overline{\cP}_{(L,\cA)}\subset\cM_{(L,\cA)}$ (or if we wish to include $S$ in the notation, by $\overline{\cP}_{(L,\cA)/S}$).
\end{definition}

The properties in the conclusion of Theorem \ref{thm:rep, easy version} are inherited by substacks. 
Thus $\overline{\cP}_{(L,\cA)}$ is algebraic, is Deligne--Mumford, locally of finite type, quasi-separated, and Zariski-locally separated over $S$, and has finite inertia.

\begin{remark}
    The stack $\overline{\cP}_{(L,\cA)}$ recovers the moduli stack of primitively $(L,\cA)$-polarized K3 surfaces if $S$ is defined over $\Q$, and also recovers Ogus's moduli space of lattice-polarized supersingular K3 surfaces \cite{Ogus83} if $S$ is defined in positive characteristic and $L$ is taken to be a supersingular K3 lattice.
\end{remark}

\subsection{Properties of the main component}

Complications of various kinds tend to arise along the supersingular locus. 
In this paper, we will avoid considering such issues and focus on results that hold away from the supersingular locus. 
Let $p$ be a prime. For a scheme $T$ the \emph{characteristic $p$ locus} of $T$ is the closed subscheme $T_p:=T\otimes_{\Z}\F_p\subset T$. 
By the compatibility of the moduli stack with base change, we have
\[
    (\cM_{(L,\cA)})_p:=(\cM_{(L,\cA)/S})_p\simeq\cM_{(L,\cA)/S_p}.
\]
The infinitesimal structure of the supersingular locus (i.e., the non-reduced structure) is subtle. 
We shall ignore it and only use the reduced structure. 
The \emph{(reduced) characteristic $p$ supersingular locus} in $\cM_{(L,\mathcal{A})}$, denoted
\[
    (\cM_{(L,\cA)})^{\mathrm{ss}}_p\subset(\cM_{(L,\cA)})_p,
\]
is the unique reduced substack of the characteristic $p$ locus of $\cM_{(L,\cA)}$ whose objects over an $S$-scheme $T$ of characteristic $p$ are families of $(L,\cA)$-polarized K3 surfaces $(f:X\to T_{\mathrm{red}},j)$ over the reduction of $T$ such that all geometric fibers of $f$ are supersingular. 
The \emph{(reduced) supersingular locus} in $\cM_{(L,\cA)}$ is the substack
\[
    \bigsqcup_{p\text{ prime}}(\cM_{(L,\cA)})^{\mathrm{ss}}_p\subset\cM_{(L,\cA)}.
\]
In particular, if $S$ is defined over $\Q$, then the supersingular locus is empty. 
We similarly define the supersingular loci in substacks of $\cM_{(L,\cA)}$. 
In general, the supersingular locus is closed under specialization, but need not itself be closed in $\cM_{(L,\cA)}$ (see Remark \ref{rem:ssing locus}).

\begin{remark}
    If $S$ has infinitely many different residue characteristics, then the supersingular locus may have infinitely many components, and some care is needed when working with this definition. 
    For example, in such situations the image of the supersingular locus need not be a closed substack of $\cM_{(L,\cA)}$ (Remark \ref{rem:ssing locus}), the inclusion of the supersingular locus need not be a homeomorphism onto its image, and thus need not be an immersion, and the non-supersingular locus (the complement of the supersingular locus in $\cM_{(L,\cA)}$) need not have any algebraic stack structure. 
    (To see the potential issue, consider the subscheme $\left\{n|n\in\Z\right\}$ of $\A^1=\A^1_{\Q}$. 
    This subscheme is a disjoint union of infinitely many closed points, but is not itself closed in $\A^1$, the inclusion $\left\{n|n\in\Z\right\}\hookrightarrow\A^1$ is not an immersion, and the complement $\A^1\setminus\left\{n|n\in\Z\right\}$ is not a scheme, as none of its points admit open affine neighborhoods).
\end{remark}

\begin{pg}[The specialization map]\label{pg:specialization map}
    Let $R$ be a DVR with residue field $k$ and field of fractions $K$. 
    Let $\eta\in\spec R$ (resp.\ $t\in\spec R$) be the generic (resp.\ closed) point. 
    Let $\overline{K}$ be an algebraic closure of $K$ and let $\overline{\eta}$ be the corresponding geometric generic point of $\spec R$. 
    Let $\overline{R}$ be the integral closure of $R$ in $\overline{K}$ and choose a closed point say $t'$ of $\spec\overline{R}$. 
    Then the residue field of $t'$ is an algebraically closed extension of the residue field of $R$. 
    Write $\overline{t}$ for the resulting geometric closed point of $\spec R$. For a family of K3 surfaces $X$ over $R$, there is a canonical (depending only on the above choices) \emph{specialization map}
    \begin{equation}\label{eq:specialization map}
        \mathrm{sp}:\Pic(X_{\overline{\eta}})\hookrightarrow\Pic(X_{\overline{t}}).
    \end{equation}
For the construction of this map, see \cite{MaulikPoonen}*{Propositions~3.3 \&~3.6} and the associated references. 
This map is compatible with the pairings and is thus always injective. Furthermore, if $X_{\overline{t}}$ (and hence $X_{\overline{\eta}}$) is non-supersingular then the specialization map is primitive.
In the case where $X_{\overline{t}}$ is supersingular, the specialization map may be non-primitive; in this case, the torsion part of its cokernel is $p$-primary, where $p$ is the characteristic of the residue field of $R$.
\end{pg}

\begin{proposition}\label{prop:boundary on ssing locus}
    The boundary $\overline{\cP}_{(L,\cA)}\setminus\cP_{(L,\cA)}$ is supported on the supersingular locus.
\end{proposition}

\begin{proof}
    Let $y\in |\overline{\cP}_{(L,\cA)}\setminus\cP_{(L,\cA)}|=|\overline{\cP}_{(L,\cA)}|\setminus|\cP_{(L,\cA)}|$ be a point of the boundary. 
    Then there exists a point $x\in|\cP_{(L,\cA)}|$ in the interior that specializes to $y$. 
    By Lemma \ref{lem:specializations for stacks} we can find a DVR $R$ and a family $(X,j)$ of $(L,\cA)$-polarized K3 surfaces over $R$ such that the induced morphism $\spec R\to\cM_{(L,\cA)}$ maps the generic point to $x$ and the closed point to $y$. 
    Choose geometric points $\overline{\eta},\overline{t}\in\spec R$ as in \ref{pg:specialization map}. We have a commuting diagram
    \[
        \begin{tikzcd}[column sep=small]
            &L\arrow[hook]{dr}{j}&\\
            \Pic(X_{\overline{\eta}})\arrow[hookleftarrow]{ur}{j}\arrow[hook]{rr}{\mathrm{sp}}&&\Pic(X_{\overline{t}}).
        \end{tikzcd}
    \]
    By assumption the polarization $j$ restricts to a primitive embedding $L\hookrightarrow\Pic(X_{\overline{\eta}})$ and a non-primitive embedding $L\hookrightarrow\Pic(X_{\overline{t}})$. 
    Thus the specialization map itself must be non-primitive, so by the above $X_{\overline{t}}$ is supersingular and thus $y$ is in the supersingular locus.
\end{proof}

\begin{corollary}\label{cor:boundary corollary}
    There exists a closed substack $\cZ\subset\cM_{(L,\cA)}\setminus\cP_{(L,\cA)}$ that is supported on the supersingular locus such that $\cP_{(L,\cA)}$ is open and closed in $\cM_{(L,\cA)}\setminus\cZ$.
\end{corollary}

\begin{proof}
    Take $\cZ=\overline{\cP}_{(L,\cA)}\setminus\cP_{(L,\cA)}$. 
    It is immediate that $\cP_{(L,\cA)}$ is open and closed in $\cM_{(L,\cA)}\setminus\cZ$, and $\cZ$ is supported on the supersingular locus by Proposition \ref{prop:boundary on ssing locus}.
\end{proof}

\begin{corollary}
    If $S$ is defined over $\Q$ then the primitively polarized substack $\cP_{(L,\cA)}$ is open and closed in $\cM_{(L,\cA)}$.
\end{corollary}

\begin{proof}
    If $S$ is defined over $\Q$, then the supersingular locus is empty, so this follows from Corollary \ref{cor:boundary corollary}.
\end{proof}

\begin{theorem}\label{thm:smoothness and dimension}
    There exists a closed substack $\mathcal{Z}\subset\overline{\cP}_{(L,\cA)}$ that is both (1) supported on the supersingular locus, and (2) supported over the union of the characteristic $p$ loci $S\otimes\F_p$ of $S$ for primes $p$ dividing the discriminant of $L$, such that $\overline{\cP}_{(L,\cA)}\setminus\mathcal{Z}$ is smooth over $S$ of relative dimension\footnote{By convention, we interpret the statement that a stack has negative (relative) dimension as equivalent to it being empty.} $20-\rk(L)$.
\end{theorem}

\begin{proof}
    Set $d=\disc(L)$, $r=\rk(L)$, $\overline{\cP}=\overline{\cP}_{(L,\cA)}$, and $\cM=\cM_{(L,\cA)}$. 
    We first claim that $\overline{\cP}$ is smooth over $S$ of relative dimension $20-\rk(L)$ except possibly at points in the supersingular locus whose residue characteristic divides $d$. 
    Every point of $\overline{\cP}$ has a Zariski open neighborhood of the form $\overline{\cP}_{(L,h)}$ for some $h\in\cC(L)^{\circ}\cap L$, so to prove this we may assume that $\cA=\left\{h\right\}$ for some such $h$. 
    Let $k$ be an algebraically closed field of positive characteristic $p$. 
    Let $(X,j)$ be a $(L,h)$-polarized K3 surface over $k$ corresponding to a $k$-point $x\in\overline{\cP}(k)$. 
    By Proposition \ref{prop:boundary on ssing locus}, if $X$ is non-supersingular, then $x$ is not in the boundary, so $j$ is a primitive polarization. 
    By the openness of the ample locus the formal neighborhood of $x$ in $\overline{\cP}$ (or in $\cM=\cM_{(L,h)}$) is identified with the universal deformation space $\mathrm{Def}_{(X,j(L))}$ of $(X,j(L))$, which is the functor on Artinian local rings with residue field identified with $k$ whose value on such a ring $A$ is the set of isomorphism classes of pairs $(X_A,j_A)$, where $X_A$ is a flat deformation of $X$ over $A$ and $j_A:L\hookrightarrow\Pic(X_A)$ is an isometric inclusion whose composition with the restriction map $\Pic(X_A)\to\Pic(X)$ is equal to $j$. 
    Let $V\subset\rH^1(X,\Omega^1_X)$ be the $k$-subspace generated by the image of $j(L)$ under the first Hodge Chern class map
    \begin{equation}\label{eq:first Hodge Chern class}
        c_1:\Pic(X)\to\rH^1(X,\Omega^1_X).    
    \end{equation}
    Let $V^{\perp}\subset\rH^1(X,T_X)$ be the orthogonal to $V$ under the natural perfect pairing
    \begin{equation}\label{eq:cup product pairing}
        \rH^1(X,T_X)\times\rH^1(X,\Omega^1_X)\to\rH^2(X,\cO_X).
    \end{equation}
    By standard results in deformation theory, the tangent space to $\mathrm{Def}_{(X,L)}$ is naturally identified with $V^{\perp}$ (see e.g., \cite{Deligne}). 
    On the other hand, the universal deformation space $\mathrm{Def}_X$ of $X$ is formally smooth over $W$ of dimension $20$, hence is isomorphic to the formal scheme $\mathrm{Spf}\,W[[t_1,\ldots,t_{20}]]$, where $W$ is the ring of Witt vectors of $k$. 
    Furthermore, the forgetful morphism identifies $\mathrm{Def}_{(X,j(L))}$ with the closed formal subscheme
    \begin{equation}\label{eq:local def theory}
        \mathrm{Spf}\,W[[t_1,\ldots,t_{20}]]/(f_1,\ldots,f_r)\subset\mathrm{Spf}\,W[[t_1,\ldots,t_{20}]]
    \end{equation}
    for some $f_1,\ldots,f_r\in W[[t_1,\ldots,t_{20}]]$. 
    This implies that the tangent space to $\mathrm{Def}_{(X,j(L))}$ has dimension $\geq 20-r$, and that if equality holds then $\mathrm{Def}_{(X,j(L))}$ is formally smooth over $W$ of dimension $20-r$. 
    Suppose that this is not the case, so that the tangent space has dimension $>20-r$. 
    By the perfectness of the pairing~\eqref{eq:cup product pairing}, we have
    \[
        \dim_kV^{\perp}=20-\dim_kV,
    \]
    so this implies that $\dim_kV<r$, and therefore the composition
    \[
        L\otimes k\xrightarrow{j\otimes k}\Pic(X)\otimes k\xrightarrow{c_1\otimes k}\rH^1(X,\Omega^1_X)
    \]
    is not injective. 
    This map is compatible with the respective $k$-valued pairings on $L\otimes k$ and on $\rH^1(X,\Omega^1_X)$. 
    The pairing on the latter is perfect, so the pairing on $L\otimes k$ must be degenerate. 
    This pairing is obtained by tensoring the $\F_p$-valued pairing on $L\otimes\F_p$ by $k$, so this implies that the pairing on $L\otimes\F_p$ is degenerate, hence $p$ divides $d$. 
    Furthermore, it implies that at least one of $j\otimes k$ and $c_1\otimes k$ must be non-injective. 
    If $j\otimes k$ is not injective, then by faithful flatness $j\otimes\F_p$ is not injective, so $j$ is not primitive. 
    As noted above, this can only occur on the boundary, and \ref{prop:boundary on ssing locus} shows that $X$ is supersingular. 
    On the other hand, if $c_1\otimes k$ is non-injective, then $c_1\otimes\F_p$ must be non-injective, which implies that $X$ is supersingular \cite{Ogus78}*{Corollary~1.4} (in fact $X$ must be supersingular of Artin invariant 1, which determines $X$ up to isomorphism). 
    In either case, $X$ is supersingular, which proves the claim.
    
    We can now deduce the result from this. 
    Suppose that there is a point of $\overline{\cP}$ that is both contained in the non-supersingular locus and whose residue characteristic does not divide $d$ (if there is no such point, then we may take $\cZ=\overline{\cP}$ and there is nothing to prove). 
    Then there is a nonempty open substack, say $\cU\subset\overline{\cP}$ having dimension $20-r$ at every point (apply \cite{stacks-project}*{057F} to the identity morphism of an \'{e}tale cover of $\overline{\cP}$ by a scheme). 
    Let $\cY:=\overline{\cP}\setminus\cU$ be its complement. 
    The smooth locus of a morphism is open on the source, so there is a nonempty open substack $\cU'\subset\overline{\cP}$, whose complement $\cY':=\overline{\cP}\setminus\cU'$ is supported on the supersingular locus and over points of $S$ whose residue characteristics divide $d$, such that $\cU'$ is smooth over $S$. 
    The closed substack $\cZ=\cY\cup\cY'$ has the desired properties.
\end{proof}

\begin{corollary}\label{cor:smoothness and dimension over Q}
    If $S$ is defined over $\Q$ then $\overline{\cP}_{(L,\cA)}$ is smooth over $S$ of relative dimension $20-\rk(L)$.
\end{corollary}

\begin{theorem}\label{thm:rank leq 10}
    If $\rk(L)\leq 10$ then $\overline{\cP}_{(L,\cA)}$ is reduced and is flat and lci over $S$ of relative dimension $20-\rk(L)$ at every point.
\end{theorem}

\begin{proof}
    Set $r=\rk(L)$, $\overline{\cP}=\overline{\cP}_{(L,\cA)}$, and $\cM=\cM_{(L,\cA)}$. 
    By base change, we may assume that $S=\spec\Z$. 
    As in the proof of Theorem \ref{thm:smoothness and dimension}, it will further suffice to consider the case when $\cA=\left\{h\right\}$ for a vector $h\in\cC(L)^{\circ}\cap L$. 
    Let $\cW$ be an irreducible component of $\overline{\cP}$. 
    Let $(X,j)$ be an $(L,h)$-polarized K3 surface over an algebraically closed field $k$ of positive characteristic $p$ corresponding to a $k$-point of $\cW$. 
    Applying a result of Lieblich--Olsson \cite{LO15}*{Proposition~A.1} to the saturation of the subgroup $j(L)\subset\Pic(X)$, we deduce that there exists a DVR say $R$ with residue field $k$ and fraction field of characteristic 0 such that $(X,j)$ lifts to a family of $(L,\cA)$-polarized K3 surfaces over $R$. 
    Thus $\cW$ cannot be contained in the fiber modulo $p$, so $\cW$ dominates $\Z$. 
    The primitively polarized locus $\cW\cap\cP$ is dense in $\cW$, so Theorem \ref{thm:smoothness and dimension} shows that $\cW$ is generically smooth over $\Z$ and has dimension $20-r$ generically, hence everywhere. 
    This shows that $\overline{\cP}$ has dimension $20-r$ at every point. 
    Local deformation theory (specifically, the existence of a representation of the form~\eqref{eq:local def theory}) shows that $\cM$ is a local complete intersection at any of its points having dimension $20-r$. 
    It follows that $\overline{\cP}$ is a local complete intersection. 
    We have shown that each irreducible component of $\overline{\cP}$ is generically reduced, so $\overline{\cP}$ is reduced and in particular has no embedded points. 
    Because every irreducible component of $\overline{\cP}$ dominates $\Z$, this implies that $\overline{\cP}$ is flat over $\Z$.
\end{proof}

\begin{remark}
    In the preceding, we deduced results on the closure $\overline{\cP}_{(L,\cA)}$ from properties at generic points. 
    The key fact enabling these sorts of arguments is that the entire stack $\cM_{(L,\cA)}$ has a well-behaved infinitesimal structure everywhere, and in particular even at singular points.
\end{remark}

\begin{remark}
    The conclusion of Theorem \ref{thm:smoothness and dimension} in general fails if we do not remove a closed substack. 
    Indeed, if $L$ has rank 22, then only a supersingular K3 surface could admit an $L$-polarization. 
    Thus, in this case $\overline{\cP}_{(L,\cA)}$ is entirely supported on the supersingular locus. 
    Taking for example $S=\spec\Z$ and $L=\Lambda_{p,10}$ to be the $p$-supersingular K3 lattice of Artin invariant $10$, the moduli stack $\overline{\cP}_{\Lambda_{p,10}/\Z}$ is supported entirely in positive characteristic (in fact, entirely in characteristic $p$), and hence is not flat over $\Z$. 
    In this case the fiber of $\overline{\cP}_{\Lambda_{p,10}/\Z}$ over $\spec\mathbb{F}_p$ is exactly Ogus's moduli space of lattice-polarized supersingular K3 surfaces \cite{Ogus83}. 
    In particular, this has dimension $9$ \cite{Ogus83}*{2.7}, so the dimension formula in \ref{thm:smoothness and dimension} fails as well.
    
    This example also shows that the conclusion of Theorem \ref{thm:rank leq 10} fails in general without the assumption on the rank of $L$. 
    The proof of Theorem~\ref{thm:rank leq 10} suggests that the conclusion could also be made to fail for suitable lattices $L$ with ranks in the range $11\leq\rk(L)\leq 20$, but we do not have an example.
\end{remark}

\begin{remark}\label{rem:ssing locus}
    The supersingular locus in $\cM_{(L,\cA)}$ need not be closed. 
    Roughly speaking, for each prime $p$, the supersingular locus in characteristic $p$ is defined by the vanishing of a system of polynomial equations, but these polynomials depend on $p$.
    
    We give an example. 
    Let $L=\langle h\rangle$ be the rank 1 lattice generated by a vector $h$ satisfying $h^2=4$ and set $\cA=\left\{h\right\}$. 
    Take $S=\spec\Z[1/2]$ and consider the moduli stack $\cM_{(L,h)/S}$. 
    This is the same as the moduli stack of degree 4 K3 surfaces $\cM_4$ over $S$. 
    Furthermore, $\cM_{(L,h)/S}$ is separated over $S$, and we have $\cM_{(L,h)/S}=\cM_{L/S}$.
    Consider the Fermat family of K3 surfaces
    \[
        \mathbb{V}(x^4+y^4+z^4+w^4)\subset\mathbb{P}^3_{\Z[1/2]}
    \]
    over $\Z[1/2]$. 
    This has a canonical $(L,h)$-lattice polarization coming from the defining projective embedding, and so gives rise to a section of the morphism $\cM_{L/\Z[1/2]}\to\spec\Z[1/2]$. 
    Since $\cM_{L/\Z[1/2]}$ is separated, this section is closed. 
    If the supersingular locus were closed, then its intersection with this section would be a closed subset of $\spec\Z[1/2]$. 
    But this is not the case: by~\cite{Shioda}*{Corollary to Proposition~1}, the reduction modulo $p$ of the Fermat quartic is supersingular if and only if $p\equiv 3\pmod{4}$. 
    Thus, the intersection of the section with the supersingular locus is an infinite set of nonzero primes, and so is not closed in $\spec\Z[1/2]$.
\end{remark}

\subsection{Coarse moduli spaces}

By Theorem \ref{thm:rep, easy version}, the stack $\cM_{(L,\cA)}$ has finite inertia, so by a result of Keel--Mori~\cite{KeelMori} (see also~\cite{Rydh}), it admits a coarse moduli space. 
This is an algebraic space over $S$, which we denote by $\rM_{(L,\cA)}$ (or, if we wish to include the base scheme in the notation, by $\rM_{(L,\cA)/S}$). 
The coarse moduli space comes equipped with a morphism
\[
    \pi:\cM_{(L,\cA)}\to\rM_{(L,\cA)}
\]
that is universal for maps from $\cM_{(L,\cA)}$ to algebraic spaces. 
The following properties of the coarse moduli space and the coarse moduli space morphism follow from Theorem \ref{thm:rep, easy version} combined with \cite{Rydh}*{Theorem~6.12}. 
First, as $\cM_{(L,\cA)}$ is locally of finite type, quasi-separated, and Zariski-locally separated over $S$, the same are all true for $\rM_{(L,\cA)}$. 
Next, the morphism $\pi$ is proper, separated, and quasi-finite, and induces a bijection on associated topological spaces
\[
    |\pi|:|\cM_{(L,\cA)}|\to|\rM_{(L,\cA)}|
\]
(universally, in fact). 
Finally, $\pi$ is of formation compatible with flat base change on $\rM_{(L,\cA)}$, and is of formation compatible with arbitrary base change if $\cM_{(L,\cA)}$ is tame (e.g., if $S$ is defined over $\Q$).

Any closed (resp.\ open) substack of $\cM_{(L,\cA)}$ also has finite inertia, and hence also admits a coarse moduli space, which is then a closed (resp.\ open) subspace of $\rM_{(L,\cA)}$ in a natural way. 
In particular, we denote by $\rP_{(L,\cA)}$ the coarse moduli space of $\cP_{(L,\cA)}$ (or including the base scheme, by $\rP_{(L,\cA)/S}$).

\section{Small and Very Small Ample Cones}

In this section, for an even hyperbolic lattice $L$ admitting a primitive embedding into the K3 lattice $\Lambda$, we define subdivisions of the ample cones of $L$, whose chambers we call the \emph{small} and \emph{very small} ample cones of $L$. 
This construction is inspired by the small cones introduced by Alexeev--Engel over the complex numbers \cite{AE25},
which we refer to as \emph{small ample cones} in this paper.

\subsection{Cones: definitions and examples}

Let $\Emb(L)$ denote the set of primitive embeddings $L\hookrightarrow\Lambda$; this set is nonempty, by assumption. 
Given a primitive embedding $e:L\hookrightarrow\Lambda$ we define a subset $\Delta^e(\Lambda)\subset\Delta(\Lambda)$ by
\[
    \Delta^e(\Lambda)=\left\{\delta\in\Lambda\setminus e(L)^{\perp}\,|\,\delta^2=-2\text{ and }\langle e(L),\delta\rangle\text{ is hyperbolic}\right\}.
\]
Here $\langle e(L),\delta\rangle\subset\Lambda$ denotes the sublattice of $\Lambda$ generated by $e(L)$ and $\delta$. 
We note that if $\delta\in\Delta(L)$ is a $(-2)$-class in $L$ then $e(\delta)$ is an element of $\Delta^e(\Lambda)$. 
Thus, if we identify $L$ with its image in $\Lambda$, then we have the containments
\begin{equation}\label{eq:delta containments}
    \Delta(L)\subset\Delta^e(\Lambda)\subset\Delta(\Lambda).
\end{equation}
    
\begin{definition}[Small ample cones]
    \label{def:small amples cones}
    An \emph{$e$-small ample cone} of $L$ is a connected component of the set
    \begin{equation}\label{eq:small ample cones defn}
        \cC(L)\setminus\left(\textstyle\bigcup_{\delta\in\Delta^e(\Lambda)}e^{-1}(\delta^{\perp})\right).
    \end{equation}
\end{definition}

A $(-2)$-class $\delta\in\Delta^e(\Lambda)$ is in an element of $\Lambda$. 
Thus the wall $\delta^{\perp}$ is a subset of $\Lambda_{\R}$ and $e^{-1}(\delta^{\perp})$ is a subset of $L_{\R}$. 
Furthermore, by definition of $\Delta^e(\Lambda)$, the image of $L$ in $\Lambda$ is not contained in the wall $\delta^{\perp}$, and so $e^{-1}(\delta^{\perp})$ has codimension 1 in $L_{\R}$. 
We deduce that the positive cone $\cC(L)$ is not contained in $e^{-1}(\delta^{\perp})$. 
There are only countably many walls, so this shows in particular that the set~\eqref{eq:small ample cones defn} is nonempty, as is any small ample cone of $L$.

\begin{definition}[Very small ample cones]
    \label{def:very small amples cones}
    A \emph{very small ample cone} of $L$ is a connected component of the set
    \begin{equation}\label{eq:vs cone}
        \cC(L)\setminus\bigcup_{e\in\Emb(L)}\left(\textstyle\bigcup_{\delta\in\Delta^e(\Lambda)}e^{-1}(\delta^{\perp})\right).
    \end{equation}
\end{definition}

As in the previous definition, the walls $e^{-1}(\delta^{\perp})$ for $\delta\in\Delta^e(\Lambda)$ have codimension 1 in $L_{\R}$, and there are only countably many such walls. 
Thus, the set \eqref{eq:vs cone} is nonempty, as is any very small ample cone for $L$. 
The set $\Emb(L)$ carries a canonical left action by $\rO(\Lambda)$ (postcomposition) and a canonical right action by $\rO(L)$ (precomposition). 
For a primitive embedding $e\in\Emb(L)$ the union of the bounding walls $\textstyle\bigcup_{\delta\in\Delta^e(\Lambda)}e^{-1}(\delta^{\perp})$ depends only on the $\rO(\Lambda)$-orbit of $e$. 
Thus, we may take the outer union in~\eqref{eq:vs cone} only over a subset $E\subset\Emb(L)$ containing a representative for each $\rO(\Lambda)$-orbit. Moreover, there are only finitely many such orbits. 
Finally, we note that the canonical left action of $\rO(L)$ on $L$ gives rise to an action of $\rO(L)$ on the set of very small ample cones of $L$.

By~\eqref{eq:delta containments}, for any $e\in\Emb(L,\Lambda)$, the walls removed in the $e$-small ample cone decomposition include in particular the walls removed in the ample cone decomposition. 
Thus, the $e$-small ample cone decomposition divides each ample cone into several $e$-small ample cones, and the very small ample cone decomposition divides each small ample cone into several very small ample cones. 
In particular, for a fixed $e\in\Emb(L,\Lambda)$, each very small ample cone of $L$ is contained in a unique $e$-small ample cone of $L$, and each $e$-small ample cone of $L$ is contained in a unique ample cone of $L$.

\begin{example}\label{ex:example1}
    If there is a unique primitive embedding, say $e: L\hookrightarrow\Lambda$ up to isometries of $\Lambda$, then a very small ample cone of $L$ is the same as an $e$-small ample cone of $L$. 
    For example, if $L$ is an even hyperbolic lattice of rank $\leq 10$, then by results of Nikulin \cite{Nikulin}*{1.14.4} there exists a primitive embedding of $L$ into $\Lambda$ which is unique up to $\rO(\Lambda)$.
\end{example}

\begin{example}
    Suppose that $L=\langle h\rangle$ is the rank 1 lattice generated by a single vector $h$ with positive even square, say $h^2=2d$. 
    Then $\Delta(L)$ is empty, so $L_{\R}$ is the unique ample cone for $L$. 
    By Example \ref{ex:example1}, there exists a primitive embedding $e:L\hookrightarrow\Lambda$, which is unique up to $\rO(\Lambda)$. 
    If $\delta\in\Lambda$ is a $(-2)$-class that is not contained in $e(L)^{\perp}$, then the wall $\delta^{\perp}\subset\Lambda_{\R}$ has trivial intersection with $e(L)$. 
    Thus $L_{\R}$ is an $e$-ample cone for $L$ and also a very small ample cone for $L$.
\end{example}

\begin{example}
    Suppose that $L$ has rank 20. 
    Then $L$ has signature $(1,19)$, while the K3 lattice has signature $(3,19)$. 
    It follows that if $e:L\hookrightarrow\Lambda$ is a primitive embedding and $\delta\in\Delta(\Lambda)$ is a $(-2)$-class that is not already contained in $L$, then the sublattice $\langle e(L),\delta\rangle\subset\Lambda$ has signature $(2,19)$, and in particular is not hyperbolic. 
    Therefore, the very small ample cone decomposition for $L$ is the same as the ample cone decomposition. 
    This is related to the geometric fact that the Picard group of a complex K3 surface of Picard rank 20 cannot grow under specializations in characteristic 0.
\end{example}

The following shows that the very small ample cone decomposition has similar good properties to the ample cone decomposition.

\begin{theorem}\label{thm:properties of the very small ample cone decomposition}
    The action of $\rO(L)$ on the set of very small ample cones of $L$ has only finitely many orbits. 
    Furthermore, if $\mathcal{a}$ is a very small ample cone for $L$, then $\mathcal{a}$ is open in $L_{\R}$ and locally rational polyhedral in its closure $\overline{\mathcal{a}}$, and $\mathcal{a}$ intersects the integral lattice $L$ nontrivially.
\end{theorem}

\begin{proof}
    The analogous result for small ample cones is \cite{AE25}*{Proposition~4.13}. The proof in our case follows from an identical argument.
\end{proof}

\begin{lemma}\label{lem:embedding in Lambda}
    If $X$ is a K3 surface over an algebraically closed field $k$, then there exists a primitive embedding $\Pic(X)\hookrightarrow\Lambda$ if and only if $X$ is not supersingular.
\end{lemma}

\begin{proof}
    In the case when $X$ is defined over $\C$, the boundary map induced by the exponential sequence gives a canonical primitive embedding $\Pic(X)\hookrightarrow\rH^2(X,\Z)$, and $\rH^2(X,\Z)$ and $\Lambda$ are abstractly isomorphic as lattices. 
    The case when $k$ is an arbitrary algebraically closed field of characteristic 0 follows from this and an application of the Lefschetz principle.
    
    Suppose that $k$ has positive characteristic $p$. 
    If $X$ is supersingular, then $\Pic(X)$ has rank 22 and signature $(1,21)$. 
    The K3 lattice has rank 22 and signature $(3,19)$, so $\Pic(X)$ cannot be embedded in $\Lambda$. 
    Suppose that $X$ is not supersingular. In this case the first Hodge Chern character~\eqref{eq:first Hodge Chern class} $\Pic(X)\to\rH^1(X,\Omega^1_X)$ is injective \cite{Bragg25}*{Remark~3.11}. 
    Standard deformation theory (as in the proof of Theorem \ref{thm:smoothness and dimension}) then shows that the the deformation functor associated to the pair $(X,\Pic(X)\subset\Pic(X))$ is formally smooth, from which one deduces that there exists a DVR, say $R$, with residue field $k$ and fraction field of characteristic 0 together with a lifting of $X$ over $R$ such that for any geometric generic point $\overline{\eta}$ of $\spec R$ the specialization map $\mathrm{sp}:\Pic(X_{\overline{\eta}})\hookrightarrow\Pic(X)$
    is an isomorphism (see e.g., \cite{LM}*{4.2}). 
    The point $\overline{\eta}$ is the spectrum of an algebraically closed field of characteristic 0, so by the above, there exists a primitive embedding $\Pic(X_{\overline{\eta}})\hookrightarrow\Lambda$.
\end{proof}

\begin{proposition}\label{prop:small cones and polarizations}
    Let $e:L\hookrightarrow\Lambda$ be a primitive embedding. 
    Let $X$ be a non-supersingular K3 surface over an algebraically closed field. 
    Let $j:L\hookrightarrow\Pic(X)$ be a primitive $L$-polarization on $X$, and suppose that there exists a primitive embedding $e':\Pic(X)\hookrightarrow\Lambda$ such that the diagram
    \[
        \begin{tikzcd}
            L\arrow[hook]{rr}{e}\arrow[hook]{dr}[swap]{j}&&\Lambda\\
            &\Pic(X)\arrow[hook]{ur}[swap]{e'}
        \end{tikzcd}
    \]
    commutes. 
    If $\cA\subset L_{\R}$ is an $e$-small ample cone, then the image $j(\cA)\subset\Pic(X)_{\R}$ of $\cA$ under $j$ is contained in some ample cone of $\Pic(X)$. 
    In particular, if $j$ is an $(L,\cA)$-polarization, then $j(\cA)\subset\Amp(X)$.
\end{proposition}

\begin{proof}
    To ease the notation, we suppress the symbols $j,e,$ and $e'$ by viewing them as inclusions. 
    Thus, we have primitive inclusions of lattices 
    \[
        L\subset\Pic(X)\subset\Lambda.
    \]
    Note that $\mathcal{A}$ is contained in $\cC(L)$, and hence is also contained in $\cC(X)$. 
    As $\mathcal{A}$ is connected, to prove the result it will suffice to show that $\mathcal{A}$ does not intersect the wall $\delta^{\perp}$ for any $\delta\in\Delta(X)$. 
    We claim that we have the inclusion $\Delta(X)\subset\Delta^e(\Lambda)$. 
    Using the definition of the small ample cones, it will follow from this that $\mathcal{A}$ avoids $\delta^{\perp}$, as needed. 
    Fix a $(-2)$ class $\delta\in\Delta(X)$. 
    First note that, as the intersection $L\cap\Amp(X)$ is nonempty by assumption, $L_{\R}$ is not entirely contained in the wall $\delta^{\perp}$. 
    Therefore we have $\delta\notin L^{\perp}$. 
    Next, consider the sublattice $\langle L,\delta\rangle\subset\Lambda$ generated by $L$ and $\delta$. 
    We have inclusions
    \[
        L\subset\langle L,\delta\rangle\subset\Pic(X)\subset\Lambda.
    \]
    Because both $L$ and $\Pic(X)$ are hyperbolic, so is $\langle L,\delta\rangle$. 
    Thus we have $\delta\in\Delta^{e}(\Lambda)$, and so $j(\cA)$ is contained in some ample cone of $\Pic(X)$. 
    If $j$ is an $(L,\cA)$-polarization, then $j(\cA)\cap\Amp(X)$ is nonempty, so in this case we have $j(\cA)\subset\Amp(X)$.
\end{proof}

The following is the key consequence of our definition of the very small ample cones.

\begin{proposition}\label{prop:very small cones and polarizations}
    Let $(X,j)$ be a non-supersingular primitively $L$-polarized K3 surface over an algebraically closed field. 
    If $\mathcal{a}\subset L_{\R}$ is a very small ample cone, then the image $j(\mathcal{a})\subset\Pic(X)_{\R}$ of $\mathcal{a}$ under $j$ is contained in some ample cone of $\Pic(X)$. 
    In particular, if $j$ is an $(L,\mathcal{a})$-polarization, then we have $j(\mathcal{a})\subset\Amp(X)$.
\end{proposition}

\begin{proof}
    By Lemma \ref{lem:embedding in Lambda}, there exists a primitive embedding $e':\Pic(X)\hookrightarrow\Lambda$. 
    Write $e$ for the composition $e'\circ j:L\hookrightarrow\Lambda$. 
    Let $\cA\subset L_{\R}$ be the unique $e$-small ample cone of $L$ such that $\mathcal{a}\subset\cA$. 
    By Proposition \ref{prop:small cones and polarizations}, the image of $\cA$ under $j$ is contained in some ample cone of $\Pic(X)$. 
    Thus, the same is true for $\mathcal{a}$. 
    If $j$ is moreover an $(L,\mathcal{a})$-polarization, then $j(\mathcal{a})\cap\Amp(X)$ is nonempty, and thus we have $j(\mathcal{a})\subset\Amp(X)$.
\end{proof}

\begin{corollary}\label{cor:L pol implies L,a pol}
    If $X$ is a non-supersingular K3 surface over an algebraically closed field that admits a primitive $L$-polarization, then $X$ also admits a primitive $(L,\mathcal{a})$-polarization for any very small ample cone $\mathcal{a}\subset L_{\R}$.
\end{corollary}

\begin{proof}
    If $j:L\hookrightarrow\Pic(X)$ is a primitive $L$-polarization and $\mathcal{a}\subset L_{\R}$ is a very small ample cone then by Proposition \ref{prop:very small cones and polarizations} the image of $\mathcal{a}$ under $j$ is contained in some ample cone of $\Pic(X)$. 
    The Weyl group $\pm W(X)$ acts simply transitively on the set of ample cones of $\Pic(X)$, so there exists a unique isometry $g\in\pm W(X)$ such that $g\circ j$ maps $\mathcal{a}$ into $\Amp(X)$. 
    Then $g\circ j$ is a primitive $(L,\mathcal{a})$-polarization on $X$.
\end{proof}

\subsection{Separated moduli spaces}

Fix a very small ample cone $\mathcal{a}\subset L_{\R}$ for $L$ and consider the moduli stack $\overline{\cP}_{(L,\mathcal{a})}:=\overline{\cP}_{(L,\mathcal{a})/S}$. 
For an isometry $g\in\rO(L)$ we have an isomorphism of stacks
\[
    \overline{\cP}_{(L,\mathcal{a})}\to\overline{\cP}_{(L,g^{-1}(\mathcal{a}))},\hspace{1cm}(X,j)\mapsto (X,j\circ g).
\]
Thus up to an isomorphism preserving underlying K3 surfaces the stack $\overline{\cP}_{(L,\mathcal{a})}$ depends only on the $\rO(L)$-orbit of $\mathcal{a}$. 
By Theorem \ref{thm:properties of the very small ample cone decomposition}, there are only finitely many $\rO(L)$-orbits of very small ample cones, hence only finitely many isomorphism classes of moduli spaces $\overline{\cP}_{(L,\mathcal{a})}$.

\begin{remark}
    If $X$ is a non-supersingular K3 surface that admits a primitive $L$-polarization, then $X$ also admits a primitive $(L,\mathcal{a})$-polarization by Corollary \ref{cor:L pol implies L,a pol}. 
    Thus, every non-supersingular primitively $L$-polarizable K3 surface appears in the moduli stack $\cP_{(L,\mathcal{a})}$ (possibly multiple times).
\end{remark}

The following shows that the stack $\overline{\cP}_{(L,\mathcal{a})}$ satisfies a version of the uniqueness part of the valuative criterion, at least away from the supersingular locus.

\begin{proposition}\label{prop:uniqueness part of valuative criterion}
      Let $R$ be a DVR, set $T=\spec R$, let $\eta\in T$ be the generic point, and let $t\in T$ be the closed point. 
      Let $(X,j)$ be a family of primitively $(L,\mathcal{a})$-polarized K3 surfaces over $T$ whose special fiber is non-supersingular. 
      Let $(Y,l)$ be a family of $(L,\mathcal{a})$-polarized K3 surfaces over $T$. 
      If $u_{\eta}:X_{\eta}\simeq Y_{\eta}$ is an isomorphism of $(L,\mathcal{a})$-polarized K3 surfaces over $\eta$, then $u_{\eta}$ extends uniquely to an isomorphism $X\simeq Y$ of $(L,\mathcal{a})$-polarized K3 surfaces over $T$. 
\end{proposition}
\begin{proof}
     Let $\overline{\eta}$ be a geometric generic point of $T$. 
     Fix a geometric closed point $\overline{t}\in T$ as in \ref{pg:specialization map}. 
     It follows from \cite{Ogus83}*{2.8} that there exists an isomorphism $u_{\overline{t}}:X_{\overline{t}}\iso Y_{\overline{t}}$ and an isometry $r\in W(X)$ such that the diagram
     \[
        \begin{tikzcd}
            \Pic(X_{\overline{\eta}})\arrow[hookleftarrow]{dr}\arrow[hook]{rr}{\sp}&&\Pic(X_{\overline{t}})&\\
            &L\arrow[hook]{ur}\arrow[hook]{dr}&&\Pic(X_{\overline{t}})\arrow[ul,"\sim"{sloped},"r"]\\
            \Pic(Y_{\overline{\eta}})\arrow[hook]{rr}{\sp}\arrow[hookleftarrow]{ur}\arrow[uu,"u^*_{\overline{\eta}}","\sim"'{sloped}]&&\Pic(Y_{\overline{t}})\arrow[dashed]{uu}\arrow[ur,"\sim"{sloped},"u_{\overline{t}}^*"{swap}]
        \end{tikzcd}
    \]
    commutes. 
    In particular, $(Y,l)$ is also primitively polarized, and $Y_{\overline{t}}$ is abstractly isomorphic to $X_{\overline{t}}$. 
    Our assumptions on $X$ then imply that all geometric fibers of $X$ and $Y$ are non-supersingular. 
    By Proposition \ref{prop:very small cones and polarizations} the image of $\mathcal{a}$ in $\Pic(X_{\overline{\eta}})_{\R}$ is contained in the ample cone $\Amp(X_{\overline{\eta}})$, and similarly the image of $\mathcal{a}$ in $\Pic(X_{\overline{t}})_{\R}$ is contained in $\Amp(X_{\overline{t}})$. 
    The same is true for the image of $\mathcal{a}$ in the Picard groups of the geometric fibers of $Y$. 
    Choose an integral class $h\in\mathcal{a}\cap L$. 
    Then $j(h)$ is relatively ample on $X$ and $l(h)$ is relatively ample on $Y$. 
    By Matsusaka--Mumford, the isomorphism $u_{\eta}:X_{\eta}\simeq Y_{\eta}$ extends uniquely to an isomorphism $u:X\iso Y$ of families of K3 surfaces over $T$. As $u_{\eta}$ is compatible with the polarizations, $u$ is automatically an isomorphism of families of $(L,\mathcal{a})$-polarized K3 surfaces.    
\end{proof}

\begin{theorem}\label{thm:separated moduli stack}
    There exists a closed substack $\mathcal{Z}\subset\overline{\cP}_{(L,\mathcal{a})}$ supported on the supersingular locus such that the complement $\overline{\cP}_{(L,\mathcal{a})}\setminus\mathcal{Z}$ is separated over $S$.
\end{theorem}

\begin{proof}
    Write $\cP$ for $\cP_{(L,\mathcal{a})}$ and $\overline{\cP}$ for $\overline{\cP}_{(L,\mathcal{a})}$. 
    By Proposition \ref{prop:boundary on ssing locus}, the boundary $\overline{\cP}\setminus\cP$ is supported on the supersingular locus, so it will suffice to prove the result with $\overline{\cP}$ replaced with $\cP$. 
    Consider the map on topological spaces 
    \[
        |\Delta_{\cP}|:|\cP|\to|\cP\times_S\cP|
    \]
    induced by the diagonal morphism $\Delta_{\cP}$ of $\cP$. 
    Because $S$ is a scheme, the natural continuous map $|\cP\times_S\cP|\to|\cP|\times_{|S|}|\cP|$ is bijective (but typically not a homeomorphism). 
    Thus we may identify the underlying set of points of $|\cP\times_S\cP|$ with the fiber product $|\cP|\times_{|S|}|\cP|$. 
    With this identification the map $|\Delta_{\cP}|$ is given by the usual diagonal map $x\mapsto (x,x)$, and in particular $|\Delta_{\cP}|$ is injective.
    
    Let $|\Delta|\subset |\cP\times_S\cP|$ denote the image of $|\Delta_{\cP}|$ and let $\overline{|\Delta|}\subset |\cP\times_S\cP|$ be its closure. 
    Let $\mathcal{Z}\subset\cP$ denote the unique reduced closed substack whose associated topological space $|\mathcal{Z}|$ is the closure of the set $|\pi_1|(\overline{|\Delta|}\setminus|\Delta|)\subset|\cP|$, where $\pi_1:\cP\times_S\cP\to\cP$ is the first projection. 
    Note that $\overline{|\Delta|}\setminus|\Delta|$ is symmetric, so we would obtain the same $\mathcal{Z}$ if we instead used the second projection.
   
   We claim that $\mathcal{Z}\subset\cP$ is contained in the supersingular locus. Consider a point $x\in|\mathcal{Z}|$. 
   If $x$ is not contained in the primitive substack, then by Proposition \ref{prop:boundary on ssing locus}, it is contained in the supersingular locus anyway, so we may assume that $x$ is contained in the primitive substack. 
   By definition of $\cZ$, $x$ is in the closure of $|\pi_1|(\overline{|\Delta|}\setminus|\Delta|)$, which means we can find points $y,z,w\in|\cP|$ with $y\neq z$ and a specialization $(w,w)\rightsquigarrow (y,z)$ of points of $|\cP\times_S\cP|$. 
   Thus, we have specializations
   \[
    \begin{tikzcd}[row sep=tiny]
        &y\arrow[r,rightsquigarrow]&x\\
        w\arrow[ur,rightsquigarrow]\arrow[dr,rightsquigarrow]\\
        &z
    \end{tikzcd}
   \]
   in $|\cP|$. 
   By Lemma \ref{lem:specializations for stacks}, we may realize the specialization 
    $(w,w)\rightsquigarrow (y,z)$ by a morphism from the spectrum of a DVR. 
    The primitive substack is open, hence closed under generization, so $y$ is also in the primitive substack. 
    But $y\neq z$, so Proposition \ref{prop:uniqueness part of valuative criterion} implies that $y$ must be contained in the supersingular locus. 
    The point $z$ has isomorphic underlying K3 surface, hence is also contained in the supersingular locus, and because $y$ specializes to $x$, it follows that $x$ is also in the supersingular locus. 
    We conclude that $\mathcal{Z}$ is contained in the supersingular locus.

   It remains to show that $\cP\setminus\mathcal{Z}$ is separated. 
   Set $\cP':=\cP\setminus\mathcal{Z}$. 
   We will verify the conditions of Proposition \ref{prop:sep criterion} for $\cP'$. 
   By Theorem \ref{thm:rep, easy version}, the stack $\cP$ is locally of finite type, quasi-separated, and Zariski-locally separated over $S$. 
   The same is then also true for the open substack $\cP'$. 
   It remains to show that the map $|\Delta_{\cP'}|:|\cP'|\hookrightarrow|\cP'\times_S\cP'|$ is closed. 
   As this map is quasi-compact, for this it will suffice to show that specializations lift along $|\Delta_{\cP'}|$ \cite{stacks-project}*{01KC}. 
   Let $x,y,z\in|\cP'|$ be points and suppose we have a specialization $(x,x)\rightsquigarrow(y,z)$ in $|\cP'\times_S\cP'|$, as in the diagram
   \[
    \begin{tikzcd}
        \left|\cP'\right|\arrow[hook]{d}\arrow[r,phantom,"\ni"]&x\arrow[maps to]{d}&\\
        \left|\cP'\times_S\cP'\right|\arrow[r,phantom,"\ni"]&(x,x)\arrow[rightsquigarrow]{r}&(y,z)
    \end{tikzcd}
   \]
   Then $(y,z)$ is in the closure of the image of $|\Delta_{\cP}|$. 
   By assumption both $y$ and $z$ are not contained in $|\mathcal{Z}|$. 
   It follows from our definition of $\mathcal{Z}$ that $(y,z)$ is contained in the image of $|\Delta_{\cP}|$. 
   Therefore $y=z$, and $(x,x)\rightsquigarrow(y,z)=(y,y)$ lifts to the specialization $x\rightsquigarrow y$ in $|\cP'|$.
\end{proof}

\begin{corollary}
    \label{cor: separated stack over Z[1/n]}
    There exists an integer $d$ (depending on $L$ and $\mathcal{a}$) such that if $n$ is an integer that is divisible by $d$ then $\overline{\cP}_{(L,\mathcal{a})/\Z[1/n]}$ is separated over $\Z[1/n]$.
\end{corollary}

\begin{proof}
    Let $\mathcal{Z}\subset\overline{\cP}_{(L,\mathcal{a})/\Z}$ be a closed substack as in Theorem \ref{thm:separated moduli stack}. 
    Then $\mathcal{Z}$ is contained in the supersingular locus, and so its image in $\spec\Z$ is a finite set of closed points. 
    We can take $d$ to be the product of their residue characteristics.
\end{proof}

\begin{corollary}
    If the base scheme $S$ is defined over $\Q$ then $\cP_{(L,\mathcal{a})}=\overline{\cP}_{(L,\mathcal{a})}$ is separated over $S$.
\end{corollary}

\section{Lattice-polarized twisted K3 surfaces}

In this section, we consider moduli spaces of lattice-polarized \emph{twisted} K3 surfaces. 
Such spaces have been studied over the complex numbers \cite{Brakkee}*{\S2.2} and in arithmetic settings \cite{Bragg}*{\S 7} by the second and first authors, respectively. 

\subsection{Twisted K3 surfaces} 
Fix a positive integer $n$.

\begin{definition}\label{def:n twisted K3s}
    An \emph{$n$-twisted K3 surface} over an algebraically closed field $k$ is a pair $(X,\alpha)$, where $X$ is a K3 surface over $k$ and $\alpha\in\rH^2(X,\mu_n)$, where the $\rH^2$ denotes cohomology in the flat topology.
\end{definition}

\begin{definition}\label{def:n twisted K3s in families}
     A \emph{family of $n$-twisted K3 surfaces} over a scheme $T$ is a pair $(f:X\to T,\alpha)$, where $X$ is a family of K3 surfaces over $T$ and  $\alpha\in\rH^0(T,\rR^2f_*\mu_n)$ is a global section of the sheaf $\rR^2f_*\mu_n$ on $T$, where $\rR f_*$ denotes the pushforward in the flat topology.
\end{definition}

\begin{remark}
    If $X$ is a K3 surface over a field in which $n$ is invertible, then the flat cohomology group $\rH^2(X,\mu_n)$ is the same as the corresponding \'{e}tale cohomology group. 
    Similarly, if $f:X\to T$ is a family of K3 surfaces and $n$ is invertible everywhere on $T$, then the flat pushforward $\rR^2f_*\mu_n$ is the same as the corresponding \'{e}tale pushforward.
\end{remark}

For a family of K3 surfaces $f:X\to T$ the abutment in the Leray spectral sequence for $\mu_n$ under $\rR f_*$ gives a canonical map
\begin{equation}\label{eq:canonical Leray map}
    \rH^2(X,\mu_n)\to\rH^0(T,\rR^2f_*\mu_n).
\end{equation}
The Leray spectral sequence combined with the identification $f_*\mu_n=\mu_n$ and the vanishing $\rR^1f_*\mu_n=0$ gives an exact sequence
\begin{equation}\label{eq:canonical Leray sequence}
    0\to\rH^2(T,\mu_n)\to\rH^2(X,\mu_n)\to\rH^0(T,\rR^2f_*\mu_n)\to\rH^3(T,\mu_n).
\end{equation}
If $T=\spec k$ is the spectrum of an algebraically closed field, then any fppf cover of $T$ is split, so the higher flat cohomology groups of any sheaf on $T$ vanish. 
Thus~\eqref{eq:canonical Leray map} is an isomorphism, and so Definitions \ref{def:n twisted K3s} and \ref{def:n twisted K3s in families} are compatible. 
Furthermore, if $(f:X\to T,\alpha)$ is a family of $n$-twisted K3 surfaces over a general scheme $T$ then for each geometric point $t\in T$ the fiber $(X_t,\alpha_t)$ is (naturally identified with) an $n$-twisted K3 surface. 
More generally, we have the following.

\begin{lemma}\label{lem:leray is an iso}
    If $T$ is strictly Henselian, then the map~\eqref{eq:canonical Leray map} is an isomorphism.
\end{lemma}

\begin{proof}
    It follows from the Kummer sequence that for any scheme $Y$ classes in the flat cohomology groups $\rH^i(Y,\mu_n)$ for $i\geq 2$ die on an \'{e}tale cover of $Y$. 
    As $T$ is strictly Henselian, any \'{e}tale cover of $T$ is split. 
    Thus $\rH^i(T,\mu_n)=0$ for $i\geq 2$, and the claim follows from the sequence~\eqref{eq:canonical Leray sequence}.
\end{proof}

\begin{proposition}\label{prop:rep thm for flat coho}
    If $f:X\to T$ is a family of K3 surfaces, then the big fppf sheaf $\rR^2f_*\mu_n$ is a group algebraic space that is separated and of finite presentation over $T$.
\end{proposition}

\begin{proof}
    By a result of the first author and Lieblich, the sheaf $\rR^2f_*\mu_n$ is a group algebraic space of finite presentation over $T$ \cite{Bragg}*{Theorem~1.2}. 
    It remains to show that it is separated. We remark that if $n$ is invertible everywhere on $T$, then the smooth and proper base change theorems in \'{e}tale cohomology show that $\rR^2f_*\mu_n$ is a locally constant sheaf of finite groups (indeed, an \'{e}tale form of $(\underline{\Z}_T/n)^{\oplus 22}$), and in particular is separated. 
    To treat the general case when $n$ is not necessarily invertible, we give an ad-hoc proof. We will show that the zero section $T\to\rR^2f_*\mu_n$ is closed, which is equivalent to $\rR^2f_*\mu_n$ being separated \cite{stacks-project}*{047G}. We may assume that $T$ is Noetherian. 
    Any specialization of points in an algebraic space of finite presentation over $T$ can be realized by a morphism from the spectrum of a DVR, and hence (by strict Henselizing) also by a morphism from the spectrum of a regular strictly Henselian local ring. 
    Thus, it will suffice to prove the following claim. 
    Assume that $T$ is the spectrum of a regular strictly Henselian local ring, say with generic (resp.\ closed) point $\eta$ (resp.\ $t$). 
    Then by Lemma \ref{lem:leray is an iso} we have canonical identifications $(\rR^2f_*\mu_n)(T)=\rH^2(X,\mu_n)$ and $(\rR^2f_*\mu_n)(\eta)=\rH^2(X_{\eta},\mu_n)$, and we claim that the restriction map $\rH^2(X,\mu_n)\to\rH^2(X_{\eta},\mu_n)$ is injective. 
    This follows from the fact that the restriction on Picard groups $\Pic(X)\to\Pic(X_{\eta})$ is an isomorphism and the restriction on Brauer groups $\br(X)\hookrightarrow\br(X_{\eta})$ is injective, combined with the exact diagram
    \[
     \begin{tikzcd}
          \Pic(X)\arrow{r}{\cdot n}\arrow[d,sloped,"\sim"]&\Pic(X)\arrow{r}\arrow[d,sloped,"\sim"]&\rH^2(X,\mu_n)\arrow{d}\arrow{r}&\br(X)\arrow[hook]{d}\\
          \Pic(X_{\eta})\arrow{r}{\cdot n}&\Pic(X_{\eta})\arrow{r}&\rH^2(X_{\eta},\mu_n)\arrow{r}&\br(X_{\eta}).
      \end{tikzcd}
      \eqno\qed
    \]
    \hideqed
\end{proof}

\subsection{Moduli of lattice-polarized twisted K3 surfaces}

\begin{definition}
    Fix a subset $\cA\subset L_{\R}$ and a positive integer $n$. 
    The \emph{moduli stack of $(L,\cA)$-polarized $n$-twisted K3 surfaces over $S$} is the stack $\cM^{[n]}_{(L,\cA)/S}$ whose objects over an $S$-scheme $T$ are triples $(f:X\to T,j,\alpha)$ where $(X,j)$ is a family of $(L,\cA)$-polarized K3 surfaces over $T$ and $(f:X\to T,\alpha)$ is a family of $n$-twisted K3 surfaces over $T$, and whose morphisms $(f:X\to T,\alpha)\to (f':X'\to T',\alpha')$ lying over a morphism of schemes $v:T\to T'$ is a morphism $u:X\to X'$ of families of K3 surfaces (in the sense of Definition \ref{def:moduli stack of lattice-pol K3s}) such that the pullback map
     \[
        u^*:\rH^0(T',\rR^2f'_*\mu_n)\to\rH^0(T,\rR^2f_*\mu_n)
     \]
     takes $\alpha'$ to $\alpha$.
\end{definition}

When $n=1$ we have a natural identification $\cM_{(L,\cA)}^{[1]}=\cM_{(L,\cA)}$. 
There is a forgetful morphism
\begin{equation}\label{eq:forget the twist}
    \cM_{(L,\cA)}^{[n]}\to\cM_{(L,\cA)},\hspace{1cm}(X,j,\alpha)\mapsto (X,j),
\end{equation}
and $\cM_{(L,\cA)}^{[n]}$ has a natural structure of a group over $\cM_{(L,\cA)}$. 
If $n$ is invertible everywhere on $S$, then this map is an \'{e}tale local system with fiber $(\Z/n)^{\oplus 22}$, and in particular is finite \'{e}tale. 
If $n$ is not invertible on $S$, then this map need not be flat, and even has positive-dimensional fibers over the supersingular locus.

\begin{theorem}\label{thm:rep, easy version, for twisted K3s}
    For any subset $\cA\subset L_{\R}$ and positive integer $n$ (possibly not invertible in $\cO_S$), the stack $\cM_{(L,\cA)}^{[n]}$ is algebraic, is Deligne--Mumford, locally of finite type, quasi-separated, and Zariski-locally separated over $S$, and has finite inertia.
\end{theorem}

\begin{proof}
    Let $(f:X\to T,j,\alpha)$ be a family of $(L,\cA)$-polarized $n$-twisted K3 surfaces over an $S$-scheme $T$, let $\mu:T\to\cM_{(L,\cA)}^{[n]}$ be the corresponding morphism, and let $\cF_T$ denote the fiber product in the Cartesian diagram
    \[
        \begin{tikzcd}
            \cF_T\arrow{d}\arrow{r}&\cM_{(L,\cA)}^{[n]}\arrow{d}\\
            T\arrow{r}{\mu}&\cM_{(L,\cA)}.
        \end{tikzcd}
    \]
    Then $\cF_T$ is naturally identified with the big fppf sheaf $\rR^2f_*\mu_n$. 
    The result follows by combining Theorem \ref{thm:rep, easy version} and Proposition \ref{prop:rep thm for flat coho}.
\end{proof}

For a positive integer $d$, raising to the $d$th power defines a map of group schemes $\mu_{nd}\to\mu_n$. 
This gives rise to a morphism of stacks
\begin{equation}\label{eq:multiplication by d map}
    \cM_{(L,\cA)}^{[nd]}\to\cM_{(L,\cA)}^{[n]},\hspace{1cm}(X,j,\alpha)\mapsto (X,j,\alpha^d).
\end{equation}
When $d=1$, this map is an isomorphism, and when $d=n$, it recovers the forgetful morphism. 
We now define a twisted analog of the primitively polarized substack.

\begin{definition}
    We let $\cM^{(n)}_{(L,\cA)}\subset\cM_{(L,\cA)}^{[n]}$ denote the full substack whose objects over an $S$-scheme $T$ are those tuples $(f:X\to T,j,\alpha)$ such that for each geometric point $t\in T$ the image $\alpha_t\in\rH^2(X_t,\mu_n)$ of $\alpha$ in the fiber over $t$ has exact order $n$.
\end{definition}

\begin{proposition}\label{prop:primitive twisted substack is open}
    The inclusion $\cM^{(n)}_{(L,\cA)}\hookrightarrow\cM_{(L,\cA)}^{[n]}$ is an open immersion.
\end{proposition}

\begin{proof}
    Let $(X,j,\alpha)$ be a family of $(L,\cA)$-polarized $n$-twisted K3 surfaces over an $S$-scheme $T$. 
    Let $U\subset T$ be the locus of points where $\alpha$ has exact order $n$. 
    To show the result, we need to show that $U$ is open in $T$. 
    By Proposition \ref{prop:rep thm for flat coho}, the group $\rR^2f_*\mu_n$ is separated over $T$, so the vanishing locus of any of its global sections is closed, or equivalently, the nonvanishing locus of any global section is open. 
    Now note that $U$ is the union of the nonvanishing loci of the sections $d\alpha\in\rH^0(T,\rR^2f_*\mu_n)$ as $d$ ranges over the positive divisors of $n$ such that $d<n$. 
    Thus $U$ is open in $T$.
\end{proof}

\subsection{Primitive lattice polarizations}

\begin{definition}
    Let $\cP_{(L,\cA)}^{[n]}\subset\cM_{(L,\cA)}^{[n]}$ denote the full substack of objects whose underlying polarization is primitive. 
    Let $\cP^{(n)}_{(L,\cA)}:=\cM^{(n)}_{(L,\cA)}\cap\cP_{(L,\cA)}^{[n]}$ be the substack of objects whose polarization is primitive and whose twisting has exact order $n$ everywhere. 
    By Propositions \ref{prop:primitive substack is open} and \ref{prop:primitive twisted substack is open} the stack $\cP^{(n)}_{(L,\cA)}$ is open in $\cM_{(L,\cA)}^{[n]}$. 
    We let $\overline{\cP}^{(n)}_{(L,\cA)}$ denote the closure of $\cP^{(n)}_{(L,\cA)}$ in $\cM_{(L,\cA)}^{[n]}$.
\end{definition}

\begin{remark}
    The forgetful map~\eqref{eq:forget the twist} restricts to a morphism $\overline{\cP}^{(n)}_{(L,\cA)}\to\overline{\cP}_{(L,\cA)}$.
\end{remark}

\begin{remark}
    Even in characteristic 0, the stack $\overline{\cP}^{(n)}_{(L,\cA)}$ will typically include some twisted K3 surfaces whose twisting has order strictly less than $n$.
\end{remark}

We show in the following that the stacks $\overline{\cP}^{(n)}_{(L,\cA)}$ for $n\geq 1$ have similar properties as in the untwisted case. 
The proofs are very similar, so we shall be brief.

\begin{theorem}\label{thm:properties of the moduli stacks, twisted version}
    Let $\cA\subset L_{\R}$ be a subset, let $\mathcal{a}\subset L_{\R}$ be a very small ample cone, and let $n$ be a positive integer.
    \begin{enumerate}[leftmargin=*]
        \item\label{item:twisted1} There exists a closed substack $\cZ\subset\overline{\cP}^{(n)}_{(L,\cA)}$ that is both supported on the supersingular locus and supported over the union of the characteristic $p$ loci $S_p$ of $S$ for primes $p$ dividing both $n$ and the discriminant of $L$, such that $\overline{\cP}^{(n)}_{(L,\cA)}\setminus\cZ$ is smooth over $S$ of relative dimension $20-\rk(L)$ at every point. 
        \smallskip
        
        \item\label{item:twisted2} If $\rk(L)\leq 9$ then $\overline{\cP}^{(n)}_{(L,\cA)}$ is reduced and is flat and lci over $S$ of relative dimension $20-\rk(L)$ at every point.
        \smallskip
        
        \item\label{item:twisted3} There exists a closed substack $\cZ\subset\overline{\cP}^{(n)}_{(L,\mathcal{a})}$ supported on the supersingular locus such that $\overline{\cP}^{(n)}_{(L,\cA)}\setminus\cZ$ is separated over $S$.
    \end{enumerate}
\end{theorem}

\begin{proof}
    \eqref{item:twisted1}: This follows from an identical argument as in the untwisted case (Theorem \ref{thm:smoothness and dimension}), using the deformation theory of twisted K3 surfaces as in \cite{Bragg} (the case when $\rk(L)=1$ is \cite{Bragg}*{Theorem~7.4}).

    \noindent\eqref{item:twisted2}: The case when $\rk(L)=1$ is \cite{Bragg}*{Proposition~7.9}. 
    The general case is an immediate generalization of this argument, as in the proof of Theorem \ref{thm:rank leq 10} in the untwisted case.

    \noindent\eqref{item:twisted3}: By Proposition \ref{prop:rep thm for flat coho} the forgetful map $\cM^{[n]}_{(L,\mathcal{a})}\to\cM_{(L,\mathcal{a})}$ is separated. 
    Thus the forgetful map $\overline{\cP}^{(n)}_{(L,\mathcal{a})}\to\overline{\cP}_{(L,\mathcal{a})}$ is also separated. 
    The result follows from this, combined with Theorem \ref{thm:separated moduli stack}.
\end{proof}

\begin{remark}
    The above results are generally false for other irreducible components of the stack $\cM^{[n]}_{(L,\cA)}$ (see e.g. \cite{Bragg}*{Corollary~7.5}).
\end{remark}

\subsection{From $n$-twists to Brauer classes}\label{subsec_twiststoBrauer}

An $n$-twisted K3 surface $(X,\alpha)$ has an associated Brauer class $\tau(\alpha)\in\br(X)$, where
\[
    \tau:\rH^2(X,\mu_n)\to\br(X)
\]
is the map on $\rH^2$ induced by the inclusion $\mu_n\hookrightarrow\G_m$. This map fits into the long exact sequence
\[
    0\to\Pic(X)\xrightarrow{\cdot n}\Pic(X)\xrightarrow{\partial_n}\rH^2(X,\mu_n)\xrightarrow{\tau}\br(X)\xrightarrow{\cdot n}\br(X)\to\ldots
\]
coming from taking cohomology of the Kummer sequence on $X$. 
This sequence shows that $\tau$ induces an isomorphism
\[
    \rH^2(X,\mu_n)/\partial_n(\Pic(X))\simeq\br(X)[n].
\]
In this subsection, we consider a certain quotient of the stack $\cP_{(L,\cA)}^{[n]}$, whose general points correspond to lattice polarized K3 surfaces equipped with an $n$-torsion Brauer class.

There is a morphism of stacks
\begin{equation}\label{eq:a section of the forgetful map}
    \sigma:\cM_{(L,\cA)}\times (L\otimes\Z/n)\to\cM_{(L,\cA)}^{[n]}
\end{equation}
which on objects over a scheme $T$ sends a pair $((X,j),v)$, where $(X,j)$ is a family of $(L,\cA)$-polarized K3 surfaces over $T$ and $v$ is a locally constant family of elements of $L$ over $T$, to the triple $(X,j,\partial_n(j(v)))$, where $\partial_n:\Pic_{X/T}\to\rR^2f_*\mu_n$ is the boundary map in the long exact sequence
\begin{equation}\label{eq:LES from Kummer}
    0\to\Pic_{X/T}\xrightarrow{\cdot n}\Pic_{X/T}\xrightarrow{\partial_n}\rR^2f_*\mu_n\xrightarrow{\tau}\rR^2f_*\G_m\xrightarrow{\cdot n}\rR^2f_*\G_m\to\ldots
\end{equation}
obtained by applying $\rR f_*$ to the Kummer sequence for $\mu_n$ on $X$. 
The morphism $\sigma$ is a homomorphism of group spaces over $\cM_{(L,\cA)}$, and is injective over the primitive substack $\cP_{(L,\cA)}$.
    
\begin{definition}\label{def_stackswithBrauerclasses}
    We write $\cQ_{(L,\cA)}^{[n]}$ for the quotient of $\cP_{(L,\cA)}^{[n]}$ by the action of $\cP_{(L,\cA)}\times(L\otimes\Z/n)$. 
    Explicitly, an object of the stack $\cQ_{(L,\cA)}^{[n]}$ over an $S$-scheme $T$ is a tuple $(f:X\to T,j,\alpha)$, where $(f:X\to T,j)$ is a family of primitively $(L,\cA)$-polarized K3 surfaces over $T$ and $\alpha$ is a section over $T$ of the quotient sheaf
    \[
        \mathrm{coker}\left(\underline{L}_T\xrightarrow{\partial_n\circ j}\rR^2f_*\mu_n\right). 
    \]
    We write $\cQ_{(L,\cA)}^{(n)}\subset\cQ_{(L,\cA)}^{[n]}$ for the full substack whose objects over an $S$-scheme $T$ are tuples $(f:X\to T,j,\alpha)$ such that for each geometric point $t\in T$ the fiber $\alpha_t$ has exact order $n$ as an element of the quotient $\rH^2(X,\mu_n)/\partial_n(j(L))$.
\end{definition}

We have a short exact sequence of groups over $\cP_{(L,\cA)}$
\[
    0\to\cP_{(L,\cA)}\times (L\otimes\Z/n)\to\cP_{(L,\cA)}^{[n]}\to\cQ_{(L,\cA)}^{[n]}\to 0.
\]
This shows that $\cQ_{(L,\cA)}^{[n]}$ is the quotient of $\cP_{(L,\cA)}^{[n]}$ by a finite \'{e}tale equivalence relation. 
It follows that $\cQ_{(L,\cA)}^{[n]}$ is algebraic, and has the properties listed in Theorem \ref{thm:rep, easy version, for twisted K3s}. 
Furthermore, arguing as in the proof of Proposition \ref{prop:primitive twisted substack is open}, it follows that the inclusion $\cQ^{(n)}_{(L,\cA)}\hookrightarrow\cQ_{(L,\cA)}^{[n]}$ is an open immersion.

We are interested in the quotient $\cQ_{(L,\cA)}^{[n]}$ because it is more closely related to Brauer groups of K3 surfaces. 
Let $(X,j)$ be a primitively $(L,\cA)$-polarized K3 surface over an algebraically closed field $k$ and let $x\in\cP_{(L,\cA)}(k)$ be the corresponding $k$-point. 
If $x$ is general, then the marking $j:L\hookrightarrow\Pic(X)$ is an isomorphism, and $\tau$ induces an isomorphism 
    \[
        \rH^2(X,\mu_n)/\partial_n(j(L))\iso\br(X)[n].
    \]
Thus, the fiber of the morphism $\cQ_{(L,\cA)}^{[n]}\to\cP_{(L,\cA)}$ over $x$ is in bijection with the $n$-torsion in the Brauer group of $X$, and the fiber of the morphism $\cQ_{(L,\cA)}^{(n)}\to\cP_{(L,\cA)}$ over $x$ is in bijection with the elements in $\br(X)$ of exact order $n$.

\section{Complex moduli spaces and period morphisms}\label{sec:complex moduli spaces}

In this section, we consider moduli spaces of lattice-polarized K3 surfaces over the complex numbers and their associated period morphisms. 
As before, we fix an even hyperbolic lattice $L$ that can be embedded in the K3 lattice $\Lambda$ and a subset $\cA\subset L_{\R}$.

\subsection{Period morphisms for lattice-polarized K3 surfaces}

\begin{definition}
    Let $e:L\hookrightarrow\Lambda$ be a primitive embedding of $L$ into the K3 lattice. 
    We let $\widetilde{\cP}_{(L,\cA)}^{e}$ denote the stack on the category of $\C$-schemes whose objects over a $\C$-scheme $T$ are triples $(f:X\to T,j,l)$, where $X$ is a family of K3 surfaces over $T$, $j:\underline{L}_T\hookrightarrow\pic_{X/T}$ is an $(L,\cA)$-polarization, and $l:\mathrm{R}^2f_*\Z\iso \underline{\Lambda}_T$ is a marking (an isomorphism compatible with the pairings) that is compatible with $j$ in the sense that the diagram
    \begin{equation}\label{eq:marking and polarization square}
        \begin{tikzcd}
            \underline{L}_T\arrow[hook]{d}[swap]{e}\arrow[hook]{r}{j}&\Pic_{X/T}\arrow[hook]{d}{\partial}\\
            \underline{\Lambda}_T&\mathrm{R}^2f_*\Z\arrow[l,"\sim","l"']
        \end{tikzcd}
    \end{equation}
    commutes. 
    Here, $\rR f_*$ denotes the pushforward in the classical (Euclidean) topology, and $\partial$ is the boundary map coming from the exponential exact sequence. 
    We note that because $e$ is assumed to be primitive, the polarization $j$ is automatically primitive as well. 
    A morphism between two such triples is an isomorphism of families of K3 surfaces that is compatible with both the polarizations and the markings.

    In the special case when $\cA=L_{\R}$ we write $\widetilde{\cP}^e_{L}:=\widetilde{\cP}^e_{(L,L_{\R})}$.
\end{definition}

A $\C$-point of the stack $\widetilde{\cP}_{(L,\cA)}^{e}$ is a triple $(X,j,l)$, where $(X,j)$ is an $(L,\cA)$-polarized K3 surface over $\C$ and $l:\rH^2(X,\Z)\iso\Lambda$ is a marking such that the diagram
    \begin{equation}\label{eq:marking diagram}
        \begin{tikzcd}
            L\arrow[hook]{d}[swap]{e}\arrow[hook]{r}{j}&\Pic(X)\arrow[hook]{d}{\partial}\\
            \Lambda&\rH^2(X,\Z)\arrow[l,"\sim","l"']
        \end{tikzcd}
    \end{equation}
commutes. 
The automorphism group of a K3 surface acts faithfully on its second singular cohomology \cite{Huy16}*{\S15, 2.1}, so the objects of $\widetilde{\cP}_{(L,\cA)}^{e}$ have only trivial automorphisms. 
Thus, the stack $\widetilde{\cP}_{(L,\cA)}^{e}$ is equivalent to its associated functor of isomorphism classes of objects, which we denote by $\widetilde{\rP}_{(L,\cA)}^e$.

We compare the above to our moduli stacks of lattice-polarized K3 surfaces. 
Write $\cP_{(L,\cA)}$ for $\cP_{(L,\cA)/\C}$. 
If $X$ is a K3 surface over $\C$ and $l:\rH^2(X,\Z)\iso\Lambda$ is a marking, then there is a unique primitive embedding $e:L\hookrightarrow\Lambda$ such that the diagram~\eqref{eq:marking diagram} commutes (namely $e=l\circ\partial\circ j$). 
Furthermore, the orbit of $e$ under the $\rO(\Lambda)$-action on $\Emb(L,\Lambda)$ does not depend on $l$.

\begin{definition}
    Let $\epsilon\subset\Emb(L,\Lambda)$ be an $\rO(\Lambda)$-orbit of primitive embeddings $L\hookrightarrow\Lambda$. 
    We say that a primitively $L$-polarized K3 surface $(X,j)$ over $\C$ has \emph{type $\epsilon$} if for some (equivalently, for any) marking $l:\rH^2(X,\Z)\iso\Lambda$ the embedding $l\circ\partial\circ j:L\hookrightarrow\Lambda$ is in the orbit $\epsilon$. 
    We write
    \[
        \cP_{(L,\cA)}^{\epsilon}\subset\cP_{(L,\cA)}
    \]
    for the substack whose objects over a $\C$-scheme $T$ are primitively $(L,\cA)$-polarized K3 surfaces over $T$ that are of type $\epsilon$ on every geometric fiber.
\end{definition}

Fix a $\C$-point $x$ of $\cP_{(L,\cA)}$ corresponding to a triple $(X,j,l)$. 
The canonical monodromy action on $\rH^2(X,\Z)$ acts via orthogonal transformations that fix the image of $L$ pointwise. 
It follows that the type of a primitively $L$-polarized K3 surface is invariant in families over connected bases. 
We therefore have a disjoint union decomposition of $\cP_{(L,\cA)}$ into open and closed substacks
\[
    \cP_{(L,\cA)}=\bigsqcup_{\epsilon\in\Emb(L,\Lambda)/\rO(\Lambda)}\cP_{(L,\cA)}^{\epsilon}.
\]
Let $e\in\Emb(L,\Lambda)$ be a primitive embedding and let $\epsilon\subset\Emb(L,\Lambda)$ be its $\rO(\Lambda)$-orbit. 
Dropping the marking gives rise to a forgetful morphism of stacks
\begin{equation}\label{eq:forgetful morphism}
    \varphi^e_{(L,\cA)}:\widetilde{\rP}_{(L,\cA)}^{e}\to\cP_{(L,\cA)}^{\epsilon},\hspace{1cm}(X,j,l)\mapsto (X,j).
\end{equation}
If confusion is unlikely, we drop the decorations and write $\varphi$ for $\varphi^e_{(L,\cA)}$. 
Let $T$ be a $\C$-scheme and let $(f:X\to T,j,l)$ be an object of $\widetilde{\rP}_{(L,\cA)}^{e}$ over $T$. 
For an isometry $g\in\rO(\Lambda)$ we have a commutative diagram
\[
    \begin{tikzcd}
        &\underline{L}_T\arrow[hook]{d}[swap]{e}\arrow[hook]{r}{j}\arrow[bend right=25]{ddl}[swap]{g\circ e}&\Pic_{X/T}\arrow[hook]{d}\\
        &\underline{\Lambda}_T\arrow{dl}{g}&\mathrm{R}^2f_*\Z\arrow[l,"\sim","l"']\arrow[bend left=25]{lld}{g\circ l}\\
        \underline{\Lambda}_T&&
    \end{tikzcd}
\]
The triple $(f:X\to T,j,g\circ l)$ is an object of $\widetilde{\rP}_{(L,\cA)}^{g\circ e}$ over $T$. 
The formula
\[
    (X,j,l)\cdot g=(X,j,g\circ l)
\]
defines a left action of the orthogonal group $\rO(\Lambda)$ on the disjoint union of the $\widetilde{\rP}_{(L,\cA)}^{e}$ for $e\in\Emb(L,\Lambda)$.

Fix a primitive embedding $e:L\hookrightarrow\Lambda$ and write
\[
    \Pi:=e(L)^{\perp}\subset\Lambda
\]
for the orthogonal complement of $e(L)$ in $\Lambda$. 
Up to isometry, $\Pi$ 
depends only on the $\rO(\Lambda)$-orbit $\epsilon$ of~$e$. 
When this is important, we will write $\Pi_{\epsilon}$.
The subgroup $\widetilde{\rO}(\Pi)\subset\rO(\Pi)$ of isometries of $\Pi$ that act trivially on the discriminant group is naturally identified with the subgroup $\rO_e(\Lambda)\subset\rO(\Lambda)$ of isometries of $\Lambda$ that fix $e(L)$ pointwise. 
The $\rO(\Lambda)$-action restricts to a left action of $\rO_e(\Lambda)\simeq\widetilde{\rO}(\Pi)$ on $\widetilde{\rP}^e_{(L,\cA)}$. 
The forgetful morphism~\eqref{eq:forgetful morphism} is equivariant for this action, and so descends to a morphism
\begin{equation}\label{eq:varphi bar}
    \left[\widetilde{\rP}^e_{(L,\cA)}/\widetilde{\rO}(\Pi)\right]\to\cP_{(L,\cA)}^{\epsilon}.
\end{equation}
As $\widetilde{\rO}(\Pi)$ is exactly the group of changes of markings, this map is an isomorphism.

\begin{corollary}
    For each $e\in\Emb(L,\Lambda)$, $\widetilde{\rP}^e_{(L,\cA)}$ is an algebraic space that is locally of finite type, quasi-separated, and Zariski-locally separated over $\C$.
\end{corollary}

\begin{proof}
    The isomorphism~\eqref{eq:varphi bar} exhibits the moduli space $\widetilde{\rP}^e_{(L,\cA)}$ as an $\widetilde{\rO}(\Pi)$-torsor over $\cP^{\epsilon}_{(L,\cA)}$. 
    The result follows from Theorem \ref{thm:rep, easy version}.
\end{proof}

We now discuss period morphisms and Torelli theorems for lattice-polarized K3 surfaces. 
The \emph{(unpolarized) period domain} is the complex manifold
\[
    \cD(\Lambda):=\left\{[x]\in\mathbb{P}(\Lambda_{\C})\,|\,x^2=0\text{ and }x.\overline{x}>0\right\}.
\]
Here $\Lambda_{\C}:=\Lambda\otimes_{\Z}\C$ and $\mathbb{P}(\Lambda_{\C})$ is the complex projective space classifying lines in $\Lambda_{\C}$. 
The period domain is an analytic open subset of the quadric hypersurface in $\mathbb{P}(\Lambda_\C)$ cut out by the quadratic form on $\Lambda_{\C}$, and hence has dimension $20$. 
Let $T$ be a $\C$-scheme and let $f:X\to T$ be a family of K3 surfaces over $T$ equipped with a marking $l:\rR^2f_*\Z\iso\underline{\Lambda}_T$. 
Then $l$ induces an isomorphism
\[
    l\otimes\C:\mathrm{R}^2f_*\Omega^{\bullet}_{X/T}\simeq(\mathrm{R}^2f_*\Z)\otimes\C\iso\underline{\Lambda}_T\otimes\C.
\]
The subsheaf $f_*\Omega^2_{X/T}\subset\mathrm{R}^2f_*\Omega^{\bullet}_{X/T}$ is locally free of rank 1, as is its image $(l\otimes\C)(f_*\Omega^2_{X/T})\subset \underline{\Lambda}_T\otimes\C$. 
Thus we obtain a morphism $T\to\mathbb{P}(\Lambda_\C)$ whose image is contained in $\cD(\Lambda)$.

There is a natural variant of this construction for marked $L$-polarized K3 surfaces. 
The \emph{$L$-polarized period domain} is the complex submanifold
\[
    \cD(\Pi):=\left\{[x]\in\mathbb{P}(\Pi_{\C})\,|\,x^2=0\text{ and }x.\overline{x}>0\right\}\subset\cD(\Lambda).
\]
It is an analytic open subset of a quadric hypersurface in the projective space $\mathbb{P}(\Pi_\C)$, and hence has dimension $\rk(\Pi)-2=20-\rk(L)$. 
Let $\cA\subset L_{\R}$ be a subset and let $(f:X\to T,j,l)$ be an object of $\widetilde{\rP}^e_{(L,\cA)}$ over $T$. 
As above the pair $(f:X\to T,l)$ gives rise to a morphism $\rho(X,l):T\to\mathbb{P}(\Lambda_\C)$. 
The subsheaf $f_*\Omega^2_{X/T}$ is orthogonal to the image of $\Pic_{X/T}$, so in fact $\rho(X,l)$ factors through $\cD(\Pi)$. 
The \emph{$(L,\cA)$-polarized period morphism} is the resulting map
\begin{equation}\label{eq:period morphism, L polarized}
    \rho^e_{(L,\cA)}:\widetilde{\rP}^e_{(L,\cA)}\to\cD(\Pi).
\end{equation}
We set $\rho^e_L=\rho^e_{(L,L_{\R})}$, and if confusion is unlikely, drop the decorations entirely and write $\rho$.

\newpage

\begin{remark}
In the above, the subset $\cA\subset L_{\R}$ was arbitrary, and played no particular role. 
Varying $\cA$ produces a variety of compatible period morphisms. 
In more detail, let $h\in\cC(L)^{\circ}\cap L$ and let $\cA\subset\cB\subset L_{\R}$ be subsets such that $h\in\cA$. 
We then have a commuting diagram
\[
    \begin{tikzcd}[column sep=small]
        \widetilde{\rP}^e_{(L,h)}\arrow[hook]{r}\arrow[bend right=15]{drrrrrr}&\cdots\arrow[hook]{r}&\widetilde{\rP}^e_{(L,\cA)}\arrow[hook]{r}\arrow[bend right=10]{drrrr}&\cdots\arrow[hook]{r}&\widetilde{\rP}^e_{(L,\cB)}\arrow[hook]{r}\arrow{drr}&\cdots\arrow[hook]{r}&\widetilde{\rP}^e_L\arrow{d}{\rho^e_L}\\
        &&&&&&\cD(\Pi)
    \end{tikzcd}
\]
where the downward-trending arrows are the appropriate period morphisms. 
Thus, for a subset $\cA\subset L_{\R}$, the $(L,\cA)$-polarized period morphism is simply the restriction of the $L$-polarized period morphism $\rho^e_L$ to an open substack.
\end{remark}

For a $(-2)$-class $\delta\in\Delta(\Pi)$ write $H_{\delta}\subset\cD(\Pi)$ for the projectivization of the perpendicular to $\delta$ in $\Pi_{\C}$. 
The pairing on $\Pi$ is nondegenerate, so each $H_{\delta}$ is a proper subset of $\cD(\Pi)$, of codimension 1 at every point. 
We put
\[
    \cD(\Pi)^{\circ}:=\cD(\Pi)\setminus\left(\bigcup_{\delta\in\Delta(\Pi)}H_{\delta}\right).
\]

\begin{theorem}\label{thm:torelli, marked version}
    Let $e\in\Emb(L,\Lambda)$ be a primitive embedding. 
    If $\cA\subset L_{\R}$ is a nonempty subset that is contained in an $e$-small ample cone (e.g., a very small ample cone), then the period morphism $\rho^e_{(L,\cA)}$~\eqref{eq:period morphism, L polarized} induces an isomorphism
    \[
        \rho^e_{(L,\cA)}:\widetilde{\rP}^e_{(L,\cA)}\to\cD(\Pi)^{\circ}.
    \]
\end{theorem}

\begin{proof}
    It follows from Proposition \ref{prop:small cones and polarizations} that if $\cA$ is a nonempty subset of an $e$-small ample cone say $\cA'$, then $\iota_{(\cA,\cA')}$ defines an isomorphism of stacks $\widetilde{\rP}^e_{(L,\cA)}\simeq\widetilde{\rP}^e_{(L,\cA')}$. 
    Thus, we may assume that $\cA$ is itself an $e$-small ample cone. This follows from \cite{AE25}*{Theorem~4.7}.
\end{proof}

We now describe how to obtain a period morphism for lattice-polarized K3 surfaces without a marking by the K3 lattice. 
The group $\widetilde{\rO}(\Pi)$ acts on the moduli space $\widetilde{\rP}^e_{(L,\cA)}$ as described above. 
It also acts naturally on the $L$-polarized period domain $\cD(\Pi)$ and on the open subset $\cD(\Pi)^{\circ}$. 
The period morphism is equivariant with respect to these actions. 
Incorporating the isomorphism~\eqref{eq:varphi bar} we obtain a morphism
\begin{equation}\label{eq:bar period morphism}
    \overline{\rho}=\overline{\rho}^e_{(L,\cA)}:\cP^{\epsilon}_{(L,\cA)}\to\left[\cD(\Pi)/\widetilde{\rO}(\Pi)\right],
\end{equation}
where $\epsilon\subset\Emb(L,\Lambda)$ is the $\rO(\Lambda)$-orbit of $e$. 
Explicitly, if $(f:X\to T,j)$ is a family of primitively $(L,\cA)$-polarized K3 surfaces of type $\epsilon$ over a $\C$-scheme $T$, then the morphism $T\to\left[\cD(\Pi)/\widetilde{\rO}(\Pi)\right]$ resulting from $\overline{\rho}^e_{(L,\cA)}$ is obtained as follows. 
Let
\[
    \widetilde{T}:=\mathcal{Isom}'_T(\rR^2f_*\Z,\underline{\Lambda}_T)
\]
denote the $T$-sheaf of marking that are compatible with $j$ and $e$ (in the sense that the diagram~\eqref{eq:marking and polarization square} commutes). 
This is a left $\rO_e(\Lambda)\simeq\widetilde{\rO}(\Pi)$-torsor over $T$. 
Let $\widetilde{X}=X\times_T\widetilde{T}$ be the pullback of $X$ to $\widetilde{T}$ and let $\widetilde{f}:\widetilde{X}\to\widetilde{T}$ be the base change of $f$. 
Then $\widetilde{X}$ is equipped with a canonical marking say $l:\rR^2\widetilde{f}_*\Z\iso\underline{\Lambda}_{\widetilde{T}}$. 
Taking the periods of $\widetilde{X}$ gives a morphism $\widetilde{T}\to\cD(\Pi)$, which fits into a Cartesian diagram
\[
    \begin{tikzcd}
        \widetilde{T}\arrow{d}\arrow{r}&\cD(\Pi)\arrow{d}\\
        T\arrow{r}&\left[\cD(\Pi)/\widetilde{\rO}(\Pi)\right]
    \end{tikzcd}
\]
where both vertical arrows are the quotient maps for the respective $\widetilde{\rO}(\Pi)$-actions. The following is essentially the same as \cite{AE25}*{Corollary~4.20}.

\begin{theorem}\label{thm:torelli, unmarked version}
    Let $e\in\Emb(L,\Lambda)$ be a primitive embedding with $\rO(\Lambda)$-orbit $\epsilon$. 
    If $\cA\subset L_{\R}$ is a subset that is contained in an $e$-small ample cone (e.g., a very small ample cone) then the period morphism~\eqref{eq:bar period morphism} induces an isomorphism
    \[
        \overline{\rho}^e_{(L,\cA)}:\cP^{\epsilon}_{(L,\cA)}\iso[\cD(\Pi)^{\circ}/\widetilde{\rO}(\Pi)].
    \]
\end{theorem}

\begin{proof}
    Write $\varphi=\varphi^e_{(L,\cA)}$, $\rho=\rho^e_{(L,\cA)}$, and $\overline{\rho}=\overline{\rho}^e_{(L,\cA)}$. 
    Consider the commutative diagram
    \[
        \begin{tikzcd}
            \widetilde{\rP}^e_{(L,\cA)}\arrow{r}{\rho}\arrow{d}[swap]{\varphi}&\cD(\Pi)\arrow{d}{\pi}\\
            \cP^{\epsilon}_{(L,\cA)}\arrow{r}{\overline{\rho}}&\left[\cD(\Pi)/\widetilde{\rO}(\Pi)\right]
        \end{tikzcd}
    \]
    where $\pi$ is the quotient map. 
    The period morphism $\rho$ is equivariant with respect to the $\widetilde{\rO}(\Pi)$-actions, and the map $\varphi$ identifies $\cP^{\epsilon}_{(L,\cA)}$ with the quotient of $\widetilde{\rP}^e_{(L,\cA)}$ by the $\widetilde{\rO}(\Pi)$-action. 
    It follows formally that the square is Cartesian. 
    The result follows from Theorem \ref{thm:torelli, marked version} and the fact that $\pi$ is faithfully flat.
\end{proof}

\subsection{Lattice-polarized twisted K3 surfaces over the complex numbers}

We now discuss lattice-polarized twisted K3 surfaces over the complex numbers. 
We work over $\C$ throughout this section. 
Fix a positive integer $n$.

\begin{definition}
    For an $\rO(\Lambda)$-orbit of primitive embeddings $\epsilon\subset\Emb(L,\Lambda)$ we let $\cP^{[n],\epsilon}_{(L,\cA)}$ denote the full substack of $\cP_{(L,\cA)}^{[n]}$ whose objects over a $\C$-scheme $T$ are tuples $(f:X\to T,j,\alpha)$ such that each geometric fiber of $X$ has type $\epsilon$. 
    We let $\cQ^{[n],\epsilon}_{(L,\cA)}$ be the full substack of $\cQ_{(L,\cA)}^{[n]}$ consisting of objects whose underlying polarized K3 surfaces is of type $\epsilon$. 
    Equivalently, this is the quotient of $\cP^{[n],\epsilon}_{(L,\cA)}$ by the natural $L\otimes\Z/n$-action.
\end{definition}

Anticipating our applications to Brauer groups, we will focus on the stacks $\cQ^{[n],\epsilon}_{(L,\cA)}$. 
We have a disjoint union decomposition
\[
    \cQ_{(L,\cA)}^{[n]}=\bigsqcup_{\epsilon\in\Emb(L,\Lambda)/\rO(\Lambda)}\cQ^{[n],\epsilon}_{(L,\cA)}
\]
(and similarly for $\cP$).

These moduli spaces can be described in more concrete terms analogous to the description in~\cite{Brakkee}*{\S2.2} of the moduli space of complex polarized twisted K3 surfaces. 
We outline this here. 
Let $e:L\hookrightarrow\Lambda$ be a primitive embedding with $\rO(\Lambda)$-orbit $\epsilon$ and orthogonal complement $\Pi_{\epsilon}:=e(L)^{\perp}\subset\Lambda$. 
Consider the forgetful morphism
\begin{equation}\label{eq:forgetful morphism on Q}
    \cQ^{[n],\epsilon}_{(L,\cA)}\to\cP^{\epsilon}_{(L,\cA)}.
\end{equation}
Let $(X,j)$ be a primitively $L$-polarized K3 surface over $\C$ and let $l:\rH^2(X,\Z)\iso\Lambda$ be a marking of type $e$. 
Then $l$ reduces modulo $n$ to an isomorphism $\rH^2(X,\mu_n)\iso\Lambda\otimes\Z/n$, which (using that $j$ is primitive) induces an isomorphism
\[
    \rH^2(X,\mu_n)/\partial_n(j(L))\simeq(\Lambda/e(L))\otimes\Z/n.
\]
There is a similar isomorphism in families. 
This shows that each fiber of the map~\eqref{eq:forgetful morphism on Q} is isomorphic to $(\Lambda/e(L))\otimes\Z/n\simeq(\Z/n)^{\oplus 22-\rk(L)}$, and $\cQ^{[n],\epsilon}_{(L,\cA)}$ is an \'{e}tale local system over $\cP^{\epsilon}_{(L,\cA)}$. 
The above isomorphisms describe a trivialization of this local system after pullback along the $\widetilde{\rO}(\Pi_{\epsilon})$-torsor $\widetilde{\rP}^{e}_{(L,\cA)}\to\cP^{\epsilon}_{(L,\cA)}$. 
That is, we have a Cartesian diagram
\[
    \begin{tikzcd}
        \widetilde{\rP}_{(L,\cA)}^e\times(\Lambda/e(L)\otimes\Z/n)\arrow{r}{\varphi'}\arrow{d}&\cQ^{[n],\epsilon}_{(L,\cA)}\arrow{d}\\
        \widetilde{\rP}_{(L,\cA)}^e\arrow{r}{\varphi}&\cP^{\epsilon}_{(L,\cA)},
    \end{tikzcd}
\]
where the right vertical arrow is the forgetful morphism. 
The lower horizontal map $\varphi$ is the quotient of $\widetilde{\rP}^e_{(L,\cA)}$ by the $\widetilde{\rO}(\Pi_{\epsilon})$-action (in particular, it is an $\widetilde{\rO}(\Pi_{\epsilon})$-torsor). 
Thus $\widetilde{\rP}_{(L,\cA)}^e\times(\Lambda/e(L)\otimes\Z/n)$ is equipped with an action of $\widetilde{\rO}(\Pi)$, and the upper horizontal map induces an isomorphism
\begin{equation}\label{eq:quotient presentation of Q}
    \left[\widetilde{\rP}_{(L,\cA)}^e\times(\Lambda/e(L)\otimes\Z/n)/\widetilde{\rO}(\Pi_{\epsilon})\right]\iso\cQ^{[n],\epsilon}_{(L,\cA)}.
\end{equation}
    
We describe this action explicitly. 
The association $v\mapsto v.\uu$ induces an isomorphism
\[
    \Lambda/e(L)\iso\hom(\Pi_{\epsilon},\Z)=:\Pi_{\epsilon}^{\vee},
\]
hence an isomorphism
\[
    (\Lambda/e(L))\otimes\Z/n\iso\hom(\Pi_{\epsilon},\Z/n)=\Pi_{\epsilon}^{\vee}\otimes\Z/n.
\]
The group $\widetilde{\rO}(\Pi_{\epsilon})$ acts on the left on $\widetilde{\rP}^e_{(L,\cA)}$, and on $\Pi_{\epsilon}^{\vee}$ by the formula $g\cdot w=w\circ g^{-1}$. 
The morphism $\varphi'$ is the quotient of $\widetilde{\rP}_{(L,\cA)}^e\times(\Lambda/e(L)\otimes\Z/n)$ by the resulting diagonal action of $\widetilde{\rO}(\Pi_{\epsilon})$.

This quotient presentation of $\cQ^{[n],\epsilon}_{(L,\cA)}$ gives rise to a disjoint union decomposition indexed by elements of $\hom(\Pi_{\epsilon},\Z/n)/\widetilde{\rO}(\Pi_{\epsilon})$. 
More precisely, let $w\in \hom(\Pi_{\epsilon},\Z/n)$ be a functional and let $\omega\subset\hom(\Pi_{\epsilon},\Z/n)$ be its $\widetilde{\rO}(\Pi_{\epsilon})$-orbit. 
We say that an object $(X,j,\alpha)$ of $\cQ^{[n]}_{(L,\cA)}$ over $\C$ has \emph{type $(\epsilon,\omega)$} if for some (equivalently, for any) marking $l:\rH^2(X,\Z)\iso\Lambda$ the composition $l\circ\partial\circ j:L\hookrightarrow\Lambda$ is in the orbit $\epsilon$, and the image of $\alpha$ under the isomorphism
\[
    \rH^2(X,\mu_n)/\partial_n(j(L))\iso(\Lambda/e(L))\otimes\Z/n\simeq\hom(\Pi_{\epsilon},\Z/n)
\]
induced by $l$ is in the orbit $\omega$. 
We let $\cQ^{[n],(\epsilon,\omega)}_{(L,\cA)}$ denote the full substack $\cQ_{(L,\cA)}^{[n]}$ consisting of objects of type $(\epsilon,\omega)$ on every fiber. 
We have a disjoint union decomposition
\[
    \cQ^{[n],\epsilon}_{(L,\cA)}=\bigsqcup_{\omega\in\hom(\Pi_{\epsilon},\Z/n)/\widetilde{\rO}(\Pi_{\epsilon})}\cQ^{[n],(\epsilon,\omega)}_{(L,\cA)}.
\]
Furthermore, letting $\widetilde{\rO}_w(\Pi_{\epsilon})\subset\widetilde{\rO}(\Pi_{\epsilon})$ be the stabilizer of $w$, the above quotient presentation of $\cQ^{[n],\epsilon}_{(L,\cA)}$~\eqref{eq:quotient presentation of Q} restricts to an isomorphism
\[
    \left[\widetilde{\rP}^e_{(L,\cA)}/\widetilde{\rO}_w(\Pi_{\epsilon})\right]=\left[(\widetilde{\rP}^e_{(L,\cA)}\times\left\{w\right\})/\widetilde{\rO}_w(\Pi_{\epsilon})\right]\iso\cQ^{[n],(\epsilon,\omega)}_{(L,\cA)}.
\]
The period morphism~\eqref{eq:period morphism, L polarized} is equivariant with respect to the $\widetilde{\rO}(\Pi_{\epsilon})$-actions, and passing to the quotient gives a period morphism for $n$-twisted $L$-polarized K3 surfaces
\begin{equation}\label{eq:period morphism for n twisted k3s}
    \overline{\rho}^{[n]}:\cQ^{[n],(\epsilon,\omega)}_{(L,\cA)}\to\left[\cD(\Pi_{\epsilon})/\widetilde{\rO}_w(\Pi_{\epsilon})\right].
\end{equation}

This fits into the commutative diagram
\begin{equation}\label{eq:double diagram}
    \begin{tikzcd}
        \widetilde{\rP}^e_{(L,\cA)}\arrow{r}{\rho}\arrow{d}&\cD(\Pi_{\epsilon})\arrow{d}\\
        \cQ^{[n],(\epsilon,\omega)}_{(L,\cA)}\arrow{d}\arrow{r}{\overline{\rho}^{[n]}}&\left[\cD(\Pi_{\epsilon})/\widetilde{\rO}_w(\Pi_{\epsilon})\right]\arrow{d}\\
        \cP^{\epsilon}_{(L,\cA)}\arrow{r}{\overline{\rho}}&\left[\cD(\Pi_{\epsilon})/\widetilde{\rO}(\Pi_{\epsilon})\right],
    \end{tikzcd}
\end{equation}
both squares of which are Cartesian.
    
\begin{theorem}\label{thm:period morphism corollary, twisted}
    If $\cA\subset L_{\R}$ is a subset that is contained in an $e$-small ample cone of $L$ (e.g., a very small ample cone for $L$) then the period morphism~\eqref{eq:period morphism for n twisted k3s} induces an isomorphism
    \[
        \overline{\rho}^{[n]}:\cQ^{[n],(\epsilon,\omega)}_{(L,\cA)}\iso\left[\cD(\Pi_{\epsilon})^{\circ}/\widetilde{\rO}_w(\Pi_{\epsilon})\right].
    \]
\end{theorem}

\begin{proof}
    This follows from Theorem \ref{thm:torelli, marked version} combined with the diagram~\eqref{eq:double diagram}.
\end{proof}

We will mainly be interested in the stack $\cQ_{(L,\cA)}^{(n)}\subset\cQ_{(L,\cA)}^{[n]}$ of tuples $(f:X\to T,j,\alpha)$ for which each $\alpha_t$ has exact order $n$
in $\rH^2(X,\mu_n)/\partial_n(j(L))$ (Definition~\ref{def_stackswithBrauerclasses}).
It has a disjoint union decomposition
\[
    \cQ^{(n)}_{(L,\cA)}=\bigsqcup_{\epsilon,\omega}\cQ^{[n],(\epsilon,\omega)}_{(L,\cA)}.
\]
where $\omega$ runs over all classes in $\hom(\Pi_{\epsilon},\Z/n)/\widetilde{\rO}(\Pi_{\epsilon})$ of order exactly $n$.
From the results in \S\ref{subsec_twiststoBrauer} it follows that the stacks $\cQ^{[n]}_{(L,\cA)}$ and $\cQ^{(n)}_{(L,\cA)}$ have coarse moduli spaces.

\begin{definition}
    Denote by $\mathrm{Q}_{(L,\mathcal{A})}^{[n]}$ and $\mathrm{Q}_{(L,\mathcal{A})}^{(n)}$ the coarse moduli spaces of the stacks $\mathcal{Q}_{(L,\mathcal{A})}^{[n]}$ and $\mathcal{Q}_{(L,\mathcal{A})}^{(n)}$, respectively.
\end{definition}

\begin{corollary}
    \label{cor: Complex Period Space Components Comparison}
   We have isomorphisms
    \[
        \begin{split}
            \mathrm{Q}_{(L,\mathcal{A}),\C}^{[n]}&\cong \bigsqcup_{\substack{\epsilon\in\Emb(L,\Lambda),\\
            \omega\in\hom(\Pi_{\epsilon},\Z/n)/\widetilde{\rO}(\Pi_{\epsilon})}} \cD(\Pi_{\epsilon})^{\circ}/\widetilde{\rO}_w(\Pi_{\epsilon})\\
            \mathrm{Q}_{(L,\mathcal{A}),\C}^{(n)}&\cong \bigsqcup_{\substack{\epsilon\in\Emb(L,\Lambda),\\
            \omega\in\hom(\Pi_{\epsilon},\Z/n)/\widetilde{\rO}(\Pi_{\epsilon}),\\
            \ord(\omega)=n}} \cD(\Pi_{\epsilon})^{\circ}/\widetilde{\rO}_w(\Pi_{\epsilon})\rtag{\qed}
        \end{split}
    \]
\end{corollary}

\section{Polarizations of Rank 19}
\label{S:Rank19Polariations}

In this section, we specialize to the case where $L$ is an even hyperbolic lattice of rank $19$ that can be embedded in the K3 lattice $\Lambda$. 
We fix a positive integer $n$ and a very small ample cone $\mathcal{a} \subset L_{\R}$. 
Our goal is to study the complex coarse moduli spaces $\mathrm{Q}_{(L,\mathcal{a}),\C}^{(n)}$; they are $1$-dimensional finite type $\C$-schemes whose connected components are smooth quasi-projective curves (Proposition~\ref{Prop_ComponentsSmooth}). 
Our goal is to prove that if $n=\ell^m$ is a prime power, then for any $g\geq 0$ the components of $\mathrm{Q}_{(L,\mathcal{a}),\C}^{(\ell^m)}$ have genus $> g$ for all sufficiently large $m$. 
We use this to deduce statements about points of bounded degree over number fields for $\mathrm{Q}_{(L,\mathcal{a})}^{(\ell^m)}$; c.f., Corollary~\ref{cor:finitely many degree d points}.

\subsection{Descending chains within (special) orthogonal groups}

By Corollary~\ref{cor: Complex Period Space Components Comparison}, when $\rk(L)=19$, the components of $\mathrm{Q}_{(L,\mathcal{a}),\C}^{(n)}$ are Zariski open subsets of varieties $\cD(\Pi)/G$, where
\[
    \cD(\Pi)=\{[x]\in\mathbb{P}(\Pi\otimes\C)\mid (x.x)=0,\; (x.\overline{x})>0\}
\]
is the period domain of an even indefinite lattice $\Pi$ of rank 3,
and $G$ a finite index subgroup of the orthogonal group $\tO(\Pi)$.
This period domain has two connected components, which we label $\cD^+(\Pi)$ and $\cD^-(\Pi)$; they are preserved by the action of $\tO^+(\Pi)$. 
If $G$ contains an element that interchanges these components, then the quotient $\cD(\Pi)/G$ has one connected component; otherwise, it has two.
Either way, writing $G^+=G\cap\tO^+(\Pi)$, the components of $\cD(\Pi)/G$ are isomorphic to $\cD^+(\Pi)/G^+$. 
Hence, we need only consider quotients of the form $\cD^+(\Pi)/G^+$.

\medskip

For the remainder of \S\ref{S:Rank19Polariations}, we fix a primitive embedding $e \in \Emb(L,\Lambda)$ and let 
\[
    \Pi:=e(L)^{\perp}\subset \Lambda.
\]
Let $w\in\hom(\Pi,\Z/n\Z)$ be an element of order $n$, and let $\kernel(w)\subset \Pi$ be its kernel, a full-rank sublattice of $\Pi$ of index $n$. 
Write
\[
    \mathrm{O}(\Pi,\kernel(w)) := \{ f \in \mathrm{O}(\Pi) \mid f \text{ preserves }\kernel(w)\},
\]
and let $\widetilde{\mathrm{O}}^+(\Pi,\kernel(w))=\mathrm{O}(\Pi,\kernel(w))\cap\widetilde{\mathrm{O}}^+(\Pi)$.
Note that if $f\in \widetilde{\mathrm{O}}^+(\Pi)$, then $w\circ f=w$ if and only if for all $y\in \Pi$, we have $w(f(y)-y)=0$, that is, $f(y)-y\in\kernel(w)$. 
This shows:

\begin{lemma}
    The stabilizer subgroup $\widetilde{\mathrm{O}}^+_w(\Pi)$ of $w$ in $\widetilde{\mathrm{O}}^+(\Pi)$ equals
    \[
        \{f\in \widetilde{\mathrm{O}}^+(\Pi,\kernel(w))\mid f\textrm{ induces the identity on }\Pi/\kernel(w)\}.\eqno\qed
    \]
\end{lemma}

The subgroup $\widetilde{\mathrm{O}}^+_w(\Pi)$ has finite index in $\widetilde{\mathrm{O}}^+(\Pi,\kernel(w))$. 
Hence, to prove that the genus of $\cD^+(\Pi)/\widetilde{\mathrm{O}}^+_w(\Pi)$ grows with the order of $w$, it suffices to show the same statement for the curve $\cD^+(\Pi)/\tO^+(\Pi,\kernel(w))$. 
We thus turn our attention to describing $\tO^+(\Pi,\kernel(w))$ concretely.

Any element $w\in \hom(\Pi,\Z/n\Z)$ of order $n$ is of the form $(v.-)$ for some $v\in \Pi^{\vee}\subset \Pi_\mathbb{Q}$, unique up to $n\Pi^{\vee}$.
We may assume that $v$ is primitive: if it is divisible by $t \in \Z_{>0}$, then $t$ is coprime to~$n$ because $\ord(w)=n$. 
We can add to $v$ an element $nv_0\in n\Pi^{\vee}$, with $v_0\in \Pi^{\vee}$ primitive and not a multiple of $v$, to find a primitive element $v'=v+nv_0\in \Pi^{\vee}$ with $w=(v'.-)$.

The map $(v.-)\colon \Pi\to \Z$ is surjective: If $x_1$, $x_2$, $x_3$ is a basis of $\Pi$, then $(v.x_1)$, $(v.x_2)$, and $(v,x_3)$ are coprime, hence there exists an integral linear combination $x$ of $x_1$, $x_2$, $x_3$ such that $(v.x)=1$. 
Let $V$ be the kernel of $(v.-)$, so $\Pi=\langle V,x\rangle$.
Note that this ``decomposition'' $\Pi=\langle V,x\rangle$ is not an orthogonal direct sum in general.
Define
\[\Pi_{n}:=\kernel(w)=\langle V,nx\rangle.\]
The group $\tO^+(\Pi,\kernel(w))=\tO^+(\Pi,\Pi_{n})$ is the subgroup of $\tO^+(\Pi)$ consisting of elements preserving $V$ modulo $n$.

\medskip

Assume from now on that $n$ is a prime power $\ell^m$. 
We want to show that the index of $\tO^+(\Pi,\Pi_{\ell^m})$ in $\tO^+(\Pi)$ gets arbitrarily large when $m$ grows.
Consider the chain of inclusions of finite index subgroups
\[
    \dots\subset \tO^+(\Pi,\Pi_{\ell^m})\subset\tO^+(\Pi,\Pi_{\ell^{m-1}})\subset\dots\subset\tO^+(\Pi,\Pi_{\ell})\subset\tO^+(\Pi).
\]
We claim that for each $m$ there is an $m'$, independent of $w$ and the decomposition $\Pi=\langle V,x\rangle$, such that $\tO^+(\Pi,\Pi_{\ell^{m+m'}})$ is a strict subgroup of $\tO^+(\Pi,\Pi_{\ell^{m}})$.
As $\tO^+(\Pi,\Pi_{\ell^m})$ has index at most 2 in $\tO(\Pi,\Pi_{\ell^m})$ for each $m$, it is equivalent to show this for $\tO(\Pi,\Pi_{\ell^m})$ instead of $\tO^+(\Pi,\Pi_{\ell^m})$.

Let $N:=v_{\ell}(\disc(\Pi))$; then $v_{\ell}(\disc(\Pi_{\ell^m}))=2m+N$. 
Set
\[
    N' = 
    \begin{cases}
        2(2m + N) + 1 & \text{if $\ell$ is odd}, \\
        \max\{3(2m + N) + 4,4(2m + N)+1\} & \text{if }\ell = 2.
    \end{cases}
\]
By Theorem~\ref{thm_IndexReductionO}, for any $f \geq N'$, the image of the map
\[
    \tO(\Pi_{\ell^m})\to \tO(\Pi_{\ell^m}/\ell^f\Pi_{\ell^m})
\]
has index bounded by
\[
    B_{\ell,2m+N} := 
    \begin{cases}
        24(2m+N)+4 & \text{if $\ell$ is odd}, \\
        20(2m+N)+48 & \text{if }\ell = 2.
    \end{cases}
\]
Note that $B_{\ell,2m+N}$ depends linearly on $m$.

\begin{proposition}
\label{Prop_StrictSubgroups}
    Assume that $m$ is large enough so that $(m+1)^2>B_{\ell,2m+N}$.
    Take $f\geq N'$.
    Then $\tO(\Pi,\Pi_{p^f})$ is a strict subgroup of $\tO(\Pi,\Pi_{p^m})$. 
\end{proposition}

\begin{proof}
    Let $\varphi$ be the restriction map
    \[
        \varphi\colon \tO(\Pi,\Pi_{\ell^m})\overset{\varphi}{\longrightarrow}\tO(\Pi_{\ell^m}),
    \]
    which is injective. 
    We show that $\varphi(\tO(\Pi,\Pi_{\ell^f}))$ is a strict subgroup of $\varphi(\tO(\Pi,\Pi_{\ell^m}))$. 
    The image $\varphi(\tO(\Pi,\Pi_{\ell^m}))$ consists of those $g\in\tO(\Pi_{\ell^m})$ for which the induced action of $g^{\vee}$ on $\Pi_{\ell^m}^{\vee}/\Pi_{\ell^m}$ preserves $\Pi/\Pi_{\ell^m} =\langle x\rangle/\langle \ell^mx\rangle$ \cite{Nikulin}*{Proposition~1.4.2}. 
    On the other hand, 
    \begin{align*}
        g^{\vee}\text{ preserves }\Pi/\Pi_{\ell^m} &\Leftrightarrow g^{\vee} \text{ preserves }\langle x\rangle \mod \Pi_{\ell^m}\\
        &\Leftrightarrow g \text{ preserves }\langle \ell^mx\rangle \mod \ell^m\Pi_{\ell^m}\\ 
        &\Leftarrow g \text{ preserves }\langle \ell^mx\rangle \mod \ell^f\Pi_{\ell^m}.
    \end{align*}
    Thus, $\varphi(\tO(\Pi,\Pi_{\ell^m}))$ contains the group
    \[
        G:=\{g\in\tO(\Pi_{\ell^m})\mid \text{induced map on $\Pi_{\ell^m}/\ell^f\Pi_{\ell^m}$ preserves }\langle \ell^m x\rangle\}.
    \]
    Since
    \[
        \varphi(\tO(\Pi,\Pi_{\ell^f}))=\varphi(\tO(\Pi,\Pi_{\ell^m}))\cap\tO(\Pi_{\ell^m},\Pi_{\ell^f}),
    \]
    it is enough to show that $G\supsetneq G\cap\tO(\Pi_{\ell^m},\Pi_{\ell^f})=:G'$. 
    This last group equals
    \[
        \{g\in\tO(\Pi_{\ell^m})\mid \text{induced map on $\Pi_{\ell^m}/\ell^f\Pi_{\ell^m}$ preserves }\langle \ell^m x\rangle\text{ and }\overline{V}:=V/\ell^f\Pi_{\ell^m}\}.
    \]
    Consider the following subgroups of $\tO(\Pi_{\ell^m}/\ell^f\Pi_{\ell^m})$:
    \begin{align*}
        \overline{G}:=&\{g\in\tO(\Pi_{\ell^m}/\ell^f\Pi_{\ell^m})\mid g\text{ preserves }\langle \ell^m x\rangle\}\\
        \overline{G}':=&\{\ell\in\tO(\Pi_{\ell^m}/\ell^f\Pi_{\ell^m})\mid g\text{ preserves }\langle \ell^m x\rangle\text{ and }\overline{V}\}.
    \end{align*}
    The group $G$ is the inverse image of $\overline{G}$ under the reduction map $\psi\colon\tO(\Pi_{\ell^m})\to \tO(\Pi_{\ell^m}/\ell^f\Pi_{\ell^m})$. 
    By assumption on $f$, the image $\psi(\tO(\Pi_{\ell^m}))$ has index at most $B_{\ell,2m+N}$ in $\tO(\Pi_{\ell^m}/\ell^f\Pi_{\ell^m})$; hence, the same is true for the index of $\psi(G)$ in $\overline{G}$. 
    Because the diagram
    \[
        \xymatrix{
            G\ar[r]^{\psi} & \overline{G}\\
            G'\ar@{^(->}[u] \ar[r] & \overline{G}'\ar@{^(->}[u]
        }
    \]
commutes, it is enough to show that $\overline{G}'\subset\overline{G}$ has index larger than $B_{\ell,2m+N}$. 
Write $y:=\ell^mx$. 
Every element of $\overline{G}'$ is of the form
    \[
        \begin{pmatrix}
            a & b & 0\\
            c & d & 0 \\
            0 & 0 & f
        \end{pmatrix}
    \]
    with $\left(\begin{smallmatrix} a & b \\ c & d\end{smallmatrix}\right)\in\tO(V)$ and $(h)\in\tO(\langle y\rangle)$. 
    On the other hand, $\ell^{f-m}y$ lies in the radical of $\Pi_{\ell^m}/\ell^f\Pi_{\ell^m}$. 
    Hence, $\overline{G}$ contains all matrices of the form 
    \[
        \begin{pmatrix}
            a & b & 0\\
            c & d & 0 \\
            e & e' & f
        \end{pmatrix}
    \]
    with $\left(\begin{smallmatrix} a & b \\ c & d\end{smallmatrix}\right)\in\tO(V)$,  $(h)\in\tO(\langle y\rangle)$ and $e,e'\in \{0,\ell^{f-m},\ell^{f-m+1},\dots,\ell^{f-1}\}$. 
    It follows that the index of $\overline{G}'$ in $\overline{G}$ is at least $(m+1)^2>B_{\ell,2m+N}$, which finishes the proof.
\end{proof}

\begin{remark}
    \label{rem:Independence}
    We stress that the required lower bounds for $m$ and $f$ in Proposition~\ref{Prop_StrictSubgroups} depend only on $N = v_{\ell}(\disc(\Pi))$; in particular, they do not depend on the choice of $w$ or the decomposition $\Pi=\langle V,x\rangle$, or indeed the embedding $L\hookrightarrow\Lambda_{\mathrm{K}3}$.
\end{remark}

\subsection{Congruence arithmetic Fuchsian groups}
\label{ss:CongruenceFuchsian}

We describe the curves $\cD^+(\Pi)/\tO^+(\Pi,\Pi_n)$ from the perspective of modular curves.
Recall that for the rank-3 lattice $\Pi$, the space $\cD^+(\Pi)$ is biholomorphic to the upper half plane $\cH$.
For instance, we may choose a basis $e_1,e_2,e_3$ of $\Pi_{\R}$ such that $e_1,e_2$ form the standard basis of a hyperbolic plane $U_{\R}$ and $\Pi_{\R}=\langle e_1,e_2\rangle\oplus\langle e_3\rangle=U_{\R}\oplus\langle 2\rangle$.
Then the map
\[
    \cH\to \cD^+(\Pi),\;\; z\mapsto [e_1-z^2e_2+\sqrt{2}ze_3]
\]
is biholomorphic.
This induces a group isomorphism $\PSL_2(\R) \xrightarrow{\simeq} \SO^+(\Pi_{\R})$ from the group of biholomorphic automorphisms of $\cH$: Abstractly, the even Clifford algebra $\clif^0(\Pi_{\R})$ of $\Pi_{\R}$ is isomorphic to $M_2(\R)$ (because $\Pi$ is indefinite), and the image of the Spin group $\spin(\Pi_{\R})\subset \clif^0(\Pi_{\R})$ under this isomorphism is $\SL_2(\R)$. 
On the other hand, the map $\spin(\Pi_{\R})\to\tO(\Pi_{\R})$ has image $\SO^+(\Pi_{\R})$ and kernel $\{\pm 1\}$. 
Putting these facts together gives $\PSL_2(\R) \xleftarrow{\simeq} \spin(\Pi_R)/\{\pm 1\} \xrightarrow{\simeq} \SO^+(\Pi_R)$.
Concretely, $\Pi_{\R}\cong U_{\R}\oplus \langle 2\rangle$ can be viewed as the subspace of $M_2(\R)$ with basis
\[
    e_1 =   \begin{pmatrix}
                0 & 0 \\ -1 & 0 
            \end{pmatrix},
    \;\; e_2 =  \begin{pmatrix}
                    0 & -1 \\ 0 & 0 
                \end{pmatrix},
    \;\; e_3 =  \begin{pmatrix}
                    -1 & 0 \\ 0 & 1 
                \end{pmatrix}
\]
and quadratic form given, up to a sign, by $x\mapsto \det(x)$.
Then $\SL_2(\R)$ acts on $\Pi_{\R}$ by conjugation, yielding an isomorphism $\PSL_2(\R)\xrightarrow{\simeq}\SO^+(\Pi_{\R})$.

In any case, we can view $\SO^+(\Pi) \subset \SO^+(\Pi_\R)$ as a discrete subgroup of $\PSL_2(\R)$, acting on $\cH$ by M\"obius transformations. 
More precisely, the even Clifford algebra $\clif^0(\Pi)$ is an order in $\clif^0(\Pi_{\Q})$, and $\spin(\Pi)\subset\clif^0(\Pi)$ is the subgroup of elements of reduced norm 1. 
The group $\SO^+(\Pi)$ is commensurable to $\spin(\Pi)/\{\pm 1\}$; that is, it is an arithmetic Fuchsian group (\cite{Voight}*{\S38.1.5}). 
The same is true for every finite index subgroup $G\subset\SO^+(\Pi)$.
The quotient $\cD^+(\Pi)/G\cong\cH/G$ is naturally an orbifold, but also has the structure of a Riemann surface, with a natural smooth compactification $\overline{\cH/G}=\cH^*/G$; see~\cite{Shimura}*{\S1.5}. 
We deduce:

\begin{proposition}\label{Prop_ComponentsSmooth}
    Let $G\subset\SO^+(\Pi)$ be a finite index subgroup.
    Then $\cD^+(\Pi)/G$ is a smooth quasi-projective curve.
    \qed
\end{proposition}

Note that $\tO^+(\Pi)=\langle \SO^+(\Pi),-\id\rangle$ and $-\id$ acts as the identity on $\cD(\Pi)\subset\mathbb{P}(\Pi_{\C})$.
Hence, Proposition~\ref{Prop_ComponentsSmooth} applies to the curves $\cD^+(\Pi)/\tO^+(\Pi,\Pi_{\ell^m})\cong \cD^+(\Pi)/\SO^+(\Pi,\Pi_{\ell^m})$.
Moreover, the lower bounds for $m$ and $f$ in Proposition~\ref{Prop_StrictSubgroups} hold for the groups $\SO^+(\Pi,\Pi_{\ell^m})$ as well.
Hence, from now on we will replace $\tO^+(\Pi,\Pi_{\ell^m})$ with $\SO^+(\Pi,\Pi_{\ell^m})$.

\begin{definition}
    Let $\Gamma_{\Pi_n}$ be the image of $\SO^+(\Pi,\Pi_n)$ under the identification $\SO^+(\Pi_{\R})\cong\PSL(2,\R)$.
\end{definition}

Note that $\SO^+(\Pi,\Pi_n)$ contains
\[
    \{g\in \SO^+(\Pi)\mid g\equiv\id\text{ modulo }n\}.
\]
Hence, $\Gamma_{\Pi_n}$ is a congruence arithmetic Fuchsian group \cite{BergeronClozel_book}*{pp.~9--10}.

\subsection{Finiteness of rational points with increasing Brauer level structures}

The proof of the main result of this section follows~\cite{LMR} and uses the following well-known results about Fuchsian groups.
Let $\Gamma\subset\PSL_2(\R)$ be a Fuchsian group of the first kind; denote by $\lambda_1(\Gamma)$ the first non-zero eigenvalue of the Laplacian acting on the space of square-integrable functions $L^2(\cH/\Gamma)$, and let $\mu(\overline{\cH/\Gamma})$ indicate the hyperbolic area of a fundamental domain for the action of $\Gamma$ on~$\cH$~\cite{Katok}*{\S3.1}.

\begin{theorem}[{\cite{Zograf}*{Theorem~2}; see \cite{YangYau} for the cocompact case}]\label{Thm_genusvseigenvalue}
    Assume that $\mu(\overline{\cH/\Gamma})\geq 32\pi(g(\overline{\cH/\Gamma})+1)$.
    Then
    \[
        \lambda_1(\Gamma)<\frac{8\pi(g(\overline{\cH/\Gamma})+1)}{\mu(\overline{\cH/\Gamma})}.
        \eqno\qed
    \]
\end{theorem}

\begin{theorem}[{\cite{BergeronClozel}*{Theorem~1.2}}]\label{Thm_spectralgap}
    There exists a positive constant $\lambda$ such that for any congruence arithmetic subgroup $\Gamma\subset\PSL_2(\R)$, we have 
    \[
        \lambda_1(\Gamma)\geq \lambda.\eqno\qed
    \]
\end{theorem}

We apply these results to the groups $\Gamma_{\Pi_{\ell^m}}$ defined in~\S\ref{ss:CongruenceFuchsian}. 
Denote by $I_{\Pi_{\ell^m}}$ the index of $\SO^+(\Pi,\Pi_{\ell^m})$ in $\SO^+(\Pi)$. 
Then the area $\mu(\overline{\cH/\Gamma_{\Pi_{\ell^m}}})$ equals $I_{\Pi_{\ell^m}}\cdot \mu(\overline{\cH/\Gamma_\Pi})$; see \cite{Katok}*{Theorem~3.1.2}.
By Proposition~\ref{Prop_StrictSubgroups}, the indices $I_{\Pi_{\ell^m}}$ grow with $m$.

\begin{theorem}
    Fix $g \in \Z_{\geq0}$. 
    There is an integer $m_0 := m_0\left(g,v_\ell(\disc(\Pi))\right)$ such that, for any $m\geq m_0$, the curve $\overline{\cH/\Gamma_{\Pi_{\ell^m}}}$ has genus $> g$.
\end{theorem}

\begin{proof}
    By Proposition~\ref{Prop_StrictSubgroups}, there exists an $m_0$ such that for all $m\geq m_0$, we have 
    \[
        I_{\Pi_{\ell^m}}\cdot \mu(\overline{\cH/\Gamma_\Pi})\geq 32(g+1)\;\text{ and }\; I_{\Pi_{\ell^m}}\cdot\mu(\overline{\cH/\Gamma_\Pi})\geq\frac{8\pi (g+1)}{\lambda}.
    \]
    (These bounds are not independent, of course: once a precise value of $\lambda$ is known, one inequality will imply the other. 
    This point does not affect our argument.)
    If
    \[
        I_{\Pi_{\ell^m}}\cdot\mu(\overline{\cH/\Gamma_\Pi})<32(g(\overline{\cH/\Gamma_{\Pi_{\ell^m}}})+1)
    \]
    then we have $32(g(\overline{\cH/\Gamma_{\Pi_{\ell^m}}})+1)>32(g+1)$ and hence $g(\overline{\cH/\Gamma_{\Pi_{\ell^m}}})>g$. 
    Otherwise, we must have 
    \[
        I_{\Pi_{\ell^m}}\cdot\mu(\overline{\cH/\Gamma_\Pi})\geq 32(g(\overline{\cH/\Gamma_{\Pi_{\ell^m}}})+1).
    \]
    In this case, Theorems~\ref{Thm_genusvseigenvalue} and~\ref{Thm_spectralgap} then imply that
    \[
        8\pi(g(\overline{\cH/\Gamma_{\Pi_{\ell^m}}})+1)>\lambda \cdot I_{\Pi_{\ell^m}}\mu(\overline{\cH/\Gamma_\Pi})\geq 8\pi (g+1),
    \]
    and hence $g(\overline{\cH/\Gamma_{\Pi_{\ell^m}}})>g$ as well.
\end{proof}

Since $\disc(\Pi) = \disc(L)$, we obtain the following corollary.

\begin{corollary}
    Fix $g \in \Z_{\geq0}$. 
    There is an $m_1 := m_1(g,\ell,L) \in \Z_{>0}$ such that, for all $m\geq m_{1}$, every connected component of $\mathrm{Q}_{(L,\mathcal{a}),\C}^{(\ell^m)}$ has genus $ > g$.
    \qed
\end{corollary}

The next result follows directly from \cite{Abramovich}*{Theorem~1.1}. 
Recall that the $k$-\emph{gonality} $\gamma_k(C)$ of a geometrically integral curve $C$ over a field $k$ is the smallest possible degree of a dominant $k$-rational map $C\dasharrow \PP^1_k$. 
If $K/k$ is a field extension, then $\gamma_k(C) \geq \gamma_K(C_K)$.

\begin{corollary}
    \label{cor:gonality grows}
    Fix $g \in \Z_{\geq0}$. 
    There is an $m_2 := m_2(g,\ell,L) \in \Z_{>0}$ such that, for all $m\geq m_2$, every connected component $C$ of $\mathrm{Q}_{(L,\mathcal{a}),\C}^{(\ell^m)}$, has $\C$-gonality $\gamma_\C(C) > g$.
    \qed
\end{corollary}

\begin{corollary} 
    \label{cor:finitely many degree d points}
    Fix a number field $k$. 
    For every $d$, there is an $n_0 := n_0(d,k,\ell,L) \in \Z_{> 0}$ such that, for all $m\geq n_0$, the space $\mathrm{Q}_{(L,\mathcal{a})}^{(\ell^m)}$ has finitely many points over fields $K/k$ with $[K:k]\leq d$.
\end{corollary}

\begin{proof}
    Let $C$ be a connected component of $\mathrm{Q}_{(L,\mathcal{a})}^{(\ell^m)}$. 
    By~\cite{Frey}*{Proposition~2.2}, if $K$ is a number field with $[K:k] \leq d$, and $C(K)$ is infinite, then there is a $K$-rational covering $\pi \colon C_K \dasharrow \PP^1_K$ of degree $\leq 2d$, i.e., $\gamma_K(C) \leq 2d$. 
    Applying Corollary~\ref{cor:gonality grows} with $g = 2d$, and setting $n_0 = m_2(g,\ell,L)$ then shows that $C(K)$ is finite once $m \geq n_0$.
\end{proof}

\section{Application to Brauer Groups of K3 surfaces with High Picard Rank}

As an application of the moduli-theoretic constructions, we prove Theorem~\ref{thm:uniform bound}.

\subsection{Preliminaries}

We gather a few facts, some standard, and some more recent advances in the arithmetic of K3 surfaces.

\subsubsection{Splitting Picard groups uniformly}

First, we show that Picard groups of K3 surfaces can be split by an extension of the ground field whose degree is bounded by an absolute constant.

\begin{lemma}
    \label{lemma_PicXL=PicXbar}
    There exists a universal constant $C$ such that, if $k$ is a field, $\overline{k}$ is an algebraic closure of $k$, $\ksep$ is a separable closure of $k$ in $\kbar$, and $X$ is a K3 surface over $k$, then there exists an intermediate field extension $k\subseteq K\subseteq \ksep$, Galois over $k$, such that $\pic(X_K) = \pic(X_{\kbar})$ and $[K:k] \leq C$.
\end{lemma}

\begin{proof}
    For a K3 surface $X$, the natural map $\pic(X_{\ksep}) \to \pic(X_{\kbar})$ is an isomorphism~\cite{BLvL}*{Lemma~3.1}. 
    Since $\pic(X_{\ksep})$ is finitely generated, there is a finite separable extension of $k$ of minimal degree over which all the generators of $\pic(X_{\ksep})$ are defined. 
    Let $K$ be the normal closure of this extension; it is a finite Galois extension of $k$.

    Set $\Lambda=\pic(X_{\ksep})$, considered as a $\Z$-lattice, and let $O(\Lambda)$ denote the orthogonal group of $\Lambda$. 
    This is a subgroup of $\mathrm{GL}_n(\mathbf{Z})$ for some $n\leq 22$, hence a subgroup of $\mathrm{GL}_{22}(\mathbf{Z})$. 
    The action of $\Gal(\ksep/k)$ on $\pic(X_{\ksep})$ gives a map
    \[
        \Gal(\ksep/k)\to O(\Lambda).
    \]
    Note that its kernel, which is a Galois group, is $\Gal(\ksep/K)$, by construction. 
    Hence, we get an injection from the finite group $\Gal(\ksep/k)/\Gal(\ksep/K)\cong\Gal(K/k)$ into $O(\Lambda)$.
    By a theorem of Minkowski, there are, up to isomorphism, only finitely many finite groups which occur as subgroups of $\mathrm{GL}_{22}(\mathbf{Z})$; see~\cite{Minkowski} and~\cite{Serre}*{\S1}. 
    In particular, there exists an integer $C$ (independent of everything, including $\Lambda$) such that $\Gal(K/k)$ has order $\leq C$. 
    It follows that $[K:k]\leq C$.    
\end{proof}

\subsubsection{Transcendental Brauer groups and base change}

Second, we investigate the behavior of the transcendental Brauer group of a K3 surface under a base extension of the ground field.

\begin{lemma}
	\label{lem:BrBaseChange}
	Let $X$ be a variety over a field $k$, and let $K/k$ be a field extension. 
    The natural map $h\colon \br (X) \to \br (X_K)$ induces an injection of transcendental Brauer groups
	\[
        \frac{\br(X)}{\br_1(X)} \hookrightarrow \frac{\br(X_K)}{\br_1(X_K)}
	\]
\end{lemma}

\begin{proof}
    Consider the composition
    \[
        \br(X)\overset{\res_{K/k}}{\longrightarrow}\br(X_K)\to\br(\overline{X}),
    \]
    whose kernel is $\br_1(X)$.
    The image $\res_{K/k}(\br_1(X))$ is contained in the kernel of $\br(X_K)\to\br(\overline{X})$, which is $\br_1(X_K)$. 
    Hence, we get an induced sequence of maps
    \[
        \frac{\br(X)}{\br_1(X)}\overset{\res_{K/k}}{\longrightarrow}\frac{\br(X_K)}{\br_1(X_K)}\to \br(\overline{X}),
    \]
    the second of which is injective. 
    Since the composition is injective as well, the map 
    \[
        \displaystyle\res_{K/k}\colon \frac{\br(X)}{\br_1(X)}\to\frac{\br(X_K)}{\br_1(X_K)}
    \]
    must also be injective.
\end{proof}

\subsubsection{Finiteness theorems of Orr--Skorobogatov}

In the proof of Theorem~\ref{thm:uniform bound}, K3 surfaces with CM will be treated separately. 
We record and apply recent work by Orr and Skorobogatov on the finiteness of arithmetic structures for K3 surfaces with complex multiplication (CM) defined over number fields of bounded degree.

\begin{theorem}[\cite{OrrSkorobogatov}*{Theorem~B}]
    \label{thm:OSThmB}
    There are only finitely many $\overline{\Q}$-isomorphism classes of K3 surfaces of CM type which can be defined over number fields of given degree.
\end{theorem}

\begin{theorem}[\cite{OrrSkorobogatov}*{Theorem~C}]
    \label{thm:OSThmC}
    Fix a positive integer $n$, and let $X$ be a K3 surface over a number field $k$. 
    There is a constant $D = D(n,X)$ such that, for each $(\overline{k}/K)$-form $Y$ of $X$ defined over a field $K$  with $[K:k] \leq n$, we have $\#\br(Y_{\overline{k}})^{\Gal(\overline{k}/K)} < D$.
\end{theorem}

\begin{corollary}
    \label{cor:Rank20}
    Fix a positive integer $d$.  
    There is a constant $D' = D'(d)$ such that, for each K3 surface $Y$ of geometric Picard rank $20$ defined over a number field $k$ of degree $d$, we have $\#\br(Y_{\kbar})^{\Gal(\kbar/k)} < D'.$
\end{corollary}

\begin{proof}
    By Theorem~\ref{thm:OSThmB}, there are only finitely many K3 surfaces $X_1,\dots,X_r$ of CM type, up to $\overline{\Q}$-isomorphism, that can be defined over respective number fields $k_1,\dots,k_r$ of degree~$d$. 
    A surface $Y$ as in the statement is necessarily of CM type, so it becomes isomorphic to some $X_i$ over $\overline{\Q}$. 
    Write $K = k_ik$ for the compositum of $k_i$ and $k$ in $\kbar_i = \kbar$. 
    The base change $Y_{K}$ is a K3 surface over $K$, with $[K:k_i] \leq d$. 
    Applying Theorem~\ref{thm:OSThmC} to $X = X_i$, there is a constant $D(d,X_i)$ such that $\#\br(Y_{\overline{k}_i})^{\Gal(\overline{k}_i/K)} < D(d,X_i)$.  
    Note that
    \[
        \#\br (Y_{\kbar})^{\Gal(\kbar/k)} \leq
        \#\br (Y_{\kbar})^{\Gal(\kbar/K)} =
        \#\br (Y_{\overline{k}_i})^{\Gal(\overline{k}_i/K)} < D(d,X_i) 
    \]
    Take $D' = \max_i\{D(d,X_i)\}$.
\end{proof}

\subsection{Proof of Theorem~\ref{thm:uniform bound}}

Let $X$ be an $L$-polarized K3 surface over a number field $k$ of degree $d$, as in the statement of Theorem~\ref{thm:uniform bound}. 
We distinguish two cases, according to the rank $\rho$ of $\pic(X_{\kbar})$, namely $\rho = 19$ or $20$.  
If $\rho = 20$, then by Corollary~\ref{cor:Rank20} there is a constant $D'$, depending on $d$ but not on $L$ or $X$, such that $\#\br(X_{\kbar})^{\Gal(\kbar/k)} < D'$. 
Since $\#\br(X)/\br_1(X) \leq \#\br(X_{\kbar})^{\Gal(\kbar/k)}$, the result for the $\ell$-primary part follows in this case.

Assume now that $\rho = 19$. 
Let $C$ be the universal constant in Lemma~\ref{lemma_PicXL=PicXbar}. 
By Corollary~\ref{cor:finitely many degree d points}, there is an $n \in \Z_{>0}$ such that the scheme $\mathrm{Q}_{(L,\mathcal{a})}^{(\ell^n)}$ has only finitely many points of degree at most $C$ over $k$.
Assume first that $\br (X)/\br_1(X)$ has no class of order $\ell^n$. Then we have
\begin{align*}
    \#\br(X)/\br_1(X)[\ell^{\infty}] &=\#\br(X)/\br_1(X)[\ell^{n-1}] \\
    &\leq \#\br(X_{\overline{k}})[\ell^{n-1}] \\ &=(\ell^{n-1})^{22-\rho(X_{\overline{k}})}=\ell^{3(n-1)}.
\end{align*}
Suppose that $\br (X)/\br_1(X)$ does have a class $[\alpha]$ of order $\ell^n$.
Let $K/k$ be an extension of degree $\leq C$ such that $\pic X_{K} \simeq \pic X_{\kbar}$. 
By Lemma~\ref{lem:BrBaseChange}, the image $[\alpha_{K}]$ of $[\alpha]$ in $\br(X_{K})/\br_1(X_{K})$ has order $\ell^n$. 
Let $\alpha_{K} \in \br(X_{K})$ be any lift of $[\alpha_{K}]$, so $\beta := \ell^n\cdot \alpha_{K}$ lies in $\br_1(X_{K})$. 
Let $m \in \Z{>0}$ be its order in $\br_1(X_{K})$. 
Then the class $m\cdot\alpha_{K} \in \br(X_{K})$ has order $\ell^n$.  
Thus, the triple $(X_{K},j\colon L \hookrightarrow \pic X_{K},\alpha_{K})$ gives rise to a $K$-rational point of $\mathrm{Q}_{(L,\mathcal{a})}^{(\ell^n)}$.

Let $(X_1,j_1,\alpha_1),\dots,(X_r,j_r,\alpha_r)$ in $\mathrm{Q}_{(L,\mathcal{a})}^{(\ell^n)}$ be the finitely many triples that can be defined over number fields $k_1,\dots,k_r$, respectively, such that $[k_i:k] \leq C$. 
The triple $(X_{K},j,\alpha_{K})$ is $\overline{\Q}$-isomorphic to some $(X_i,j_i,\alpha_i)$. 
Applying Theorem~\ref{thm:OSThmC} to $X_i/K_i$ we deduce there is a constant $D = D(C,X_i)$ such that the $\kbar_i/Kk_i$-form $X_{Kk_i}$ of $X_i$, satisfies $\#\br (X_{\kbar_i})^{\Gal(\kbar_i/Kk_i)} < D$. 
Then
\[
    \#\br (X_{\kbar})^{\Gal(\kbar/k)} \leq
    \#\br (X_{\kbar})^{\Gal(\kbar/Kk_i)} =
    \#\br (X_{\kbar_i})^{\Gal(\kbar_i/Kk_i)} < D(C,X_i).
\]
Take $B = \max_i\{D(C,X_i),\ell^{3(n-1)}\}$. 
Then $\#\br(X)/\br_1(X)[\ell^{\infty}] \leq B$. 
\qed

\appendix

\section{Reduction of the Orthogonal Group Modulo Prime Powers}

In this appendix, we prove the following quantitative estimate on the index of the image of reduction modulo powers of a prime $p$ for orthogonal groups of integral lattices whose discriminant is divisible by $p$. Throughout, $v_p\colon \Q_p \to \Z\cup\{\infty\}$ denotes the usual $p$-adic valuation.

\begin{theorem}\label{thm_IndexReductionO}
    Let $p$ be a prime and let $L$ be a non-degenerate $\Z$-lattice of rank 3. Set $N := v_p(\disc(L))$ and
    \[
        N' := 
        \begin{cases}
            2N + 1 & \text{if $p$ is odd}, \\
            \max\{3N+4,4N+1\} & \text{if } p = 2.
        \end{cases}
    \]
    Then, for all $f\geq N'$, the index of the image of the projection map
    \[
        \tO(L)\to\tO(L/p^fL)
    \]
    in $\tO(L/p^fL)$ is at most $24N+4$ when $p$ is odd, and at most $20N+48$ when $p=2$. 
\end{theorem}

\subsection{Hensel lifting for the orthogonal group}

For lack of an easy reference, we prove the following fact about orthogonal groups of rank-$3$ $p$-adic lattices.

\begin{theorem}
    \label{thm:Hensel lifting all primes}
    Fix $N \in \Z_{\geq 0}$, and let $p$ be a rational prime. 
    Define
    \[
        N' := 
        \begin{cases}
            2N + 1 & \text{if $p \neq 2$,} \\
            \max\{3N + 4, 4N + 1\} & \text{if $p = 2$.}
        \end{cases}
    \]
    Let $L$ be a $\Z_p$-lattice of rank $3$ such that $p^{N+1} \nmid \disc(L)$. 
    Then the reduction map
    \[
        O(L) \to O\left( L/p^{N'}L\right)
    \]
    is surjective.
\end{theorem}

The case when $N = 0$ and $p$ is odd follows directly from Hensel's Lemma for smooth schemes over complete Noetherian local rings~\cite{Poonen}*{Theorem~3.5.63}. 
The general case follows from a more punctilious version of the same principle, which we now describe.

Let $A$ be a ring. 
To a collection $\mathbf{f} = (f_1,\dots,f_m) \in A[x_1,\dots,x_n]^m$ of $m$ polynomials in $n$ variables and an element $\mathbf{a} = (a_1,\dots,a_n) \in A^n$ we associate the vector $\mathbf{f}(\mathbf{a}) = (f_1(\mathbf{a}),\dots,f_m(\mathbf{a})) \in A^m$ and the Jacobian matrix
\[
    (D\mathbf{f})(\mathbf{a}) := \left[\frac{\partial f_i}{\partial x_j}(\mathbf{a})\right],
\]
an $m\times n$ matrix with entries in $A$.

\begin{theorem}[Multivariate Hensel's Lemma]
    \label{thm:Hensel}
    Let $A$ be a complete discrete valuation ring, with maximal ideal $\frakm$. 
    Fix integers $e\geq 0$ and $n \geq m$, let $\mathbf{f} = (f_1,\dots,f_m) \in A[x_1,\dots,x_n]^m$ and suppose there is an $\mathbf{a} := (a_1,\dots,a_n) \in A^n$ such that
    \[
        f_i(\mathbf{a}) \equiv 0 \bmod \frakm^{2e+1}\quad 1\leq i \leq m.
    \]
    If the Jacobian matrix $(D\mathbf{f})(\mathbf{a})$ has a maximal minor that does not belong to $\frakm^{e+1}$, then there exists a unique $\mathbf{b} = (b_1,\dots,b_n) \in A^n$ such that 
    \begin{itemize}
        \item $\mathbf{f}(\mathbf{b}) = 0 \in A^m$,
        \item $b_i \equiv a_i \bmod \frakm^{e+1}$ for $1 \leq i \leq m$, and
        \item $b_i = a_i$ for $m+1 \leq i \leq n$.
    \end{itemize}
\end{theorem}

\begin{proof}
    See~{\cite{CLNS18}*{Lemma~1.3.3}}. 
    Note, however, that the hypothesis in~\emph{op.\!\! cit.}\ that $\mathbf{a} \in \frakm^n$ can be relaxed to $\mathbf{a} \in A^n$, as a careful inspection of the proof reveals.
\end{proof}

We begin by treating the case of Theorem~\ref{thm:Hensel lifting all primes} where $p$ is odd.

\begin{proposition}\label{prop_ReductionforOsurjective}
    Let $p$ be and odd prime and let $L$ be a $\Z_p$-lattice of rank $3$ such that $p^{N+1} \nmid \disc(L)$. 
    Then the reduction map
    \[
        O(L) \to O(L/p^{2N + 1}L)
    \]
    is surjective.
\end{proposition}

\begin{proof}
    By~\cite{Cassels}*{Proof of Theorem~3.1}, up to $\Z_p$-equivalence we may assume that a Gram matrix for $L$ is the diagonal matrix $T = \Diag(p^\alpha u_1,p^\beta u_2,p^\gamma u_3)$, where $u_1$, $u_2$, and $u_3 \in \Z_p^\times$ and $\alpha$, $\beta$, and $\gamma$ are nonnegative integers satisfying $\alpha \leq \beta \leq \gamma$. 
    A matrix 
    \[
        X := \begin{pmatrix}
           x_1 & x_2 & x_3 \\
            x_4 & x_5 & x_6 \\
            x_7 & x_8 & x_9 
        \end{pmatrix}
    \]
    is in $O(L)$ if $X^tTX = T$. 
    There are six distinct entries in $X^tTX - T$, given by
    \begin{align*}
        f_1 &:= p^\alpha u_1x_1^2 + p^\beta u_2x_4^2 + p^\gamma u_3x_7^2 - p^\alpha u_1, \\
        f_2 &:= p^\alpha u_1x_2^2 + p^\beta u_2x_5^2 + p^\gamma u_3x_8^2 - p^\beta u_2,  \\
        f_3 &:= p^\alpha u_1x_3^2 +
        p^\beta u_2x_6^2 + p^\gamma u_3x_9^2 - p^\gamma u_3, \\
        f_4 &:= p^\alpha u_1x_1x_2 + p^\beta u_2x_4x_5 + p^\gamma u_3x_7x_8, \\
        f_5 &:= p^\alpha u_1x_2x_3 +
        p^\beta u_2x_5x_6 + p^\gamma u_3x_8x_9, \\
        f_6 &:= p^\alpha u_1x_1x_3 +
        p^\beta u_2x_4x_6 + p^\gamma u_3x_7x_9.
    \end{align*}
    Let $\mathbf{f} = (f_1,\dots,f_6) \in \Z_p[x_1,\dots,x_9]^6$, let $N := \alpha + \beta + \gamma$, and suppose that $\mathbf{a} = (a_1,\dots,a_9) \in \Z_p^9$ satisfies $f_i(\mathbf{a}) \equiv 0 \bmod p^{2N + 1}$, which is to say that the matrix
    \[
        A := 
        \begin{pmatrix}
            a_1 & a_2 & a_3 \\
            a_4 & a_5 & a_6 \\
            a_7 & a_8 & a_9 
        \end{pmatrix} \in \text{Mat}_3(\Z_p)
    \]
    projects to $O(L/{p^{2N + 1}L})$ when reduced modulo $p^{2N + 1}$. 
    We apply Theorem~\ref{thm:Hensel} to lift the matrix $A$ to an element $B \in O(L)$. 
    Since the matrix $A$ is invertible over $\Z/p^{2N+1}\Z$, we know that $p \nmid \det(A)$, so there is a $2\times 2$ minor of $A$ that is not divisible by $p$.  
    Without loss of generality, assume that the principal $2\times 2$ minor $a_1a_5 - a_2a_4$ is not divisible by $p$.  
    A calculation shows that, among the $6 \times 6$ minors of the $6\times 9$ Jacobian $(D\mathbf{f}(\mathbf{a}))$, we find the following three quantities:
    \begin{align}
        \label{exp:a1}
        a_1 &\cdot (a_1a_5 - a_2a_4)\cdot\det(A)\cdot p^{3\alpha + 2\beta + \gamma} \cdot u_1^3u_2^2u_3, \\
        \label{exp:a4}
        a_4 &\cdot (a_1a_5 - a_2a_4)\cdot\det(A)\cdot p^{2\alpha + 3\beta + \gamma} \cdot u_1^2u_2^3u_3, \\
        \label{exp:a7}
        a_7 &\cdot (a_1a_5 - a_2a_4)\cdot\det(A)\cdot p^{2(\alpha + \beta + \gamma)} \cdot u_1^2u_2^2u_3^2.
    \end{align}
    Since $p\nmid \det(A)$, at least one of $a_1$, $a_4$, or $a_7$ must be a $p$-adic unit. 
    This means that at least one of the three expressions~\eqref{exp:a1}--\eqref{exp:a7} has its $p$-adic valuation supported entirely on the visible power of $p$ in the expressions. 
    By our assumption that $\alpha \leq \beta \leq \gamma$, we have 
    \begin{align*}
        3\alpha + 2\beta + \gamma &\leq 2(\alpha + \beta + \gamma) = 2N, \text{ and} \\ 
        2\alpha + 3\beta + \gamma &\leq 2(\alpha + \beta + \gamma) = 2N.
    \end{align*}
    Therefore, at least one of the maximal minors~\eqref{exp:a1}--\eqref{exp:a7} is not divisible by $p^{2N + 1}$. 
    Applying Theorem~\ref{thm:Hensel} with $e = 2N$ we deduce there is a $\mathbf{b} = (b_1,\dots,b_9) \in \Z_p^9$ such that $f(\mathbf{b}) = 0 \in \Z_p^9$ and $a_i \equiv b_i \bmod p^{2N + 1}$. 
    This lift $\mathbf{b}$ gives rise to a matrix $B \in O(L)$ that reduces to $A \bmod p^{2N + 1}$. 
\end{proof}

When $p = 2$, the proof of Theorem~\ref{thm:Hensel lifting all primes} follows the same strategy as in Proposition~\ref{prop_ReductionforOsurjective}, but there are a few more possible canonical forms for the Gram matrix of the lattice $L$.

\begin{proposition}\label{prop_ReductionforOsurjectivep=2}
    Let $L$ be a $\Z_2$-lattice of rank $3$ such that $2^{N+1} \nmid \disc(L)$. 
    Let $N' = \max\{3N + 4, 4N + 1\}$. 
    Then the reduction map
    \[
        O(L) \to O(L/2^{N'}L)
    \]
    is surjective.
\end{proposition}

\begin{proof}
    By~\cite{Cassels}*{Lemma~4.1}, up to $\Z_2$-equivalence we may assume the Gram matrix of the lattice $L$ is
    \begin{equation}
        \label{eq:Z2equiv}
        T_0 = 
        \begin{pmatrix}
            2^\alpha u_1 & 0 & 0 \\
            0 & 2^\beta u_2 & 0 \\
            0 & 0 & 2^\gamma u_3
        \end{pmatrix},
        \ 
        T_1 = 
        \begin{pmatrix}
            2^\alpha u & 0 & 0 \\
            0 & 0 & 2^\beta \\
            0 & 2^{\beta} & 0
        \end{pmatrix},
        \ \text{or}\ 
        T_2 =
        \begin{pmatrix}
            2^\alpha u & 0 & 0 \\
            0 & 2^{\beta+1} & 2^\beta \\
            0 & 2^{\beta} & 2^{\beta+1}
        \end{pmatrix},
    \end{equation}
    where $u$, $u_1$, $u_2$, $u_3 \in \{1,3,5,7\}$ and $\alpha$, $\beta$, and $\gamma$ are nonnegative integers. 
    We consider the cases of $T_1$ and $T_2$ first. 
    We still assume without loss of generality that $2 \nmid a_1a_5 - a_2a_4$, and find the following maximal minors in $(D\mathbf{f}(a))$ when the Gram matrix is $T_1$:
    \begin{align}
        \label{exp:a1p=2}
        a_1 &\cdot (a_1a_5 - a_2a_4)\cdot\det(A)\cdot 2^{3(\alpha + \beta + 1)} \cdot u^3, \\
        \label{exp:a4p=2}
        a_4 &\cdot (a_1a_5 - a_2a_4)\cdot\det(A)\cdot 2^{2\alpha + 4\beta + 3} \cdot u^2, \\
        \label{exp:a7p=2}
        a_7 &\cdot (a_1a_5 - a_2a_4)\cdot\det(A)\cdot 2^{2\alpha + 4\beta + 3} \cdot u^2.
    \end{align}
    Set $N = v_2(\det(T_1)) = \alpha + \beta + 1$. 
    Then 
    \begin{align*}
        3(\alpha + \beta + 1) &< 4(\alpha + \beta + \gamma) = 4N, \text{ and} \\ 
        2\alpha + 4\beta + 3 &\leq 4(\alpha + \beta + 1) = 4N,
    \end{align*}
    so at least one of the maximal minors~\eqref{exp:a1p=2}--\eqref{exp:a7p=2} is not divisible by $2^{4N + 1}$. 
    Applying Theorem~\ref{thm:Hensel} with $e = 4N$ we deduce there is a $\mathbf{b} = (b_1,\dots,b_9) \in \Z_2^9$ such that $f(\mathbf{b}) = 0 \in \Z_2^9$ and $a_i \equiv b_i \bmod 2^{4N + 1}$. 
    This lift $\mathbf{b}$ gives rise to a matrix $B \in O(L)$ that reduces to $A \bmod 2^{4N + 1}$. 

    When the Gram matrix is $T_2$, we assume without loss of generality that $p \nmid a_4a_8 - a_5a_7$, and find the following maximal minors in $(D\mathbf{f}(a))$:
    \begin{align}
        \label{exp:a1p=2case2}
        a_1 &\cdot (a_4a_8 - a_5a_7)\cdot\det(A)\cdot 2^{2\alpha + 4\beta + 3}\cdot 3^2 \cdot u^2, \\
        \label{exp:a4p=2case2}
        (a_4 + 2a_7) &\cdot (a_4a_8 - a_5a_7)\cdot\det(A)\cdot 2^{3\alpha + 4\beta + 3}\cdot 3^2 \cdot u, \\
        \label{exp:a7p=2case2}
        (2a_4 + a_7) &\cdot (a_4a_8 - a_5a_7)\cdot\det(A)\cdot 2^{3\alpha + 4\beta + 3}\cdot 3^2 \cdot u.
    \end{align}
    Since $\det(A)$ is invertible in $\Z_2$, at least one of $a_1$, $a_4 + 2a_7$, or $2a_4 + 2a_7$ must be a $2$-adic unit, so the $2$-adic valuation of these minors is supported on the visible power of $2$. 
    Set $N = v_2(\det(T_2)) = \alpha + 2\beta$ and observe that
    \[
        2\alpha + 4\beta + 3 \leq 3\alpha + 4\beta + 3 \leq 3N + 3
    \]
    Applying Theorem~\ref{thm:Hensel} with $e = 3N + 3$ we deduce there is a $\mathbf{b} = (b_1,\dots,b_9) \in \Z_2^9$ such that $f(\mathbf{b}) = 0 \in \Z_2^9$ and $a_i \equiv b_i \bmod 2^{3N + 4}$. 
    This lift $\mathbf{b}$ gives rise to a matrix $B \in O(L)$ that reduces to $A \bmod 2^{3N + 4}$. 

    Finally, when the Gram matrix is $T_0$, we can proceed exactly as in the proof of Proposition~\ref{prop_ReductionforOsurjective} to deduce that the projection map $O(M) \to O(M/2^{2N+1}M)$ is surjective. 
    We conclude by noting that $\max\{2N+1,3N+4,4N+1\} = \max\{3N+4,4N+1\}$.
\end{proof}

\begin{proof}[Proof of Theorem~\ref{thm:Hensel lifting all primes}]
    Combine Propositions~\ref{prop_ReductionforOsurjective} and~\ref{prop_ReductionforOsurjectivep=2}.
\end{proof}

\newpage

\subsection{Hensel lifting for the spin group}

\begin{theorem}
    \label{thm:Hensel lifting Spin Group}
    Let $p$ be a rational prime, and let $L$ be a $\Z_p$-lattice of rank $3$. 
    Then the reduction map
    \[
        \spin(L) \to \spin\left( L/p^{n}L\right)
    \]
    is surjective for all $n \geq 1$ if $p$ is odd and for all $n \geq 3$ if $p=2$.
\end{theorem}

\begin{proof}
    We claim that in rank $3$, the Spin group is described by an equation of the form $H(x_1,x_2,x_3,x_4) = 1$, where $H$ is a homogeneous polynomial of degree $2$ with $\Z_p$-coefficients. 
    To see this, let $e_1$, $e_2$, and $e_3$ be a $\Z_p$-basis for $L$. 
    Then the even Clifford algebra $C^0(L)$ is a $\Z_p$-module of rank~$4$ admitting the basis $1$, $i := e_2e_3$, $j := e_3e_1$, and $k:= e_1e_2$, and the Spin group consists of those $u\in C_0(T)$ such that $uu'=1$, where $u'$ is the image of $u$ under the reversal involution. 
    Letting
    \[
        u = x_1 + x_2\cdot i + x_3\cdot j + x_4 \cdot k
    \]
    and writing 
    \[\begin{pmatrix}
        2A & D & E \\
        D & 2B & F \\
        E & F & 2C
    \end{pmatrix}\]
    for a Gram matrix of $L$, we compute
    \begin{equation}
        \label{eq:spin equation}
        \begin{split}
            uu' &= x_1^2 + Fx_1x_2 + Ex_1x_3 + Dx_1x_4 + BCx_2^2 + (EF-CD)x_2x_3 \\
            &\quad + (DF-BE)x_2x_4 + ACx_3^2 + (ED - AF)x_3x_4 + ABx_4^2
        \end{split}
    \end{equation}
Setting the right hand side of~\eqref{eq:spin equation} to $H(x_1,x_2,x_3,x_4)$ verifies the claim.
\medskip

Euler's formula for homogeneous polynomials gives
\begin{equation}
    \label{eq:EulerSpin}
    2H=x_1\frac{\partial H}{\partial x_1} + x_2\frac{\partial H}{\partial x_2} + x_3\frac{\partial H}{\partial x_3} + x_4\frac{\partial H}{\partial x_4}.
\end{equation}
Let $u \in C^0(L)$ and suppose first that $p$ is odd. 
By~\eqref{eq:EulerSpin}, if $\displaystyle \frac{\partial H}{\partial x_i}(u)$ has $p$-adic valuation $\geq n$ for $i = 1,\dots,4$, then so does $H(u)$. 
When $p = 2$, we may only conclude that the $2$-adic valuation of $H(u)$ is $\geq n - 1$.

Suppose that $p$ is odd. 
Apply Theorem~\ref{thm:Hensel} with $(A,\frakm) = (\Z_p,p\Z_p)$, $n = 4$, $m =~1$, $e = 0$ and $f_1(x_1,x_2,x_3,x_4) := H(x_1,x_2,x_3,x_4) - 1$. 
If $\mathbf{a}\in \Z_p^4$ satisfies $f_1(\mathbf{a}) \equiv 0 \bmod p$, then for at least one $i$ it must occur that 
\[
    \frac{\partial f_1}{\partial x_i}(\mathbf{a}) = \frac{\partial H}{\partial x_i}(\mathbf{a}) \not\equiv 0 \bmod p,
\]
so there is a $\mathbf{b} \in \Z_p^4$ such that $f_1(\mathbf{b}) = 0$ and $(\mathbf{a}) \equiv \mathbf{b} \bmod p$, completing the proof when $p$ is odd. 
When $p = 2$, the only change needed is to take $e = 1$ and $\mathbf{a} \in \Z_p^4$ such that $f(\mathbf{a}) \equiv 0 \bmod 2^3$.
\end{proof}

\subsection{{Bounding the index of \texorpdfstring{$\spin(L) \to \tO(L)$}{} for integral \texorpdfstring{$p$}{p}-adic lattices}}

Let $R$ be a commutative ring with unit; the cases of interest to us are $R = \Z$, $\Z/p^n\Z$, $\Z_p$, or $\Q_p$. 
For a quadratic $R$-module $L$, denote by $C^0(L)$ its even Clifford algebra, with reversal involution $u\mapsto u'$. 
If $uu' = t \in R^\times$, we set $u^{-1} = t^{-1}u'$. 
The general spin group and the spin group $\spin(L)$ of $L$, are, respectively,
\begin{align*}
    \gspin(L) &=\{u\in C^0(L)\mid uu'\in R^{\times},\; uLu^{-1}\subseteq L\}, \\
    \spin(L) &= \{ u \in \gspin(L) \mid uu' = 1\}.
\end{align*}
Given $u \in \gspin(L)$, the $R$-linear map $T_u\colon L \to L$ taking $x$ to $uxu^{-1}$ is an isometry of $L$; its determinant is $1$ because $u$ lies in the even Clifford algebra of $L$. 
We obtain a map
\begin{align*}
    T\colon \gspin(L)&\to \SO(L),\\
    u&\mapsto \left( T_u\colon x\mapsto uxu^{-1}\right).
\end{align*}
We study the index of $T(\spin(L))$ in the full orthogonal group $\tO(L)$ when $R = \Z_p$ and $p \mid \disc(L)$.

\begin{theorem}\label{thm_IndexSpinToOinZp}
    Let $p$ be a prime and let $L$ be a $\Z_p$-lattice of rank $r$ with $v_p(\disc(L)) =: N$. 
    There is a function $B = B(r,N)$, depending linearly on $N$, such that the index of $T(\spin(L))$ in $\tO(L)$ is at most $4B+4$.
\end{theorem}

\noindent Extending scalar to $\Q_p$ gives a a \emph{surjective} map~\cite{Cassels}*{\S10.3, Theorem~3.1}
\begin{equation}
    \label{eq:TGSpin}
    \begin{split}
        T\colon \gspin(L\otimes\Q_p)&\to \SO(L\otimes\Q_p)\\
        u&\mapsto \left( T_u\colon x\mapsto uxu^{-1}\right),
    \end{split}
\end{equation}
which is key to the proof. 
For one, it allows us to define the Spinor norm $\SO(L\otimes \Q_p) \to \Q_p/\Q_p^{\times 2}$ by the assignment $T_u \mapsto uu'$. 
The kernel of the spinor norm is denoted $\SO^+(L\otimes \Q_p)$, and we set $\SO^+(L) := \SO(L) \cap \SO^+(L\otimes \Q_p)$; it is an index $2$ subgroup of $\SO(L)$.

\begin{lemma}\label{lemma_lifttoGspin}
    Let $L$ be a $\Z_p$-lattice of rank $r$ with $v_p(\disc(L)) =: N$. 
    There is a constant $B = B(r,N)$, depending linearly on $N$ such that under the surjective map~\eqref{eq:TGSpin} every element $\sigma\in\SO(L)$ equals $T_u$ for some $u\in C^0(L)$ with $v_p(uu')\leq B$. 

    \noindent If $\sigma\in \SO^+(L) \subset \SO(L)$, then we can further assume $uu'=p^{2m'}$ for some $2m'\leq B$.

    \noindent When $p$ is odd, we can take $B = 2rN$; when $p = 2$ and $r = 3$, we can take $B = 5N + 11$.
\end{lemma}

\begin{proof}
    We divide the proof of the lemma according to the parity of $p$. 
    \smallskip
    
    \noindent{\bf Case 1: $p$ odd.}

    \noindent We write $s_L$ for the supremum of all $p$-adic norms $|(a)^2|_p$ as $a$ runs over $L$; here $(a)^2$ denotes the value of the quadratic form of $L$ applied to $a$. 

    We may assume that $L$ has a basis $\{e_1,\dots,e_r\}$ with respect to which its Gram matrix is the diagonal matrix $\Diag(p^{i_1}u_1,\dots,p^{i_r}u_r)$ with $i_1\geq\dots\geq i_r\geq 0$ and $u_1,\dots,u_r\in\Z_p$.
    Following the proof of~\cite{Cassels}*{\S8.3,~Corollary~1}, there is an element $\sigma_1$ such that $\sigma_1^{-1}\sigma(e_r)=e_r$, and such that $\sigma_1$ is a product of at most two reflections $\tau_{b_j}$ with $|(b_j)^2|_p=s_L=p^{-i_r}\geq p^{-N}$.
    Hence, replacing $\sigma$ with $\sigma_1^{-1}\sigma$, we may assume that $\sigma$ preserves $e_r$ and therefore also the sublattice $L_1:=e_r^{\perp}$. 
    Note that $s_{L_1}=p^{-i_{r-1}}\geq p^{-N}$.
    Continuing inductively, we find that $\sigma$ is the product of $n\leq 2r$ many reflections $\tau_{b_j}$ with $|(b_j)^2|_p\geq p^{-N}$. 
    Note that $n$ must be even, since $\sigma\in\SO(L)$.
    So $\sigma=T_u$ with $u=b_1\dots b_{n}$, and $m:=v_p(uu')$ equals $v_p((b_1)^2\dots (b_{n})^2)\leq 2rN$.

    Assume that $\sigma$ lies in $\SO^+(L)$. 
    Then $uu'$ is a square and thus equals $p^{2m'}t^2$ for some $t\in\Z_p^{\times}$. 
    Let $v=t^{-1}u$, then $v\in\gspin(L)$ satisfies $T_v=\sigma$ and $vv'=p^{2m'}$ with $m'\leq rN$.

    \smallskip
    \noindent{\bf Case 2: $p = 2$.}

    \noindent We prove the lemma in the case $r = 3$.  
    This case contains all the difficulties of the general case; the industrious reader can easily leap from this particular case to the general statement.

    We may assume that the Gram matrix of $L$ is one of $T_0$, $T_1$, or $T_2$ as in~\eqref{eq:Z2equiv}. 
    In the case of $T_0$, we may assume that $0 \leq \alpha \leq \beta \leq \gamma$. 

    Let $b_1$ be the first basis vector of $L$ for $T_i$ ($i = 0$, $1$ or $2$), so $b_1^2 = 2^\alpha u$ for some $u \in \{1,3,5,7\}$. 
    Note that $2^\alpha \mid (b_1,v)$ for all $v \in L$. 
    Let $c_1 = \sigma(b_1)$, so $c_1^2 = 2^\alpha u$ because $\sigma$ is an isometry, and define $m$ by setting $(b_1,c_1) = 2^\alpha m$. 
    Then 
    \[
        (b_1 \pm c_1)^2 = 2\left(b_1^2 \pm (b_1,c_1)\right)  = 2\left(2^\alpha u \pm 2^\alpha m\right) = 2^{\alpha + 1}(u \pm m).
    \]
    Since 
    \[
        1 = v_2(2u) = v_2((u - m) + (u + m)) \geq \min\{v_2(u - m),v_2(u+m)\},
    \]
    either $v_2(u-m)\leq 1$ or $v_2(u+m)\leq 1$, so $v_2\left((b_1 - c_1)^2\right) \leq \alpha + 2$ or $v_2\left((b_1 + c_1)^2\right) \leq \alpha + 2$. 
    Hence, composing $\sigma$ with a product $\tau$ of at most two reflections in $O(L\otimes \Q_2)$, namely $\tau = \tau_{b_1 - c_1}$ or  $\tau = \tau_{b_1 + c_1}\circ\tau_{c_1}$, we obtain an element of $O(L\otimes \Q_2)$ mapping $c_1$ to $b_1$, and since $v_2(c_1^2) = \alpha$, the $2$-adic valuation of the vectors underlying the transformation $\tau$ is at most $\alpha + (\alpha + 2)\leq 2N + 2$. 
    Therefore, $\sigma' := \tau\circ\sigma \in O(L\otimes \Q_2)$ preserves $b_1$ and its orthogonal complement $b_1^\perp \otimes \Q_2$. 
    Note that, for all $v \in L$, we have $\sigma'(2v) \in L$. 
    Now restrict to $\sigma'\big|_{b_1^\perp \otimes \Q_2}$. 

    In the case of $T_0$, let $b_2$ be the second basis vector of $L$, and argue as above starting with $2b_2$: let $c_2 = \sigma'(2b_2)$. 
    Taking $\tau' = \tau_{2b_2 - c_2}$ or $\tau' = \tau_{2b_2 + c_2}\circ\tau_{c_2}$ we produce $\sigma'' := \tau'\circ\sigma' \in O(L\otimes\Q_2)$ fixing $b_2$ (and hence its orthogonal complement), and since
    \[
        v_2(c_2^2) = \beta + 2\quad\text{and either }
        v_2\left((2b_2 - c_2)^2\right) \leq \beta + 4
        \text{ or } v_2\left((2b_2 + c_2)^2\right) \leq \beta + 4,
    \]
    the $2$-adic valuation of the vectors underlying the transformation $\tau'$ is at most $(\beta + 2) + (\beta + 4)\leq 2N + 6$. 
    Finally, restrict $\sigma''$ to $\langle b_1, b_2\rangle^\perp\otimes \Q_2 = \langle b_3\rangle\otimes \Q_2$. 
    This is a rank one lattice, so $\sigma''\big|_{\langle b_3\rangle\otimes \Q_2} = \pm\id$. 
    If the restriction is ${-\id}$ then $\tau_{b_3}\circ\sigma''$ fixes $b_1$, $b_2$, and $b_3$, so $\tau_{b_3}\circ\sigma'' = \id \in O(L)$. 
    All told, we have written $\sigma$ as a composition of reflections in $O(L\otimes \Q_2)$ whose underlying vectors have sum of $2$-adic valuations no larger than $(2N+2)+(2N+6)+N = 5N + 8$. 

    In the case of $T_1$, let $e$ and $f$ denote the second and third basis vectors of $L$, and let $b_2 := e + f$. 
    Then $(2b_2)^2 = 2^{\beta + 3}$. 
    Let $c_2 = \sigma'(2b_2)$. 
    Taking $\tau' = \tau_{2b_2 - c_2}$ or $\tau' = \tau_{2b_2 + c_2}\circ\tau_{c_2}$ we produce $\sigma'' := \tau'\circ\sigma' \in O(L\otimes\Q_2)$ fixing $b_2$, and since
    \[
        v_2(c_2^2) = \beta + 3\quad\text{and either }
        v_2\left((2b_2 - c_2)^2\right) \leq \beta + 5
        \text{ or } v_2\left((2b_2 + c_2)^2\right) \leq \beta + 5,
    \]
    the $2$-adic valuation of the vectors underlying the transformation $\tau'$ is at most $(\beta + 3) + (\beta + 5)\leq 2N + 8$. 
    We have $\langle b_1, b_2\rangle^\perp \otimes \Q_2 = \langle e - f\rangle \otimes \Q_2$. 
    The restriction $\sigma''\big|_{\langle e - f\rangle \otimes \Q_2}$ is again $\pm \id$; if it is $-\id$ then $\tau_{e - f}\circ \sigma'' = \id \in O(L)$.  
    We have written $\sigma$ as a composition of reflections in $O(L\otimes \Q_2)$ whose underlying vectors have sum of $2$-adic valuations no larger than $(2N+2)+(2N+8)+N = 5N + 10$.
    
    In the case of $T_2$, let $b_2$ denote the second basis vector of $L$. 
    Then $(2b_2)^2 = 2^{\beta + 3}$. 
    Let $c_2 = \sigma'(2b_2)$, and take $\tau' = \tau_{2b_2 - c_2}$ or $\tau' = \tau_{2b_2 + c_2}\circ\tau_{c_2}$ to produce $\sigma'' := \tau'\circ\sigma' \in O(L\otimes\Q_2)$ fixing $b_2$. 
    Arguing as before, the $2$-adic valuation of the vectors underlying $\tau'$ is at most $(\beta + 3) + (\beta + 5)\leq 2N + 8$. 
    We have $\langle b_1, b_2\rangle^\perp \otimes \Q_2 = \langle e' - 2f'\rangle \otimes \Q_2$. 
    The restriction $\sigma''\big|_{\langle e' - 2f'\rangle \otimes \Q_2}$ is again $\pm \id$; if it is $-\id$ then $\tau_{e' - 2f'}\circ \sigma'' = \id \in O(L)$.  
    We have written $\sigma$ as a composition of reflections in $O(L\otimes \Q_2)$ whose underlying vectors have sum of $2$-adic valuations no larger than $(2N+2)+(2N+8)+(N+1) = 5N + 11$.

    Finally, we take
    \[
        B(3,N) = \max\{5N + 8, 5N + 10,5N + 11\} = 5N+11.
    \]
    As in the case where $p$ is odd, if $\sigma$ lies in $\SO^+(M)$, then $uu'$ is a square and thus equals $p^{2m'}t^2$ for some $t\in\Z_2^{\times}$. 
    Let $v=t^{-1}u$, then $v\in\gspin(M)$ satisfies $T_v=\sigma$ and $vv'=p^{2m'}$ with $2m'\leq B$.
\end{proof}

\begin{proof}[Proof of Theorem~\ref{thm_IndexSpinToOinZp}]
    Let $H\subset\gspin(L\otimes\Q_p)$ be the subgroup of integral elements
    \[
        H=\{u\in C^0(L)\mid uu'\in \Q_p^{\times},\;uLu^{-1}\subseteq L\}.
    \]
    The kernel of the homomorphism
    \[
        f\colon H\to \mathbb{Z},\; u\mapsto v_p(uu')
    \]
    is $\gspin(L)$, so the induced map $\bar f\colon H/\gspin(L)\to\Z$ is injective.
    By Lemma~\ref{lemma_lifttoGspin}, the group $\SO(L)$ is contained in the image of 
    \[
        S:=f^{-1}\left(\{0,1,\dots,B\}\right)\subset H\subset \gspin(L\otimes\Q_p)
    \] 
    under $T$. 
    Hence, the subgroup $\SO(L)/T(\gspin(L))$ is contained in the image of $S/\gspin(L)$ under the quotient map
    \[
        \overline{T}\colon \gspin(L\otimes\Q_p)/\gspin(L)\twoheadrightarrow\SO(L\otimes\Q_p)/T(\gspin(L)),
    \]
    The number of elements of this image is at most 
    \[
        \#S/\gspin(L)\leq\#\bar f^{-1}\left(\{0,1,\dots,B\}\right)\leq B+1,
    \]
    where the final inequality follows from the injectivity of $\bar f$.
    This shows that the index of $T(\gspin(L))$ in $\SO(L)$ is at most $B+1$.

    The image $T(\spin(L))$ equals $T(\gspin(L))\cap\SO^+(L)$: the containment $\subseteq$ follows from the definition of the Spinor norm (see~\cite{Cassels}*{\S X.3, Corollary~3}), and the containment $\supseteq$ follows from the second part of Lemma~\ref{lemma_lifttoGspin}.
    Hence, $\SO^+(L)/T(\spin(L))$ injects into $\SO(L)/T(\gspin(L))$ and thus the index of $T(\spin(L))$ in $\SO^+(L)$ is also at most $B+1$.
    Finally, since the index of $\SO^+(L)$ in $\tO(L)$ is at most 4, we conclude that $[\tO(L):T(\spin(L))]\leq 4B+4$.
\end{proof}

\subsection{Proof of Theorem~\ref{thm_IndexReductionO}}

By Propositions~\ref{prop_ReductionforOsurjective} and \ref{prop_ReductionforOsurjectivep=2}, the reduction map $\tO(L\otimes\mathbb{Z}_p)\to\tO(L/p^fL)$ is surjective.
The map $\spin(L\otimes\Z_p)\to \spin(L/p^fL)$ is also surjective, by Theorem~\ref{thm:Hensel lifting Spin Group}; note that when $p = 2$ we always have $N' \geq 4$, so we only need surjectivity for $f' \geq 4$ in this case, which Theorem~\ref{thm:Hensel lifting Spin Group} covers. 
The commutativity of the diagram
\[
    \xymatrix{
        \tO(L\otimes\Z_p)\ar@{->>}[r] & \tO(L/p^fL)\\
        \spin(L\otimes\Z_p)\ar[u]\ar@{->>}[r] & \spin(L/p^fL)\ar[u]
    }
\]
together with Theorem~\ref{thm_IndexSpinToOinZp} and the estimates for $B(3,N)$ from Lemma~\ref{lemma_lifttoGspin} imply that the index of the image of $\spin(L/p^fL)\to \tO(L/p^fL)$ is at most $4(2\cdot 3\cdot N)+4=24N+4$ when $p$ is odd, and at most $4(5N+11)+4=20N+48$ when $p=2$.
By strong approximation, the composition
\[
    \spin(L)\to \spin(L\otimes\Z_p) \to \spin(L/p^fL)
\]
is surjective. 
Finally, the commutativity of the diagram
\[
    \xymatrix{
        \tO(L)\ar[r] & \tO(L/p^fL)\\
        \spin(L)\ar[u]\ar@{->>}[r] & \spin(L/p^fL)\ar[u]
    }
\]
shows that the image of $\tO(L) \to \tO(L/p^fL)$ has index at most $24N+4$ or $20N+48$, respectively.
\qed

\section{Results on Algebraic Stacks}

In this appendix, we prove some results on algebraic stacks of a general nature. 
We fix a base scheme $S$.

\begin{lemma}
    \label{lem:quasi-sep lemma}
    Let $\cX$ be an algebraic stack over $S$ and let $\left\{\cU_i\right\}_{i\in I}$ be a collection of open substacks of $\cX$ which cover $\cX$.
    \begin{enumerate}[leftmargin=*]
        \item If each $\cU_i$ is separated over $S$ then the inertia $\cI_{\cX}\to\cX$ is proper.
        \smallskip
        
        \item If each $\cU_i$ is Noetherian (in the sense of \cite{stacks-project}*{0510}) then $\cX$ is locally Noetherian and quasi-separated.
    \end{enumerate}
\end{lemma}

\begin{proof}
    Let $\cU$ be the disjoint union of the $\cU_i$ and let $\pi:\cU\to\cX$ be the resulting cover. 
    Open immersions are representable, hence stabilizer-preserving. 
    Thus the diagram
    \[
        \begin{tikzcd}
            \mathcal{I}_{\cU}\arrow{r}\arrow{d}&\mathcal{I}_{\cX}\arrow{d}\\
            \cU\arrow{r}&\cX
        \end{tikzcd}
    \]
    is Cartesian. If each of the $\cU_i$ is separated, then $\cU$ is separated, and hence has proper inertia. 
    It follows that $\mathcal{I}_{\cX}\to\cX$ is proper, and the first claim is proven.
    
    Suppose that each of the $\cU_i$ is Noetherian. It is immediately clear that $\cX$ is then locally Noetherian. Consider the Cartesian diagram
    \[
        \begin{tikzcd}
            \cU\times_{\cX}\cU\arrow{r}\arrow{d}[swap]{\Delta'}&\cX\arrow{d}{\Delta_{\cX}}\\
            \cU\times_S\cU\arrow{r}{\pi\times\pi}&\cX\times_S\cX.
        \end{tikzcd}
    \]
    The morphism $\Delta'$ is the disjoint union of morphisms
    \[
        \delta_{i,j}:\cU_i\times_{\cX}\cU_j\to\cU_i\times_S\cU_j
    \]
    for $i,j\in I$. The intersection $\cU_i\times_{\cX}\cU_j$ is an open substack of $\cU_i$, and hence is quasi-compact by \cite{stacks-project}*{0CPM}. It follows that each of the morphisms $\delta_{i,j}$ is quasi-compact. Therefore $\Delta'$ and hence $\Delta_{\cX}$ is quasi-compact, so $\cX$ is quasi-separated.
\end{proof}

\begin{lemma}\label{lem:specializations for stacks}
    Let $\cX$ be a quasi-separated locally Noetherian algebraic stack. 
    If $x\rightsquigarrow y$ is a specialization of points of $|\cX|$, then there exists a DVR $R$ and a morphism $\spec R\to\cX$ such that the induced map on topological spaces $|\spec R|\to|\cX|$ sends the generic point to $x$ and the closed point to $y$.
\end{lemma}

\begin{proof}
    For schemes this follows from \cite{stacks-project}*{0CM2}.
    The result follows by combining this with \cite{stacks-project}*{0GVZ}. 
\end{proof}

We have used the following criterion for separatedness in the proof of Theorem \ref{thm:separated moduli stack}.

\begin{proposition}
    \label{prop:sep criterion}
    Let $\cX$ be a quasi-separated and Zariski-locally separated algebraic stack over $S$. 
    Then $\cX$ is separated if and only if the map $|\Delta_{\cX}|:|\cX|\to|\cX\times_S\cX|$ induced by the diagonal of $\cX$ is closed.
\end{proposition}

\begin{proof}
    As in the proof of Theorem \ref{thm:separated moduli stack}, we identify the underlying set of points of $|\cX\times_S\cX|$ with that of $|\cX|\times_{|S|}|\cX|$ via the natural continuous bijection $|\cX\times_S\cX|\to|\cX|\times_{|S|}|\cX|$. 
    With this identification, the map $|\Delta_{\cX}|$ is given by the usual diagonal map $x\mapsto (x,x)$. Suppose that $\cX$ is quasi-separated and Zariski-locally separated. 
    If $\cX$ is separated, then $\Delta_{\cX}$ is proper, and in particular closed. 
    Conversely, assume that $|\Delta_{\cX}|$ is closed. 
    By Lemma \ref{lem:quasi-sep lemma}, the inertia of $\cX$ is proper. 
    It follows that the neutral section $e:\cX\to\mathcal{I}_{\cX}$ is a closed immersion. 
    By \cite{stacks-project}*{04Z0} to show that $\cX$ is separated, it will suffice to show that its diagonal morphism is universally closed. 
    We verify the existence part of the valuative criterion.
    Consider a valuation ring $R$ with field of fractions $K$, a pair of morphisms $f,g:\spec R\to\cX$, and a 2-commutative diagram of solid arrows
    \begin{equation}
        \label{eq:lifting diagram}
        \begin{tikzcd}
            \spec K\arrow[hook]{d}\arrow{r}&\cX\arrow{d}{\Delta_{\cX}}\\
            \spec R\arrow{r}[swap]{(f,g)}\arrow[dashed]{ur}&\cX\times_S\cX.
        \end{tikzcd}
    \end{equation}
     We need to show that there exists a dashed arrow and suitable 2-isomorphisms rendering this diagram 2-commutative. 
     Consider the induced maps on topological spaces
     \[
        |f|,|g|:|\spec R|\to|\cX|.
     \]
     We claim that these maps are equal. 
     By the existence of the above diagram, they agree on the generic point of $\spec R$. 
     Let $x\in|\cX|$ be the image of the generic point and let $y\in|\cX|$ (resp.\ $z\in|\cX|$) be the image of the closed point under $|f|$ (resp.\ $|g|$). 
     We have a specialization $(x,x)\rightsquigarrow(y,z)$ in $|\cX\times_S\cX|$. 
     We assumed that $|\Delta_{\cX}|$ is closed, so $(y,z)$ is in the image of $|\Delta_{\cX}|$. 
     Therefore $y=z$, and so $|f|=|g|$. 
     Now, as $\cX$ is Zariski-locally separated, we can find a separated algebraic stack $\mathcal{U}$ and an open immersion $\iota:\mathcal{U}\hookrightarrow\cX$ whose image contains $y$. 
     Then the image of $\cU$ also contains $x$, so $|f|$ and $|g|$ both factor through $|\iota|:|\cU|\hookrightarrow|\cX|$. 
     Thus $f$ and $g$ both factor through the morphism $\iota:\cU\hookrightarrow\cX$. 
     As $\iota$ is a monomorphism, the diagram~\eqref{eq:lifting diagram} gives rise to a 2-commutative diagram of solid arrows
     \[
        \begin{tikzcd}
            \spec K\arrow[hook]{d}\arrow{r}&\cU\arrow{d}{\Delta_{\cU}}\\
            \spec R\arrow{r}[swap]{(f,g)}\arrow[dashed]{ur}&\cU\times_S\cU.
        \end{tikzcd}
     \]
    Because $\cU$ is separated, we may fill in this diagram with a dashed arrow and suitable 2-commutativities. 
    Composing with $\iota$ gives the desired lift.
\end{proof}


\end{document}